\documentclass[11pt,a4paper,oldfontcommands]{article}
\usepackage[utf8]{inputenc}
\usepackage[T1]{fontenc}
\usepackage{microtype}
\usepackage[dvips]{graphicx}
\usepackage{xcolor}
\usepackage{times}

\usepackage[english]{babel}
\usepackage{amsmath,mathtools}
\usepackage{amsfonts}
\usepackage{mathrsfs}
\usepackage{amsthm}
\usepackage{subfigure}
\usepackage{envmath}
\usepackage{amssymb}
\usepackage{dsfont}
\usepackage{bigints}
\usepackage{textgreek}
\usepackage{authblk}
\usepackage{titlesec}

\usepackage[
breaklinks=true,
colorlinks=false,
bookmarks=true,bookmarksopenlevel=2]{hyperref}

\usepackage{geometry}
\newtheorem{theorem}{Theorem}[section]
\newtheorem{prop}[theorem]{Proposition}
\newtheorem{lem}[theorem]{Lemma}
\newtheorem{cor}[theorem]{Corollary}

\theoremstyle{remark}
\newtheorem{rem}[theorem]{Remark}

\numberwithin{equation}{section}

\renewcommand{\Re}{\operatorname{Re}}

\renewcommand{\Im}{\operatorname{Im}}

\newcommand*\diff{\mathop{}\!\mathrm{d}}

\DeclareMathOperator*{\supess}{sup~ess\,}

\newcommand*{\myemail}[1]{%
    \normalsize\href{mailto:#1}{#1}\par
    }

\allowdisplaybreaks

\titleformat{\section}[block]{\centering \scshape \large}{\thesection.}{0.3\baselineskip}{}
\titlespacing{\section}{0pt}{*5}{*2}

\titleformat{\subsection}[block]{\bfseries}{\thesubsection.}{.5em}{}
\titlespacing{\subsection}{0pt}{*2.5}{*1}

\titleformat{\subsubsection}[runin]{\itshape}{\normalfont \thesubsubsection.}{.5em}{}[.]
\titlespacing{\subsubsection}{0pt}{*2.5}{0.5em}

\title{Convergence rate in Wasserstein distance and semiclassical limit for the defocusing logarithmic Schrödinger equation}
\date{\vspace{-1cm}}

\author[]{Guillaume Ferriere}

\affil[]{IMAG, Univ Montpellier, CNRS, Montpellier, France \\ \myemail{guillaume.ferriere@umontpellier.fr}}

\begin{document}

\maketitle

\begin{abstract}
    We consider the dispersive logarithmic Schrödinger equation in a semi-classical scaling. We extend the results of \cite{carlesgallagher} about the large time behaviour of the solution (dispersion faster than usual with an additional logarithmic factor, convergence of the rescaled modulus of the solution to a universal Gaussian profile) to the case with semiclassical constant. We also provide a sharp convergence rate to the Gaussian profile in Kantorovich-Rubinstein metric through a detailed analysis of the Fokker-Planck equation satisfied by this modulus. Moreover, we perform the semiclassical limit of this equation thanks to the Wigner Transform in order to get a (Wigner) measure. We show that those two features are compatible and the density of a Wigner Measure has the same large time behaviour as the modulus of the solution of the logarithmic Schrödinger equation. 
    Lastly, we discuss about the related kinetic equation (which is the \textit{Kinetic Isothermal Euler System}) and its formal properties, enlightened by the previous results and a new class of explicit solutions.
\end{abstract}

\bigskip

\section{Introduction}

\subsection{Setting}

We are interested in the \textit{Logarithmic Non-Linear Schrödinger Equation} with semiclassical constant
\begin{equation}
    i \varepsilon \, \partial_t u_\varepsilon + \frac{\varepsilon^2}{2} \Delta u_\varepsilon = \lambda u_\varepsilon \ln{\lvert u_\varepsilon \rvert^2}, \qquad \qquad
    {{u_\varepsilon}_|}_{t=0} = u_{\varepsilon,\textnormal{in}}, \label{log_nls_eps}
\end{equation}
with $x \in \mathbb{R}^d$, $d \geq 1$, $\lambda \in \mathbb{R} \setminus \{ 0 \}$, $\varepsilon > 0$. It was introduced as a model of nonlinear wave mechanics and in nonlinear optics (\cite{nonlin_wave_mec}, see also \cite{inco_white_light_log, log_nls_nuclear_physics, quantal_damped_motion, solitons_log_med, log_nls_magma_transp}). The case $\lambda < 0$ is interesting from a physical point of view and has been studied formally and rigorously without semiclassical constant (i.e. $\varepsilon = 1$, see \cite{D'av_Mont_Squa_lognls, quantal_damped_motion}). On the other hand, R. Carles and I. Gallagher recently went further in the case $\lambda > 0$ (also with $\varepsilon = 1$) whose study goes back to \cite{cazenave-haraux, Guerrero_Lopez_Nieto_H1_solv_lognls}. After improving the result of \cite{Guerrero_Lopez_Nieto_H1_solv_lognls} for the Cauchy problem, they proved not only that this case is actually the defocusing case with an unusually faster dispersion but also that a universal behaviour occurs: up to a rescaling, the modulus of the solution converges to a universal Gaussian profile (see \cite{carlesgallagher}).

In the context of (non-linear) Schrödinger equations, a usual question is the behaviour of the solution when $\varepsilon \rightarrow 0$ known as the \textit{semiclassical limit}, making the link between quantum mechanics and classical mechanics in physics. It has also been studied a lot in mathematics in order to get a good and rigorous framework for reaching the limit. Indeed, $u_\varepsilon$ typically does not have a meaningful limit and that is the reason why several asymptotic techniques have been developed to treat semiclassical (also called \textit{high-frequency}, or \textit{short-wavelength} in some contexts) problems. One of the most powerful and elegant tools was introduced by Wigner (\cite{wigner}) in 1932. Known nowadays as Wigner Transform, it has been analyzed a lot (\cite{PaulLions, athanassoulis-paul, gerard-mark-mauser, Gerard9091} for instance) and usually allows a simple and nice description of the semiclassical limit. For any sequence of functions $f_\varepsilon = f_\varepsilon (x) \in L^2 (\mathbb{R}^d)$ for $\varepsilon>0$, the Wigner Transform $W_\varepsilon$ defined by
\begin{equation*}
    W_\varepsilon (x, \xi) = \frac{1}{(2 \pi)^d} \int_{\mathbb{R}^d} e^{-i \xi \cdot z} f_\varepsilon \left(x + \frac{\varepsilon z}{2} \right) \, \overline{f_\varepsilon \left(x - \frac{\varepsilon z}{2} \right)} \diff z, \qquad (x,\xi) \in \mathbb{R}^d \times \mathbb{R}^d,
\end{equation*}
is a real-valued function on the phase space. It is known that under suitable assumptions, up to a subsequence, this function converges weakly to a measure, called Wigner measure; see e.g. \cite{PaulLions, gerard-mark-mauser}. Moreover, if $u_\varepsilon$ satisfies $i \varepsilon \, \partial_t u_\varepsilon + \frac{\varepsilon^2}{2} \Delta u_\varepsilon = V_0 u_\varepsilon$ 
and if the potential $V_0$ is smooth enough then 
the Wigner Measure $W(t)$ of $(u_\varepsilon (t))_{\varepsilon>0}$ satisfies the Vlasov (or kinetic) equation
\begin{equation*}
    \partial_t W + \xi \cdot \nabla_x W + \nabla_x V_0 \cdot \nabla_\xi W = 0.
\end{equation*}

As a follow-up of \cite{carlesgallagher}, this article has two main purposes: reaching the semiclassical limit thanks to the Wigner Transform and computing the convergence rate to the Gaussian profile in Wasserstein distance. Actually, those two features are compatible since the convergence rate is actually independent of $\varepsilon \in (0,1]$ and then goes through the limit $\varepsilon \rightarrow 0$ under suitable assumptions. This is a very interesting and rare feature: it has been shown that the large time behaviour and the semiclassical limit do not usually commute, for instance for linear Schrödinger equations with potential (see \cite{Combescure_Robert__Coherent_states, Hagedorn_semiclassical_III, Hagedorn_Joye, yajima1979, Yajima1981}). Moreover, in general, Wigner measures are not a suitable tool to address nonlinear problems, except in the case of the Schrödinger-Poisson equation (see \cite{Zhang-Zheng-Mauser, Wigner_fails}). On the other hand, at least in the case $\varepsilon=1$, \eqref{log_nls_eps} exhibits rather strong nonlinear effects (modified dispersion, universal asymptotic profile), so it is rather surprising that such a result can be established.

\subsection{Universal dynamics without semiclassical constant}

Throughout the rest of this paper, we assume $\lambda > 0$. We recall the \textit{Logarithmic Non-Linear Schrödinger Equation} without semiclassical constant ($\varepsilon = 1$)
\begin{equation}
    i \, \partial_t u + \frac{\Delta u}{2} = \lambda u \ln{ \lvert u \rvert^2}, \qquad
    u (0,.) = u_{\textnormal{in}}. \label{log_nls}
\end{equation}
Following the notations used in \cite{carlesgallagher}, for $0 < \alpha \leq 1$, we define
\begin{equation*}
    \mathcal{F} \left( H^\alpha \right) := \{ u \in L^2 (\mathbb{R}^d), x \mapsto \langle x \rangle^\alpha u(x) \in L^2 (\mathbb{R}^d) \},
\end{equation*}
where $\langle x \rangle = \sqrt{1 + \lvert x \rvert^2}$ and $\mathcal{F}$ is the Fourier Transform. $\mathcal{F} \left( H^\alpha \right)$ is endowed with its natural norm. In the same way, we also define the mass, the angular momentum and the energy (with semiclassical constant) for all $f \in \{ g \in H^1 (\mathbb{R}^d), \lvert g \rvert^2 \ln \lvert g \rvert^2 \in L^1\}$:
\begin{gather*}
    M(f) \coloneqq \lVert f \rVert_{L^2}, \qquad \qquad
    \mathcal{J}_\varepsilon (f) \coloneqq \varepsilon \Im \int_{\mathbb{R}^d} \overline{f (x)} \, \nabla f (x) \diff x, \\
    E_\varepsilon (f) \coloneqq \frac{\varepsilon^2}{2} \lVert \nabla f \rVert_{L^2} + \lambda \int_{\mathbb{R}^d} \lvert f(x) \rvert^2 \ln \lvert f(x) \rvert^2 \diff x.
\end{gather*}
The Cauchy problem is investigated in \cite{Guerrero_Lopez_Nieto_H1_solv_lognls} and improved by \cite[Theorem~1.5.]{carlesgallagher}, showing well-posedness for initial data in $H^1 \cap \mathcal{F}(H^\alpha) (\mathbb{R}^d)$ with $0<\alpha\leq1$ and conservation of those three quantities (with $\varepsilon = 1$). Then, in the same paper, the authors studied large time behaviour of the solution when $\alpha = 1$.
%
Two features characterizing the dynamics associated to \eqref{log_nls} are unusual:
\begin{itemize}
    \item The dispersion is in $\left( t \sqrt{\ln t} \right)^\frac{d}{2}$. Usually in $t^\frac{d}{2}$ for the Schrödinger equation, it is altered by a logarithmic factor due to the non-linearity of the equation, accelerating the dispersion.
    \item Up to a rescaling, the modulus of the solution converges for large time to a universal Gaussian profile weakly in $L^1$.
\end{itemize}
Those aspects are stated in the main theorem of \cite{carlesgallagher}, recalled in Theorem \ref{main_th_carlesgallagher}. Denote by $\gamma (x) := e^{-\frac{\lvert x \rvert^2}{2}}$ for $x \in \mathbb{R}^d$ and $\tau \in \mathcal{C}^\infty (\mathbb{R})$ the solution to 
\begin{equation}
    \Ddot{\tau} = \frac{2 \lambda}{\tau}, \qquad
    \tau(0) = 1, \qquad \Dot{\tau} (0) = 0.
    \label{def_tau}
\end{equation}
This function satisfies $\tau (t) \sim 2t \sqrt{\lambda \ln t}$ and $\dot{\tau} (t) \sim 2 \sqrt{\lambda \ln t}$ as $t \rightarrow + \infty$ (see \cite[Lemma~1.6.]{carlesgallagher}). We also define the following quantities (with semiclassical constant) for any function $f \in \mathcal{F} (H^1) \cap H^1 (\mathbb{R}^d)$:
\begin{gather}
    \tilde{E}_\varepsilon \big(t, f\big) := \int_{\mathbb{R}^d} \left( \lvert y \rvert^2 + \left\lvert \ln \lvert f \rvert^2 \right\rvert \right) \, \lvert f \rvert^2 \diff y + \frac{\varepsilon^2}{\tau (t)^2} \lVert \nabla_y f \rVert_{L^2(\mathbb{R}^d)}^2, 
    \qquad
    \tilde{E}_\varepsilon^0 \big(f\big) = \tilde{E}_\varepsilon \big(0, f\big). \label{def_tilde_E_eps_0}
\end{gather}

\begin{theorem}[{\cite[Theorem~1.7.]{carlesgallagher}}]
\label{main_th_carlesgallagher}
Let $\lambda > 0$ and $u_{\textnormal{in}} \in \mathcal{F} (H^1) \cap H^1 (\mathbb{R}^d) \setminus \{ 0 \}$. Rescale the solution provided by \cite[Theorem~1.5.]{carlesgallagher} to $v = v(t,y)$ by setting
\begin{equation*}
    u (t,x) = \frac{1}{\tau(t)^\frac{d}{2}} \, \frac{\lVert u_{\textnormal{in}}\rVert_{L^2}}{\lVert\gamma\rVert_{L^2}} \, v \left(t, \frac{x}{\tau (t)} \right) \, e^{i \, \frac{\Dot{\tau}(t)}{\tau(t)} \frac{|x|^2}{2}}.
\end{equation*}
Then there exists $C$ such that for all $t \geq 0$,
\begin{equation*}
    \tilde{E}_1 \big(t, v(t) \big) \leq C.
\end{equation*}
We have moreover
\begin{equation*}
    \int_{\mathbb{R}^d} |y|^2 |v (t, y)|^2 \diff y \; \underset{t \rightarrow \infty}{\longrightarrow} \; \int_{\mathbb{R}^d} |y|^2 \gamma^2 (y) \diff y.
\end{equation*}
Finally,
\begin{equation*}
    |v(t, .)|^2 \underset{t \rightarrow \infty}{\rightharpoonup} \gamma^2 \qquad \text{weakly in } L^1(\mathbb{R}^d).
\end{equation*}
\end{theorem}

\begin{rem}
As a straightforward consequence, with the notations of the previous theorem, $|v (t)|^2$ converges to $\gamma^2$ in Wasserstein distance:
\begin{equation*}
    \mathcal{W}_2 \left( \pi^{-\frac{d}{2}} \, | v (t) |^2, \pi^{-\frac{d}{2}} \, \gamma^2 \right) \underset{t \rightarrow \infty}{\longrightarrow} 0,
\end{equation*}
where we recall that the Wasserstein distance is defined for $\nu_1$ and $\nu_2$ probability measures by
\begin{equation*}
    \mathcal{W}_p (\nu_1, \nu_2) = \inf \{ \left( \int_{\mathbb{R}^d \times \mathbb{R}^d} |x - y|^p \diff \mu(x,y) \right)^\frac{1}{p} ; \quad (\pi_j)_\# \mu = \nu_j \},
\end{equation*}
where $\mu$ varies among all probability measures on $\mathbb{R}^d \times \mathbb{R}^d$, and $\pi_j$ : $\mathbb{R}^d \times \mathbb{R}^d \rightarrow \mathbb{R}^d $ denotes the canonical projection onto the j-th factor (see e.g. \cite{villani2008optimal}).
\end{rem}

\begin{rem}
Another consequence, stated in \cite[Corollary~1.10.]{carlesgallagher}, is the logarithmic growth of the Sobolev norms of the solution as soon as the initial data is not null:
    \begin{gather*}
        \lVert \nabla u (t) \rVert_{L^2}^2 \underset{t \rightarrow + \infty}{\sim} 2 \lambda d \lVert u_\textnormal{in}\rVert_{L^2} \, \ln t, \qquad \text{and} \qquad
        (\ln t)^\frac{\delta}{2} \lesssim \lVert u (t) \rVert_{\dot{H}^\delta} \lesssim (\ln t)^\frac{\delta}{2}, \qquad \forall t > 1, \forall \delta \in (0,1),
    \end{gather*}
    where $\dot{H}^\delta (\mathbb{R}^d)$ denotes the standard homogeneous Sobolev space.
\end{rem}

The weak convergence in $L^1$ found in \cite[Theorem~1.7.]{carlesgallagher} actually comes from the fact that, after a change of time variable, $\rho (t) = |v (t,.) |^2$ satisfies a Fokker-Planck equation with some source terms which are negligible (in some way) when $t \rightarrow \infty$, along with the compactness of $\{ \rho (t), t \in \mathbb{R} \}$ in $L^1_w$. To provide this weak convergence, the authors first take the limit $t\rightarrow+ \infty$ (up to a subsequence) and then use the properties of this Fokker-Planck equation along with the fact that the limit satisfies the same Fokker-Planck equation without source term to conclude that the limit is a universal Gaussian profile (and that the whole sequence converges).

However, the Fokker-Planck operator $L = \Delta + \nabla \cdot (2y \, .)$ is extremely particular. Indeed, unlike most of the other Fokker-Planck operators, its form allows to compute explicitly its kernel, which leads to better estimates for the solution. Those estimates are helpful in order to compute some convergence rate. For this, we have to consider a distance which metrizes the weak convergence in $L^1$ (no strong convergence has been proved). Since there is also convergence of the first two momenta, we focus on the Wasserstein metric, and mostly on the 1-Wasserstein distance (also called Kantorovich-Rubinstein metric) because the Kantorovich-Rubinstein duality gives an easier framework.

\subsection{Main results}

\subsubsection{Universal dynamics with semiclassical constant}

Introducing the semiclassical constant in the equation, we now want to investigate \eqref{log_nls_eps}. First of all, we need to face the Cauchy problem, which is easy to state thanks to \cite[Theorem~1.5.]{carlesgallagher}.

\begin{theorem}
\label{th_cauchy_log_nls_eps}
Given any $\varepsilon>0$, $\lambda>0$ and any initial data $u_{\varepsilon,\textnormal{in}} \in H^1 \cap \mathcal{F}(H^\alpha) (\mathbb{R}^d)$ with $0<\alpha\leq1$, there exists a unique, global solution $u_\varepsilon \in L^\infty_\text{loc} \left( \mathbb{R}, H^1 \cap \mathcal{F}(H^\alpha) (\mathbb{R}^d) \right) \cap \mathcal{C} \left(\mathbb{R}, H^{-1} \cap L^2_w (\mathbb{R}^d) \right)$ to \eqref{log_nls_eps}. Moreover, the mass $M(u_\varepsilon(t))$, the angular momentum $\mathcal{J}_\varepsilon(u_\varepsilon(t))$ and the energy $E_\varepsilon(u_\varepsilon(t))$ are independent of time.
\end{theorem}

In the same way as in \cite{carlesgallagher}, the first main focus of this paper concerns large time behaviour of this solution. The results of Theorem \ref{main_th_carlesgallagher} can be extended as well, and the same features as without semiclassical constant hold (faster dispersion with a logarithmic factor, convergence to a universal Gaussian profile after rescaling). But the main new feature of this result is the convergence rate to the Gaussian profile.

\begin{theorem}
\label{main_th_log_nls_eps}
Let $\lambda > 0$, $\varepsilon > 0$ and $u_{\varepsilon,\textnormal{in}} \in H^1 \cap \mathcal{F}(H^1) (\mathbb{R}^d) \setminus \{ 0 \}$. 
Rescale the solution $u_\varepsilon$ provided by Theorem \ref{th_cauchy_log_nls_eps} to $v_\varepsilon = v_\varepsilon(t,y)$ by setting
\begin{equation}
    u_\varepsilon (t,x) = \frac{1}{\tau(t)^\frac{d}{2}} \, \frac{\lVert u_{\varepsilon,\textnormal{in}}\rVert_{L^2}}{\lVert\gamma\rVert_{L^2}}
    \, v_\varepsilon \left(t, \frac{x}{\tau (t)} \right) \, e^{i \, \frac{\Dot{\tau}(t)}{\tau(t)} \frac{|x|^2}{2 \, \varepsilon}}, \qquad v_{\varepsilon, in} := v_{\varepsilon} (0) = \frac{\lVert\gamma\rVert_{L^2}}{\lVert u_{\varepsilon,\textnormal{in}}\rVert_{L^2}} u_{\varepsilon,\textnormal{in}}. \label{rescaling}
\end{equation}
There exists a non-decreasing continuous function 
$C: [0, \infty) \rightarrow [0, \infty)$ depending only on $\lambda$ and $d$ such that for all $t \geq 0$ and all $\varepsilon > 0$,
\begin{gather}
    \tilde{E}_\varepsilon \big(t,v(t)\big)
    \leq C \hspace{-0.5mm} \left(\tilde{E}_\varepsilon^0 (v_{\varepsilon,in}) 
    \right), \label{enestschr_th} \\
    \int_0^\infty \frac{\varepsilon^2 \, \Dot{\tau} (t)}{\tau^3 (t)} \lVert \nabla_y v_\varepsilon(t)\rVert_{L^2(\mathbb{R}^d)}^2 \diff t \leq C \hspace{-0.5mm} \left(\tilde{E}_\varepsilon^0 (v_{\varepsilon,in}) 
    \right). \label{enresschr_th}
\end{gather}
Moreover, the first two momenta converge: for all $t\geq1$ and all $\varepsilon>0$,
\begin{gather}
    \int_{\mathbb{R}^d} y \, |v_\varepsilon (t,y)|^2 \diff y = \frac{1}{\tau (t)} \frac{\lVert\gamma^2\rVert_{L^1}}{\lVert u_{\varepsilon,\textnormal{in}}\rVert_{L^2}^2} (I_{1,0}^\varepsilon \, t + I_{2,0}^\varepsilon) \underset{t \rightarrow \infty}{\longrightarrow} 0, \label{momestschr_1_th} \\
    \left | \int_{\mathbb{R}^d} |y|^2 \, |v_\varepsilon (t,y)|^2 \diff y - \int_{\mathbb{R}^d} |y|^2 \, \gamma^2 (y) \diff y \right\rvert \leq C \hspace{-0.5mm} \left(\tilde{E}_\varepsilon^0 (v_{\varepsilon,in}) 
    \right) \frac{\Dot{\tau}(t) + 1}{\Dot{\tau} (t)^2} \underset{t \rightarrow \infty}{\longrightarrow} 0, \label{momestschr_2_th} 
\end{gather}
where
\begin{equation*}
    I_{1,0}^\varepsilon = \Im \varepsilon \int_{\mathbb{R}^d} \overline{u_{\varepsilon,\textnormal{in}}} \, \nabla u_{\varepsilon,\textnormal{in}} \diff y, \qquad I_{2,0}^\varepsilon = \int_{\mathbb{R}^d} y \, |u_{\varepsilon,\textnormal{in}}|^2 \diff y.
\end{equation*}
Lastly, for all $t\geq2$ and all $\varepsilon>0$,
\begin{equation*}
    \mathcal{W}_1 \left( \frac{|v_\varepsilon (t, .)|^2}{\pi^{\frac{d}{2}}},
    \frac{\gamma^2}{\pi^{\frac{d}{2}}} \right) \leq C \hspace{-0.5mm} \left(\tilde{E}_\varepsilon^0 (v_{\varepsilon,in}) 
    \right) \frac{1}{\sqrt{\ln t}}. \label{conv_wass_dist}
\end{equation*}
\end{theorem}

It is worth noticing that the bounds and convergence rates for $v_\varepsilon$ depend on $\varepsilon$ only through $\tilde{E}_\varepsilon^0 (v_{\varepsilon,in})$. In particular, if we take suitable $u_{\varepsilon,\textnormal{in}}$ such that $\tilde{E}_\varepsilon^0 (v_{\varepsilon,in})$ is bounded, then all of them may be taken \textit{independent of} $\varepsilon \in (0,1]$. This is a very important feature, which rarely happens for large time behaviour in the context of the semiclassical limit.

If we also want uniform bounds and convergence rates for $u_\varepsilon$ thanks to \eqref{rescaling}, $\lVert u_{\varepsilon,\textnormal{in}} \rVert_{L^2}$ must be bounded. Thus, we introduce Assumption \eqref{in_data_assumption}:
\begin{equation}
    (u_{\varepsilon,\textnormal{in}})_{\varepsilon>0} \text{ uniformly bounded in } L^2 (\mathbb{R}^d) \qquad \text{ and} \qquad \left( \tilde{E}_\varepsilon^0 (v_{\varepsilon,in}) \right)_{\varepsilon\in(0,1]} \text{ is bounded}, \tag{A1} \label{in_data_assumption}
\end{equation}
where $v_{\varepsilon,in}$ is defined by \eqref{rescaling}. If those assumptions are satisfied, then all the bounds and convergence rates are \textit{uniform} in $\varepsilon \in (0,1]$. Such a thing occurs for instance for WKB states:
\begin{equation}
\begin{gathered}
    u_{\varepsilon,\textnormal{in}} = \sqrt{\rho_{\textnormal{in}}} \, e^{i \, \frac{\phi_{\textnormal{in}}}{\varepsilon}}, \qquad \forall \varepsilon\in(0,1], \qquad \text{ for some } \rho_{\textnormal{in}} = \rho_{\textnormal{in}} (x) \geq 0 \text{ and } \phi_{\textnormal{in}} = \phi_{\textnormal{in}} (x) \text{ such that:} \\
    \sqrt{\rho_{\textnormal{in}}} \in \mathcal{F}(H^1) \cap H^1 (\mathbb{R}^d) \setminus \{ 0 \}, \qquad
    \phi_{\textnormal{in}} \in W^{1,1}_\text{loc} (\mathbb{R}^d), \qquad
    \sqrt{\rho_{\textnormal{in}}} \, \nabla \phi_{\textnormal{in}} \in L^2 (\mathbb{R}^d),
\end{gathered} \label{WKB_state_assumption} \tag{A2}
\end{equation}
because those assumptions imply $\rho \ln \rho \in L^1 (\mathbb{R}^d)$. Indeed, the two estimates
\begin{equation*}
    \int \rho_{\textnormal{in}}^{1 + \delta } \leq C_\delta \lVert \sqrt{\rho_{\textnormal{in}}} \rVert_{H^1}^{1+\delta}
\end{equation*}
for $\delta > 0$ small enough thanks to Sobolev embeddings and
\begin{equation}
    \int \rho_{\textnormal{in}}^{1 - \delta } \leq C_\delta \lVert \sqrt{\rho_{\textnormal{in}}} \rVert_{L^2}^{2 - 2 \delta - d \delta} \lVert \, |x| \sqrt{\rho_{\textnormal{in}}} \, \rVert^{d \delta}_{L^2} \label{est_subcritical_power_to_mom}
\end{equation}
for $0 < \delta < \frac{2}{d + 2}$ which can be readily proved by an interpolation method (cutting
the integral into $|y| < R$ and $|y| > R$, using Hölder inequality and optimizing
over $R$; see e.g. \cite{carles-miller}) yield $\int \rho_{\textnormal{in}} \, | \ln{\rho_{\textnormal{in}}} \,| \, dy < \infty$.
Moreover in such a case, $\lVert u_{\varepsilon,\textnormal{in}} \rVert_{L^2}$, $I_{1,0}^\varepsilon$ and $I_{2,0}^\varepsilon$ are independent of $\varepsilon$.

The assumptions \eqref{WKB_state_assumption} are well known as \textit{WKB} states and the corresponding Wigner Measure (without time-dependence) is a \textit{monokinetic} measure (see \cite[Exemple~III.5.]{PaulLions}). Under stronger assumptions on $\rho_{\textnormal{in}}$ and $\phi_{\textnormal{in}}$, this feature usually propagates in time for some (non-linear) Schrödinger equations and we recover time-dependent monokinetic measure (see for instance \cite{Carles_book}). However, it might be difficult to prove it for \eqref{log_nls_eps}, except in a particular case (see Section \ref{sec_explicit_sol_KIE}).

\begin{rem}
The rescaling \eqref{rescaling} is similar to that in Theorem \ref{main_th_carlesgallagher} when adding the semiclassical constant: the main complex oscillations are altered by an $\varepsilon^{-1}$ factor.
\end{rem}

\begin{rem}
The convergence in Wasserstein distance is not new, we already know that we had convergence even with respect to $\mathcal{W}_2$ (at least for $\varepsilon = 1$). Yet, the convergence rate is an interesting new feature: no convergence rate (except for the momenta) was proven in \cite{carlesgallagher}. Moreover, such a convergence rate is optimal in this way: if $I_{1,0}^\varepsilon \neq 0$ (which is often verified, unless the initial data are well prepared), the convergence rate of the first moment reads:
\begin{gather*}
    \int_{\mathbb{R}^d} y \, |v_\varepsilon (t,y)|^2 \diff y \underset{t \rightarrow \infty}{\sim} \frac{\lVert\gamma^2\rVert_{L^1}}{\lVert u_{\varepsilon,\textnormal{in}}\rVert_{L^2}^2} \frac{I_{1,0}^\varepsilon}{2 \sqrt{\lambda \ln t}}.
\end{gather*}
Therefore we cannot have a better convergence rate, at least in the general case.
\end{rem}

Thanks to the bounds on the $L^1$ norm of $| v_\varepsilon (t,.) |^2$ and on its second momentum, the following corollary also holds:
\begin{cor}
\label{cor_other_conv_log_nls}
With the notations of Theorem \ref{main_th_log_nls_eps}, for all $t \geq 2$, all $\varepsilon > 0$ and all $\delta \in (0,1)$,
\begin{gather*}
    \left\lVert |v_\varepsilon (t, .)|^2 - \gamma^2 \right\rVert_{W^{-1+\delta, 1}} \leq C \hspace{-0.5mm} \left(\tilde{E}_\varepsilon^0 (v_{\varepsilon,in})\right) \, \frac{1}{\left( \ln t \right)^\frac{1-\delta}{2}}, \\
    \mathcal{W}_{1+\delta} \left( \frac{|v_\varepsilon (t, .)|^2}{\pi^\frac{d}{2}}, \frac{\gamma^2}{\pi^\frac{d}{2}} \right) \leq C \hspace{-0.5mm} \left(\tilde{E}_\varepsilon^0 (v_{\varepsilon,in})\right) \, \frac{1}{\left( \ln t \right)^\frac{1-\delta}{2}}.
\end{gather*}
\end{cor}

Finally, the Sobolev norms of all solutions grow in the same way as in \cite{carlesgallagher}, with in addition the semiclassical constant.

\begin{cor} \label{cor_sob_norm_growth}
    Given any $\varepsilon > 0$ and $\lambda > 0$, let $u_{\varepsilon,\textnormal{in}} \in H^1 \cap \mathcal{F}(H^1) (\mathbb{R}^d) \setminus \{ 0 \}$. The solution $u_\varepsilon$ to \eqref{log_nls_eps} satisfies as $t \rightarrow \infty$,
    \begin{equation*}
        \varepsilon^2 \, \lVert \nabla u_\varepsilon \rVert_{L^2}^2 \sim 2 \lambda d \, \lVert u_{\varepsilon,\textnormal{in}} \rVert_{L^2}^2 \ln t,
    \end{equation*}
    and for all $\delta \in (0,1)$,
    \begin{equation*}
        (\ln t)^\frac{\delta}{2} \lesssim \varepsilon^\delta \, \lVert u_\varepsilon \rVert_{\dot{H}^\delta} \lesssim (\ln t)^\frac{\delta}{2},
    \end{equation*}
    where $\dot{H}^\delta$ denotes the standard homogeneous Sobolev space.
\end{cor}

\subsubsection{The semiclassical limit}

Following the previous remarks, we now want to study the semiclassical limit of \eqref{log_nls_eps}, and the Wigner Transform is a natural tool we may use along with the usual space of test functions:
\begin{equation*}
    \mathcal{A} = \{ \phi \in \mathcal{C}_0 (\mathbb{R}^d_x \times \mathbb{R}^d_\xi ), (\mathcal{F}_\xi \phi)(x,z) \in L^1 (\mathbb{R}^d_z, \mathcal{C}_0(\mathbb{R}^d_x)) \}
\end{equation*}
endowed with its natural norm which makes it a Banach space and algebra. 
In a lot of cases, the Wigner Transform converges pointwise in time to the Wigner Measure, which is continuous in time with values in $\mathcal{A}'$ and satisfies the linked kinetic (or Vlasov-type) equation, with the same potential (see for instance \cite{PaulLions,athanassoulis-paul,gerard-mark-mauser}). 
If a lot of potentials satisfy the assumptions of one of those results, this is not the case here (to the best of our knowledge). Indeed, our potential depends on the solution and is highly singular in the same time.

However, 
the framework given by Theorems \ref{th_cauchy_log_nls_eps} and \ref{main_th_log_nls_eps} for our equation is still interesting for the Wigner Transform, considering that the solutions $(u_\varepsilon (t))_{\varepsilon>0}$ at time $t \in \mathbb{R}$ satisfy the usual assumptions in order to reach the limit for the Wigner Transform and get good properties for the Wigner Measure (see for example \cite[Proposition~III.1.~and~Théorème~III.1.]{PaulLions}). Therefore, an interesting framework would be to work in $L^p$ locally in time for all $p<\infty$, leave to lose the pointwise convergence of the Wigner Transform and the continuity of the Wigner Measure in time.

Before stating the theorem for the semiclassical limit, we denote by $\mathcal{M} (\mathbb{R}^d)$ the set of non-negative finite measures on $\mathbb{R}^d$, $\mathcal{P} (\mathbb{R}^d)$ the set of all probability measures and we also define
\begin{align*}
    L^1_2 (\mathbb{R}^d) &:= \{ f \in L^1 (\mathbb{R}^d), \int_{\mathbb{R}^d} |y|^2 \, |f(y)| \diff y < \infty \}, \\
    L \log L (\mathbb{R}^d) &:= \{ f \in L^1 (\mathbb{R}^d),  |f| \log |f| \in L^1 (\mathbb{R}^d) \}, \\
    \mathcal{P}_i (\mathbb{R}^d) &:= \{ \mu \in \mathcal{P} (\mathbb{R}^d), \int |x|^i \diff \mu < \infty \} \qquad \text{endowed with } \mathcal{W}_i \text{ for } i = 1,2.
\end{align*}

\begin{theorem}
\label{main_th_wigner}
Given any $\lambda > 0$ and $u_{\varepsilon,\textnormal{in}} \in H^1 \cap \mathcal{F}(H^1) (\mathbb{R}^d) \setminus \{ 0 \}$ for all $\varepsilon > 0$ such that $(u_{\varepsilon,\textnormal{in}})_{\varepsilon>0}$ satisfies \eqref{in_data_assumption}, define $u_\varepsilon$ and $v_\varepsilon$ provided by Theorem \ref{main_th_log_nls_eps} for all $\varepsilon > 0$, and $W_\varepsilon$ (resp. $\Tilde{W}_\varepsilon$) the Wigner Transform of $u_\varepsilon$ (resp. $v_\varepsilon$). Then
%
there exists a subsequence $(\varepsilon_n)_n$ such that $\varepsilon_n \underset{n \rightarrow \infty}{\longrightarrow} 0$ and two (non-negative) finite measures $W$ and $\Tilde{W}$ in $L^\infty ((0, \infty), \mathcal{M}(\mathbb{R}^d \times \mathbb{R}^d))$ such that for every $p \in [1, \infty)$
\begin{gather*}
    W_{\varepsilon_n} \underset{n \rightarrow \infty}{\rightharpoonup} W \qquad \text{in } L^p_\text{loc} ((0, \infty), \mathcal{A}'), \qquad \qquad \qquad
    \Tilde{W}_{\varepsilon_n} \underset{n \rightarrow \infty}{\rightharpoonup} \Tilde{W} \qquad \text{in } L^p_\text{loc} ((0, \infty), \mathcal{A}'),
\end{gather*}
and the relation between $W_\varepsilon$ and $\Tilde{W}_\varepsilon$ given by
\begin{equation}
    W_\varepsilon (t, x, \xi) = \frac{\lVert u_{\varepsilon,\textnormal{in}}\rVert_{L^2}^2}{\lVert \gamma^2 \rVert_{L^1}} \, \Tilde{W}_\varepsilon \left(t, \frac{x}{\tau (t)}, \tau(t) \, \xi - \Dot{\tau}(t) \, x \right) \label{wigner_rel}
\end{equation}
still holds after passing to the limit since $\lVert u_{\varepsilon_n,\textnormal{in}}\rVert_{L^2}$ converges (to some $M_0 \geq 0$) as $n \rightarrow \infty$. Furthermore, we have
\begin{gather*}
    \pi^{-\frac{d}{2}} \, \Tilde{\rho} (t, y) := \pi^{-\frac{d}{2}} \, \int_{\mathbb{R}^d} \Tilde{W} (t, y, \diff \eta) \in L^\infty((0, \infty), L_2^1 \cap L \log L(\mathbb{R}^d) ) \cap \mathcal{C}(\mathbb{R}^{+}, \mathcal{P}_1 (\mathbb{R}^d)) , \label{rho_reg_wigner} \\
    \Tilde{\rho} (t, \mathbb{R}^d) = \lVert \gamma^2 \rVert_{L^1} \qquad \text{for all } t \geq 0, \label{mass_cons_wigner}
\end{gather*}
and there exists $C_0 > 0$ such that
\begin{gather}
    \underset{t \geq 0}{\supess} \, \frac{1}{\tau (t)^2} \iint_{\mathbb{R}^d \times \mathbb{R}^d} |\eta|^2 \, \Tilde{W}(t,\diff y,\diff \eta) 
    +
    \int_0^\infty \frac{\Dot{\tau} (t)}{\tau^3 (t)} \iint_{\mathbb{R}^d \times \mathbb{R}^d} |\eta|^2 \, \Tilde{W}(t,\diff y,\diff \eta) \diff t \leq C_0, \label{est_WM_th} \\
    \int_{\mathbb{R}^d} y \, \Tilde{\rho}(t,y) \diff y = \frac{1}{\tau (t)} (C_1 t + C_2), \qquad \qquad \forall t \geq 0, \label{1st_mom_density_WM_th}
\end{gather}
where
\begin{equation*}
    C_j = \underset{n \rightarrow \infty}{\lim} \, \frac{\lVert\gamma^2\rVert_{L^1}}{\lVert u_{\varepsilon_n,\textnormal{in}} \rVert_{L^2}^2} \, I_{j,0}^{\varepsilon_n} \qquad \qquad \text{for } j=1,2,
\end{equation*}
which yields
\begin{equation*}
    \int_{\mathbb{R}^d} \begin{pmatrix}1 \\y \\\end{pmatrix} \Tilde{\rho} (t,y) \diff y \; \underset{t \rightarrow \infty}{\longrightarrow} \; \int_{\mathbb{R}^d} \begin{pmatrix}1 \\y \\\end{pmatrix} \gamma^2 (y) \diff y.
\end{equation*}
Lastly, there exists $C_3 > 0$ such that for all $t\geq2$,
\begin{equation*}
    \mathcal{W}_1 \left( \frac{\tilde{\rho} (t)}{\pi^\frac{d}{2}}, \frac{\gamma^2}{\pi^\frac{d}{2}} \right) \leq \frac{C_3}{\sqrt{\ln t}}. \label{conv_wass_dist_wigner}
\end{equation*}
\end{theorem}

The main result of this theorem is the fact that the two limits (semiclassical limit and large time behaviour) \textit{commute}. This is a strong feature which is rather unusual for those two kinds of limit. Indeed, it is known that such limits do not commute for linear Schrödinger equations with potential, in the context of scattering, with asymptotic states under the form of either WKB (see \cite{yajima1979, Yajima1981}), or coherent states (see e.g. \cite{Combescure_Robert__Coherent_states, Hagedorn_Joye, Hagedorn_semiclassical_III}). In \cite{Carles_scattering_2016}, a similar lack of commutativity is proven in the case of the Schrödinger equation with a potential and a cubic nonlinearity. 

\begin{rem}
Even if we do not have any pointwise convergence for $W_\varepsilon$ and if $W(t)$ is defined only for \textit{almost all} $t \in (0, \infty)$ to be a non-negative measure, $\rho (t)$ can be defined for \textit{all} $t \in (0, \infty)$ and is not only a non-negative measure but also an $L^1$ function. Moreover, we do have continuity in time for $\rho (t)$ with values in $\mathcal{P}_1$ endowed with the Wasserstein metric $\mathcal{W}_1$. Actually, the proof shows that we also get locally \textit{uniform} convergence in time of $| v_{\varepsilon_n} |^2$ to $\rho (t)$ in $\mathcal{P}_1$.
\end{rem}

\begin{rem} \label{rem_conv_2nd_mom_density_WM}
    The convergence for the second momentum stated in \eqref{momestschr_2_th} is uniform in $\varepsilon$. Yet, we still cannot conclude for the case "$\varepsilon = 0$" because we do not know if $\int_{\mathbb{R}^d} |y|^2 \, \Tilde{\rho}_\varepsilon (t,y) \diff y$ converges to $\int_{\mathbb{R}^d} |y|^2 \, \Tilde{\rho} (t,y) \diff y$. This would have been the case if, for example, we had an estimate for a higher momentum.
\end{rem}

\begin{rem}
Remember that if \eqref{WKB_state_assumption} is satisfied, $I_{i,0}^\varepsilon$ ($i = 1,2$) is independent of $\varepsilon$, therefore $C_j$ ($j=1,2$) are still the same quantities. Moreover, in the same case, it is known (see \cite[Exemple~III.5.]{PaulLions}) that
\begin{gather*}
    W_{\varepsilon} (0) \underset{n \rightarrow \infty}{\rightharpoonup} \rho_{\textnormal{in}} (x) \diff x \otimes \delta_{\xi = \nabla \phi_{\textnormal{in}} (x)} \qquad
    \text{ and } \qquad
    \Tilde{W}_{\varepsilon} (0) \underset{n \rightarrow \infty}{\rightharpoonup} \frac{\lVert\gamma^2\rVert_{L^1}}{\lVert\rho_{\textnormal{in}}\rVert_{L^1}} \rho_{\textnormal{in}} (x) \diff x \otimes \delta_{\xi = \nabla \phi_{\textnormal{in}} (x)} \qquad \text{in } \mathcal{A}'_{w-*}. \label{ex_in_data_wt}
\end{gather*}
\end{rem}

In the same way as for Corollary \ref{cor_other_conv_log_nls}, we also have a convergence rate for some other metrics.

\begin{cor} \label{cor_other_conv_wigner}
With the notations of Theorem \ref{main_th_wigner}, there exists $C_4 > 0$ such that for all $t \geq 2$ and $\delta \in (0,1)$,
\begin{gather*}
    \left\lVert \tilde{\rho} (t) - \gamma^2 \right\rVert_{W^{-1+\delta, 1}} \leq \frac{C_4}{\left( \ln t \right)^\frac{1-\delta}{2}}, \qquad 
    \mathcal{W}_{1+\delta} \left( \frac{\tilde{\rho} (t)}{\pi^\frac{d}{2}}, \frac{\gamma^2}{\pi^\frac{d}{2}} \right) \leq \frac{C_4}{\left( \ln t \right)^\frac{1-\delta}{2}}.
\end{gather*}
\end{cor}

\subsubsection{Kinetic Isothermal Euler system}

In view of the previous remarks on the Wigner Transform, the Wigner Measure usually satisfies the related kinetic (or Vlasov-type) equation with the same potential, as soon as the potential is smooth enough. Therefore, there is a formal link between the Wigner Measure we found in Theorem \ref{main_th_wigner} to the kinetic/Vlasov-type equation with the same potential, i.e.:
\begin{equation}
    \partial_t f + \xi \cdot \nabla_x f - \lambda \, \nabla_x (\ln \rho) \cdot \nabla_\xi f = 0, \qquad
    f(0,x,\xi) = f_{\textnormal{in}} (x, \xi), \qquad t > 0, (x, \xi) \in \mathbb{R}^d \times \mathbb{R}^d,
    \label{def_KIE}
\end{equation}
where $\rho (t,x) = \int_{\mathbb{R}^d} f(t, x, \diff \xi)$. First of all, this equation has a strong link with the \textit{isothermal Euler system}: a time-dependent mono-kinetic measure $f(t,x,\xi) = \rho (t,x) \diff x \otimes \delta_{\xi = v(t,x)}$ satisfies \eqref{def_KIE} if and only if $(\rho, v)$ satisfies:
\begin{System}
    \partial_t \rho + \nabla_x \cdot (\rho v) = 0, \\
    \partial_t (\rho v) + \nabla_x \cdot (\rho v \otimes v) + \lambda \nabla_x \rho = 0.
    \label{iso_eul_sys}
\end{System}
This is why \eqref{def_KIE} is called the \textit{Kinetic Isothermal Euler System} (KIE). Such an equation has already been studied in other contexts, mostly because it arises as the formal \textit{quasineutral limit} of the Vlasov-Poisson system with \textit{massless electrons}, but to the best of our knowledge the studies proving rigorously this quasineutral limit stick to the tore in space (see for instance \cite{Han-Kwan_Iacobelli_quasineutral_lim_VP, Griffin-Pickering_Iacobelli_VPME_to_KIE}). Even if it does not apply to our case, another interesting result is worth mentioning: the local well-posedness in 1D for mono-kinetic solutions far from vacuum and whose parameters $(\rho, v)$ are in Zhidkov space with enough regularity (see Theorem 1.4. of \cite{carlesnouri}).

For our case where the solution should have (at least) the same properties as in Theorem \ref{main_th_wigner}, some results were already found. In particular, R. Carles and A. Nouri proved that the Wigner Transform of solutions to \eqref{log_nls_eps} in 1D with Gaussian initial data converges (and even \textit{pointwise in time}) to a mono-kinetic measure, with $\rho$ Gaussian and $v$ affine in space, solution to \eqref{def_KIE} (see \cite[Theorem~1.1.]{carlesnouri}). We will name those solutions to \eqref{def_KIE} \textit{Gaussian-monokinetic} solutions. Such a remark strengthens the intuition of a link between \eqref{log_nls_eps} and \eqref{def_KIE} through the Wigner Transform.

Actually, even if it is not our purpose to develop a full Cauchy theory in this case, a nice framework for \eqref{def_KIE} should give the usual properties for Vlasov-type equations, and such properties are enough to prove the same large time behaviour in Wasserstein distance as in Theorems \ref{main_th_log_nls_eps} and \ref{main_th_wigner} (see Section \ref{subsec_disc_KIE}). This discussion is even more enlightened by the following result, providing a new class of explicit global \textit{strong} solutions to \eqref{def_KIE} in 1D: \textit{Gaussian-Gaussian} solutions.

\begin{theorem}
\label{main_th_gauss}
\begin{enumerate}
    \item \label{part_1_main_th_gauss} For $c_{1,0} > 0$, $c_{2,0} > 0$ and $c_{1,1}, B_0, B_1 \in \mathbb{R}$, define $c_1 \in \mathcal{C}^\infty (\mathbb{R}^+)$ the solution of the ordinary differential equation
    \begin{equation}
        \Ddot{c}_1 = \frac{2 \lambda}{c_1} + \frac{\Tilde{C}^2}{c_1^3}, \qquad \qquad \Tilde{C} := c_{1,0} \, c_{2,0}, \qquad
        c_1(0) = c_{1,0}, \qquad
        \Dot{c}_1 (0) = c_{1,1}. \label{gaussian1_th}
    \end{equation}
    Then, set
    \begin{gather}
        c_2 (t) := \frac{\Tilde{C}}{c_1 (t)}, \quad \qquad
        b_1 (t) := B_1 t + B_0, \quad \qquad
        b_2(t,x) := \frac{\Dot{c}_1(t)}{c_1(t)} (x - B_1 \, t - B_0) + B_1. \label{gaussian2_th}
    \end{gather}
    The function $f = f(t,x,\xi)$ defined by
    \begin{align}
        f(t,x,\xi) = \frac{1}{\pi \, c_1(t) \, c_2(t)} \ \exp{\left[ - \frac{|x-b_1(t)|^2}{c_1(t)^2} - \frac{|\xi - b_2(t,x)|^2}{c_2(t)^2} \right]} \label{expr_gauss_gauss_sol_KIE}
    \end{align} is a solution to \eqref{def_KIE}. Moreover, if we rescale it to $\Tilde{f} = \Tilde{f} (t, y, \eta)$ by setting
    \begin{equation*}
        f (t, x, \xi) := \frac{1}{\lVert \gamma^2 \rVert_{L^1}} \, \Tilde{f} \left(t, \frac{x}{\tau (t)}, \tau(t) \, \xi - \Dot{\tau}(t) \, x \right),
    \end{equation*}
    and define
    \begin{equation*}
        \Tilde{\rho} (t, y) := \int_{\mathbb{R}^d} \Tilde{f} (t,y,\eta) \diff \eta,
    \end{equation*}
    there holds
    \begin{equation}
        \lVert \tilde{\rho} (t,.) - \gamma^2 \rVert_{L^1} = \textnormal{O} \hspace{-0.5mm} \left( \sqrt{\frac{\ln \ln t}{\ln t}} \right). \label{L1_conv_gauss_case}
    \end{equation}
    \item \label{part_2_main_th_gauss} Let $T \in (0, + \infty]$, $b_1 = b_1(t) \in \mathcal{C}^1 ([0, T), \mathbb{R})$, $c_1=c_1(t) \in \mathcal{C}^1 ([0, T), (0, \infty))$, $b_2=b_2(t,x) \in \mathcal{C}^1 ([0, T) \times \mathbb{R}, \mathbb{R})$ and $c_2 = c_2(t,x) \in \mathcal{C}^1 ([0, T) \times \mathbb{R}, (0, \infty))$ such that 
    \begin{align*}
        f(t,x,\xi) = \frac{1}{\pi \, c_1(t) \, c_2(t, x)} \ \exp{\left[ - \frac{|x-b_1(t)|^2}{c_1(t)^2} - \frac{|\xi - b_2(t,x)|^2}{c_2(t,x)^2} \right]}, \qquad t \in [0, T), \ x, \xi \in \mathbb{R},
    \end{align*} is a solution to \eqref{def_KIE}. Then $c_2$ does not depend on $x$, all the functions are $\mathcal{C}^\infty$ and \eqref{gaussian1_th}-\eqref{gaussian2_th} hold.
\end{enumerate}
\end{theorem}

\begin{rem}
    This theorem may also handle the case when $c_{2,0} = 0$, which is actually the monokinetic case where we have a Dirac in $\xi$:
    \begin{equation*}
        f_{\textnormal{in}} (x,\xi) = \frac{1}{\sqrt{\pi} \, c_{1,0}} \exp{\left[ - \frac{(x-b_{1,0})^2}{c_{1,0}^2} \right] \otimes \delta_{\xi = b_{2,0} (x)}}.
    \end{equation*}
    where $b_{2,0} (x)$  is affine. Then the previous theorem shows that $f$ is a Dirac in $\xi$ for all time (if we only consider Gaussian solutions), as $c_1 (t) \, c_2 (t) = c_1(0) \, c_2(0) = 0$ with $c_1 (t) \neq 0$ for all $t>0$. This is similar to \cite{carlesnouri}.
\end{rem}

\begin{rem}
We stated this result in 1D, however we can extend this class of solutions (and also the Gaussian-monokinetic class) to any dimension $d$ by tensor product. Indeed, in the same way as for \eqref{log_nls_eps} (see \cite{nonlin_wave_mec}), the tensor product of two solutions to \eqref{def_KIE} is still a solution to \eqref{def_KIE}.
\end{rem}

\begin{rem}
It is worth noting that the expected large time behaviour still holds, due to the fact that the behaviour of $c_1$ has already been studied in \cite{carlesnouri}:
    \begin{equation*}
        c_1 (t) \underset{t \rightarrow \infty}{\sim} 2t \sqrt{\lambda \ln t},
    \end{equation*}
    with the better result of strong convergence of $\Tilde{\rho}$ to $\gamma^2$ in $L^1$ with a slightly slower convergence. This does not mean that the convergence in Wasserstein distance is slower for this class of solutions. Actually, we can prove that the convergence rate found in Theorem \ref{main_th_wigner} still holds in this case, despite the fact that those solutions do not fit with any Wigner Measure.
\end{rem}

\subsection{Outline of the paper}

In Section \ref{section_wigner} we review and extend some of the standard facts on the Wigner Transform.
Section \ref{section_semiclassical_limit_log_nls} is devoted to the study of the semiclassical limit of \eqref{log_nls_eps} through the first part of the proof of Theorems \ref{main_th_log_nls_eps} and \ref{main_th_wigner}, which is everything except the convergence rate in Wasserstein distance: we extend the results of \cite{carlesgallagher} to \eqref{log_nls_eps} (with semiclassical constant) and then use the results on the Wigner Transform to perform the semiclassical limit.
A sharpened analysis of the Fokker-Planck equation (which already gives the weak convergence in \cite{carlesgallagher}) is provided in Section \ref{section_FP}. The estimates coming from this analysis lead to the convergence rate in Wasserstein metric. Finally, Section \ref{section_KIE} is split into two parts. 
The first part contains a discussion of the Kinetic Isothermal Euler system and its (formal) properties. We show that those properties are enough to get the same behaviour as in Theorem \ref{main_th_wigner}, through an intermediate result we prove in Section \ref{section_FP}.
The last part deals with Theorem \ref{main_th_gauss} and its new class of explicit solutions to \eqref{def_KIE}.

\subsection*{Acknowledgements}

The author wishes to thank Rémi Carles and Matthieu Hillairet for enlightening discussions about this work and numerous constructive remarks on the writing of this paper, and also Kléber Carrapatoso for his precious help on the harmonic Fokker-Planck operator.

\section{Wigner Transform and Wigner Measure}
\label{section_wigner}

This section is devoted to the Wigner Transform and Wigner Measure. Even if they have already been studied a lot (see \cite{PaulLions, athanassoulis-paul, gerard-mark-mauser, Gerard9091}), many standard facts about them were proved without taking into account the time dependence. Indeed, the further results of the convergence of the Wigner Measure of a solution to a Schrödinger equation to the related kinetic/Vlasov-type equation conclude to a convergence which is pointwise in time for a lot of cases (see for instance \cite[Théorèmes~IV.1.~and~IV.2.]{PaulLions}), therefore those facts are enough to get suitable properties for the Wigner Measure. However, those results do not fall within our framework, so we need to extend those basic facts to the case with time dependence.

\subsection{Definitions and first property}

For any sequence of functions $f_\varepsilon = f_\varepsilon (x) \in L^2 (\mathbb{R}^d)$ for $\varepsilon>0$, define the Wigner Transform $W_\varepsilon$ by
\begin{equation*}
    W_\varepsilon (x, \xi) = \frac{1}{(2 \pi)^d} \int_{\mathbb{R}^d} e^{-i \xi \cdot z} f_\varepsilon \left(x + \frac{\varepsilon z}{2} \right) \, \overline{f_\varepsilon \left(x - \frac{\varepsilon z}{2} \right)} \diff z = \mathcal{F}_z \Tilde{\rho}_\varepsilon (x, \xi), \qquad (x,\xi) \in \mathbb{R}^d \times \mathbb{R}^d,
    \label{def_wigner_transf}
\end{equation*}
where
\begin{equation*}
    \Tilde{\rho}_\varepsilon (x, z) = f_\varepsilon \left( x + \frac{\varepsilon z}{2} \right) \, \overline{f_\varepsilon \left( x - \frac{\varepsilon z}{2} \right)}, \qquad (x,z) \in \mathbb{R}^d \times \mathbb{R}^d.
    \label{def_rho_tilde_eps}
\end{equation*}
$W_\varepsilon$ is a real-valued function on the phase space. However, it may be non-integrable and sometimes negative. Both issues are fixed by working with the Husimi Transform, which is a slight modification of the Wigner Transform. For this purpose, we define the Gaussian with $\varepsilon$ variance:
\begin{gather*}
    \gamma_\varepsilon (x) = \frac{1}{(\pi \varepsilon)^\frac{d}{2}} \exp{\left( - \frac{|x|^2}{\varepsilon} \right)}, \qquad
    G_\varepsilon (x, \xi) = \gamma_\varepsilon (x) \, \gamma_\varepsilon (\xi), \qquad \text{for } x,\xi \in \mathbb{R}^d.
\end{gather*}
This leads to the definition of the Husimi Transform $W_\varepsilon^H$:
\begin{equation}
    W_\varepsilon^H = W_\varepsilon * G_\varepsilon = W_\varepsilon *_x \gamma_\varepsilon *_\xi \gamma_\varepsilon. \label{def_husimi}
\end{equation}
The fact that the Husimi Transform is non-negative and integrable is not obvious at first sight, but this is well-known (see \cite{PaulLions}).

\begin{prop}
\label{prop_int_HT}
    The Husimi Transform $W_\varepsilon^H = W_\varepsilon^H (x,\xi)$ of any function $f_\varepsilon \in L^2 (\mathbb{R}^d)$ defined by \eqref{def_husimi} is non-negative and satisfies
    \begin{equation}
        \int_{\mathbb{R}^d} W_\varepsilon^H (x, \xi) \diff \xi = |f_\varepsilon|^2 * \gamma_\varepsilon (x), \qquad \text{for all } x \in \mathbb{R}^d \label{whepsmom0}
    \end{equation}
    and
    \begin{equation*}
        \iint_{\mathbb{R}^d \times \mathbb{R}^d} W_\varepsilon^H (x,\xi) \diff \xi \diff x = \lVert f_\varepsilon\rVert_{L^2}^2.
    \end{equation*}
\end{prop}

\subsection{Momenta}

The fact that the Husimi Transform is non-negative is very useful in order to compute some momenta. As it is a slight modification of the Wigner Transform, their computation leads to some interesting estimates even in the limit $\varepsilon \rightarrow 0$.

\begin{prop}
\label{propmomht}
Given any $f_\varepsilon \in L^2 (\mathbb{R}^d)$, $\varepsilon > 0$ and its Husimi Transform $W_\varepsilon^H$, there holds for all $x \in \mathbb{R}^d$ :
\begin{enumerate}
    \item \label{item_2_wigner} If $f_\varepsilon \in H^1 (\mathbb{R}^d)$,
\begin{equation}
    \int_{\mathbb{R}^d} |\xi|^2 \, W_\varepsilon^H (x, \xi) \diff \xi = \varepsilon^2 \, |\nabla f_\varepsilon|^2 * \gamma_\varepsilon (x) - \frac{\varepsilon^2}{4} \, |f_\varepsilon|^2 * \Delta \gamma_\varepsilon (x) + \frac{\varepsilon d}{2} |f_\varepsilon|^2 * \gamma_\varepsilon (x), \label{whepsmom1}
\end{equation}
and
\begin{equation}
    \iint_{\mathbb{R}^d \times \mathbb{R}^d} |\xi|^2 \, W_\varepsilon^H (x, \xi) \diff \xi \diff x = \varepsilon^2 \lVert\nabla f_\varepsilon\rVert_{L^2}^2 + \frac{\varepsilon d}{2} \lVert f_\varepsilon\rVert_{L^2}^2.
    \label{whepsmom}
\end{equation}
In a more general way,
\begin{equation}
    \int_{\mathbb{R}^d} \xi_i \xi_j \, W_\varepsilon^H (x, \xi) \diff \xi = \varepsilon^2 \, \Re \left( \partial_i f_\varepsilon \, \overline{\partial_j f_\varepsilon} \right) * \gamma_\varepsilon (x) - \frac{\varepsilon^2}{4} \, |f_\varepsilon|^2 * \partial_i \partial_j \gamma_\varepsilon (x) + \frac{\varepsilon \, \delta_{ij}}{2} |f_\varepsilon|^2 * \gamma_\varepsilon (x), \label{whepsmom01}
\end{equation}
and
\begin{equation}
    \iint_{\mathbb{R}^d \times \mathbb{R}^d} \xi_i \xi_j \, W_\varepsilon^H (x, \xi) \diff \xi \diff x = \varepsilon^2 \int_{\mathbb{R}^d} \Re \left( \partial_i f_\varepsilon \, \overline{\partial_j f_\varepsilon } \right) \diff x + \frac{\varepsilon \, \delta_{ij}}{2} \lVert u_\varepsilon\rVert_{L^2}^2. \label{whepsmom02}
\end{equation}
\item \label{item_3_wigner} If $f_\varepsilon \in H^1 (\mathbb{R}^d)$,
\begin{equation}
    \int_{\mathbb{R}^d} \xi \, W_\varepsilon^H (x, \xi) \diff \xi = \varepsilon \Im ( \nabla f_\varepsilon \, \overline{f_\varepsilon} ) * \gamma_\varepsilon \left(x \right) ,
    \label{wheps1stmom}
\end{equation}
and therefore
\begin{equation}
    \iint_{\mathbb{R}^d \times \mathbb{R}^d} \xi \, W_\varepsilon^H (x, \xi) \diff \xi \diff x = \int_{\mathbb{R}^d} \varepsilon \Im ( \nabla f_\varepsilon \, \overline{f_\varepsilon} ) \diff x. \label{wheps1stmom2}
\end{equation}
\item \label{item_4_wigner} If $u_\varepsilon \in \mathcal{F}(H^1)$,
\begin{equation}
    \iint_{\mathbb{R}^d \times \mathbb{R}^d} |x|^2 \, W^H_\varepsilon (x, \xi) \diff x \diff \xi = \lVert x \, f_\varepsilon \rVert_{L^2}^2 + \frac{\varepsilon d}{2} \, \lVert f_\varepsilon \rVert_{L^2}^2.
    \label{whepsmom2}
\end{equation}
\end{enumerate}
\end{prop}

The proof is very computational and will be done in Appendix \ref{proofwhepsmom}.

\subsection{Semiclassical limit}

Even if the Wigner Transform is not integrable, we still have some bounds thanks to the following Banach space (and algebra) of test functions:
\begin{equation*}
    \mathcal{A} = \{ \phi \in \mathcal{C}_0 (\mathbb{R}^d_x \times \mathbb{R}^d_\xi ), (\mathcal{F}_\xi \phi)(x,z) \in L^1 (\mathbb{R}^d_z, \mathcal{C}_0(\mathbb{R}^d_x)) \}
\end{equation*}
endowed with the norm
\begin{equation*}
    \lVert \phi \rVert_\mathcal{A} = \lVert \mathcal{F}_\xi \phi \rVert_{L^1_z L^\infty_x}.
\end{equation*}
It is known that, for any sequence $(f_\varepsilon = f_\varepsilon (x))_{\varepsilon > 0}$ bounded in $L^2 (\mathbb{R}^d)$, its Wigner Transform $W_\varepsilon$ is uniformly bounded in $\mathcal{A}'$ and therefore weak-$*$ converges (up to the extraction of a subsequence) to a (non-negative) measure, called a Wigner Measure (see \cite[Proposition~III.1.]{PaulLions}). Adding the time-dependence to the boundedness is obviously easy.
%
%
However, reaching the limit with the addition of the time-dependence is a bit more difficult. Usually, the (Schrödinger) equation satisfied by $u_\varepsilon$ yields an equation on $W_\varepsilon$ from which one can derive some equicontinuity if the potential is smooth enough, but here the potential is highly singular because we do not have any control near the vacuum. Yet, the uniform bound in $L^\infty \left((0,T), \mathcal{A}' \right)$ implies the uniform bound in $L^{p'} \left((0,T), \mathcal{A}' \right) = \left( L^p \left((0,T), \mathcal{A} \right) \right)'$ for any $p > 1$. This remark shows that we can extend the result of \cite[Théorème~III.1.]{PaulLions} with time-dependence, leave to lose pointwise convergence in time.

We say that a sequence $(f_\varepsilon)_{\varepsilon>0}$ of functions $f_\varepsilon = f_\varepsilon (t,x) \in L^\infty \left( (0, T), H^1 \cap \mathcal{F} (H^1) (\mathbb{R}^d) \right)$ satisfies the assumption \eqref{assumption_3} for some $T > 0$ if
\begin{equation}
    f_\varepsilon, \, x \, f_\varepsilon \text{ and } \varepsilon \nabla f_\varepsilon \text{ are uniformly bounded in } L^\infty \left( (0,T), L^2 ( \mathbb{R}^d ) \right).
    \label{assumption_3} \tag{A3}
\end{equation}

\begin{lem}
\label{lem_wigner}
    \begin{enumerate}
    \item Given any sequence $(f_\varepsilon)_{\varepsilon>0}$ of functions $f_\varepsilon = f_\varepsilon (t,x) \in L^\infty \left( (0, T), L^2 (\mathbb{R}^d) \right)$ uniformly bounded, there exists a subsequence $(\varepsilon_n)_n$ such that $\varepsilon_n \underset{n \rightarrow \infty}{\longrightarrow} 0$ and there exists a (non-negative) measure $W$ (called Wigner Measure) in $L^\infty ((0, T), \mathcal{M}(\mathbb{R}^d \times \mathbb{R}^d))$ such that 
    for every $p \in (1, \infty)$
    \begin{gather*}
        W_{\varepsilon_n} \underset{n \rightarrow \infty}{\rightharpoonup} W \qquad \text{in } L^p \left((0, T), \mathcal{A}'_{w-*}\right), \qquad \qquad
        W_{\varepsilon_n}^H \underset{n \rightarrow \infty}{\rightharpoonup} W \qquad \text{in } L^p \left((0, T), \mathcal{M} (\mathbb{R}^d \times \mathbb{R}^d)\right).
    \end{gather*}
    \item \label{item_2_wigner_facts} Moreover, if $f_\varepsilon = f_\varepsilon (t,x)$ 
    satisfy \eqref{assumption_3}, then $W_{\varepsilon_n}^H \underset{n \rightarrow \infty}{\longrightarrow} W$ in $L^p \left((0, T), \mathcal{M} (\mathbb{R}^d \times \mathbb{R}^d)\right)$ and the following properties hold for a.e. $t \in (0, T)$ and all $p>1$,
    \begin{itemize}
        \item On the second momentum in $x$:
        \begin{equation*}
            \hspace{-15mm} |x|^2 \, W_{\varepsilon_n}^H \underset{n \rightarrow \infty}{\rightharpoonup} |x|^2 \, W \quad \text{in } L^p \left((0, T), \mathcal{M} (\mathbb{R}^d \times \mathbb{R}^d) \right), \quad \iint_{\mathbb{R}^d \times \mathbb{R}^d} |x|^2 \, W (t, \diff x, \diff \xi) \leq \underset{n \rightarrow \infty}{\liminf} \, \lVert x f_{\varepsilon_n} \rVert_{L^\infty_t L^2_x}^2,
        \end{equation*}
        \item On the second momentum in $\xi$:
        \begin{equation*}
            \hspace{-17mm} |\xi|^2 \, W_{\varepsilon_n}^H \underset{n \rightarrow \infty}{\rightharpoonup} |\xi|^2 \, W \quad \text{in } L^p \left((0, T), \mathcal{M} (\mathbb{R}^d \times \mathbb{R}^d)\right), \quad
            \iint_{\mathbb{R}^d \times \mathbb{R}^d} |\xi|^2 \, W (t, \diff x, \diff \xi) \leq \underset{n \rightarrow \infty}{\liminf} \, (\varepsilon_n)^2 \lVert \nabla f_{\varepsilon_n} \rVert_{L^\infty_t L^2_x}^2,
        \end{equation*}
        \item On the density:
        \begin{equation*}
            \hspace{-13mm} | f_{\varepsilon_n} |^2 \underset{n \rightarrow \infty}{\longrightarrow} \rho := \int_{\mathbb{R}^d} W(., ., \diff \xi) \quad \text{in } L^p \left((0, T), \mathcal{M} (\mathbb{R}^d)\right), \qquad
        \lVert f_{\varepsilon_n} \rVert_{L^2}^2 \underset{n \rightarrow \infty}{\rightharpoonup} \rho(., \mathbb{R}^d) \quad \text{ in } L^{p} ((0,T)).
        \end{equation*}
    \end{itemize}
    \end{enumerate}
\end{lem}

\begin{rem}
    $W_{\varepsilon_n} \underset{n \rightarrow \infty}{\rightharpoonup} W$ in $L^p ((0, T), \mathcal{A}'_{w-*})$ means that for any $\phi = \phi (t, x, \xi)$ such that $\lVert \phi (t) \rVert_\mathcal{A} \in L^{p'} ((0,T))$, $\int_0^T \iint_{\mathbb{R}^d \times \mathbb{R}^d} \phi(t,x,\xi) \, W_{\varepsilon_n} (t,x,\xi) \diff x \diff \xi \diff t$ converges to $\int_0^T \iint_{\mathbb{R}^d \times \mathbb{R}^d} \phi(t,x,\xi) \, W (t,\diff x,\diff \xi) \diff t$.
    In the same way, $W_{\varepsilon_n}^H \underset{n \rightarrow \infty}{\rightharpoonup} W$ in $L^p ((0, T), \mathcal{M} (\mathbb{R}^d \times \mathbb{R}^d))$ means that for any $\phi = \phi (t, x, \xi)$ such that $\phi (t) \in \mathcal{C}_0 (\mathbb{R}^d \times \mathbb{R}^d)$ (continuous and going to 0 at infinity) for a.e. $t \in (0,T)$ and $\lVert \phi (t) \rVert_{L^\infty} \in L^{p'} (0,T)$, $\int_0^T \iint_{\mathbb{R}^d \times \mathbb{R}^d} \phi(t,x,\xi) \, W_{\varepsilon_n}^H (t,x,\xi) \diff x \diff \xi \diff t$ converges to $\int_0^T \iint_{\mathbb{R}^d \times \mathbb{R}^d} \phi(t,x,\xi) \, W (t,\diff x,\diff \xi) \diff t$.
    The same kind of remark holds for the convergence $| f_{\varepsilon_n} |^2 \underset{n \rightarrow \infty}{\longrightarrow} \rho$ when taking $\phi \in L^{p'} \left( (0,T), \mathcal{C}_b (\mathbb{R}^d ) \right)$, and also for $W_{\varepsilon_n}^H \underset{n \rightarrow \infty}{\longrightarrow} W$.
\end{rem}

\begin{rem}
The assumption \eqref{assumption_3} is not the sharpest for the results about the density $\rho$. Actually, one shall only need some $\varepsilon$\textit{-oscillatory} and \textit{compact at infinity} feature \textit{uniformly in time} for the sequence $(f_\varepsilon)_{\varepsilon>0}$. However, the assumption \eqref{assumption_3} makes the proof easier, also allows to get good properties on the second momentum of the Wigner Measure and is actually sufficient for our further results.
\end{rem}

\begin{proof}
The first part of the proof is actually a re-writing of the proof of the first part of \cite[Théorème~III.1.]{PaulLions}, with in addition the time-dependence. 
$W_\varepsilon$ and $W_\varepsilon^H$ are bounded respectively in $L^\infty \left((0, T), \mathcal{A}' \right)$ and in $L^\infty \left( (0, T), L^1(\mathbb{R}^d \times \mathbb{R}^d) \right)$. Thus, there exists a subsequence $\varepsilon_n$ such that $W_{\varepsilon_n}$ (resp. $W_{\varepsilon_n}^H$) weakly converges in $L^p\left((0,T), \mathcal{A}'_{w-*}\right)$ (resp. $L^p\left((0,T), \mathcal{M} (\mathbb{R}^d \times \mathbb{R}^d) \right)$) for all $p\in(1,\infty)$ to a limit $W \in L^\infty\left((0, T), \mathcal{A}'\right)$ (resp. $W^T_H \in L^\infty \left( (0, T), \mathcal{M} (\mathbb{R}^d \times \mathbb{R}^d) \right)$).
Following the idea of \cite[Théorème~III.1.]{PaulLions}, we should be able to prove that $W = W_H$. Since we have
\begin{equation*}
    W^H_\varepsilon = W_\varepsilon * G_\varepsilon, \qquad \text{where } G_\varepsilon = \frac{1}{(\pi \varepsilon)^d} \, e^{- \frac{|x|^2 + |\xi|^2}{\varepsilon}},
\end{equation*}
it is enough to prove that, for example, for any $\phi \in \mathcal{C}_c^\infty((0, T) \times \mathbb{R}^d \times \mathbb{R}^d)$ which is a dense subset of $L^2((0, T), \mathcal{A})$
, $\phi * G_\varepsilon$ converges in $L^2 ((0, T), \mathcal{A})$ to $\phi$. Knowing that
\begin{equation*}
    \mathcal{F}_\xi ( \phi * G_\varepsilon ) (t,x,z) = \left[ \mathcal{F}_\xi \phi (t,x,z) *_x \frac{1}{(\pi \varepsilon)^\frac{d}{2}} \, e^{- \frac{|x|^2}{\varepsilon}} \right] e^{- \varepsilon \frac{|z|^2}{4}},
\end{equation*}
we see that,
\begin{multline*}
    \lVert \phi (t) * G_\varepsilon - \phi (t) \rVert_\mathcal{A} \leq \int_{\mathbb{R}^d} \underset{x}{\sup} \left\lvert \mathcal{F}_\xi \phi (t) - \mathcal{F}_\xi \phi (t) *_x \frac{1}{(\pi \varepsilon)^\frac{d}{2}} \, e^{- \frac{|x|^2}{\varepsilon}} \right\rvert \diff z
    + \int_{\mathbb{R}^d} (1 - e^{- \varepsilon \frac{|z|^2}{4}}) \, \underset{x}{\sup} \left\lvert \mathcal{F}_\xi \phi (t) \right\rvert \diff z.
\end{multline*}
The second term goes to 0 when $\varepsilon$ goes to 0 by dominated convergence for all $t \in (0,T)$, and so does the first term since 
$\mathcal{F}_\xi \phi (t) \in \mathcal{S}(\mathbb{R}^d \times \mathbb{R}^d)$. Moreover,
\begin{align*}
    \lVert \mathcal{F}_\xi ( \phi * G_\varepsilon ) (t) \rVert_{L^1_z L^\infty_x} &= \left\lVert \left[ \mathcal{F}_\xi \phi (t) *_x \frac{1}{(\pi \varepsilon)^\frac{d}{2}} e^{- \frac{|x|^2}{\varepsilon}} \right] e^{- \varepsilon \frac{|z|^2}{4}} \right\rVert_{L^1_z L^\infty_x}
    \leq \lVert \mathcal{F}_\xi \phi (t) \rVert_{L^1_z L^\infty_x} = \lVert \phi (t) \rVert_\mathcal{A},
\end{align*}
which yields
\begin{equation*}
    \lVert \phi (t) * G_\varepsilon - \phi (t) \rVert_\mathcal{A} \leq \lVert \phi (t) * G_\varepsilon \rVert_\mathcal{A} + \lVert \phi (t) \rVert_\mathcal{A} \leq 2 \lVert \phi (t) \rVert_\mathcal{A}.
\end{equation*}
Then, 
we conclude by dominated convergence
\begin{equation*}
    \int_0^T \lVert \phi (t) * G_\varepsilon - \phi (t) \rVert_\mathcal{A}^2 \diff t \underset{\varepsilon \rightarrow 0}{\longrightarrow} 0,
\end{equation*}
which is what we wanted. Therefore, $W = W^H$.

The proof of part \ref{item_2_wigner_facts} is rather usual. First, take some non-increasing $\chi \in \mathcal{C}_c^\infty ([0, \infty))$ such that $\chi \equiv 1$ on $[0,1]$ and $0 \leq \chi \leq 1$. Given any non-negative function $\phi = \phi (t) \in \mathcal{C}^\infty_c ((0,T))$, we know that for any $\delta > 0$,
\begin{equation*}
    \int_0^T \iint \phi(t) \, \chi (\delta |x|^2) \, |x|^2 \, W_{\varepsilon_n}^H (t, x, \xi) \diff x \diff \xi \diff t \underset{n \rightarrow \infty}{\longrightarrow} \int_0^T \iint \phi(t) \, \chi (\delta |x|^2) \, |x|^2 \, W (t, \diff x, \diff \xi) \diff t,
\end{equation*}

Since all the factors are non-negative, the term on the left-hand side is bounded thanks to \eqref{whepsmom2}:
\begin{align*}
    \int_0^T \iint \phi(t) \, \chi (\delta |x|^2) \, |x|^2 \, W_{\varepsilon_n}^H (t, x, \xi) \diff x \diff \xi \diff t
        &\leq \int_0^T \phi(t) \iint |x|^2 \, W_{\varepsilon_n}^H (t, x, \xi) \diff x \diff \xi \diff t \\
        &\leq \int_0^T \phi(t) \left[ \lVert x \, f_{\varepsilon_n} \rVert_{L^2}^2 + \frac{\varepsilon_n d}{2} \, \lVert f_{\varepsilon_n}\rVert_{L^2}^2 \right] \diff t \\
        &\leq \lVert \phi \rVert_{L^1} \left[ \lVert x \, f_{\varepsilon_n} \rVert_{L_t^\infty L_x^2}^2 + \frac{\varepsilon_n d}{2} \, \lVert f_{\varepsilon_n}\rVert_{L_t^\infty L_x^2}^2 \right].
\end{align*}
Therefore, we get
\begin{equation*}
    \int_0^T \iint \phi(t) \, \chi (\delta |x|^2) \, |x|^2 \, W (t, \diff x, \diff \xi) \diff t \leq \lVert \phi \rVert_{L^1} \underset{n \rightarrow \infty}{\liminf} \, \lVert x f_{\varepsilon_n} \rVert_{L^\infty_t L^2_x}^2,
\end{equation*}
and we conclude thanks to the monotone convergence theorem as $\delta \rightarrow 0$:
\begin{equation*}
    \int_0^T \iint \phi(t) \, |x|^2 \, W (t, \diff x, \diff \xi) \diff t \leq \lVert \phi \rVert_{L^1} \underset{n \rightarrow \infty}{\liminf} \, \lVert x f_{\varepsilon_n} \rVert_{L^\infty_t L^2_x}^2,
\end{equation*}
hence
\begin{equation*}
    \iint |x|^2 \, W (t, \diff x, \diff \xi) \leq \underset{n \rightarrow \infty}{\liminf} \, \lVert x f_{\varepsilon_n} \rVert_{L^\infty_t L^2_x}^2 \qquad \text{for a.e. } t \in (0, T),
\end{equation*}
and the same proof holds for the second momentum in $\xi$ thanks to \eqref{whepsmom}. Getting this second momentum leads to the following result with a usual argument:
\begin{gather*}
    |x|^2 \, W_{\varepsilon_n}^H \underset{n \rightarrow \infty}{\rightharpoonup} |x|^2 \, W \quad \text{in } L^p ((0, T), \mathcal{M} (\mathbb{R}^d \times \mathbb{R}^d)), \qquad
    | \xi |^2 \, W_{\varepsilon_n}^H \underset{n \rightarrow \infty}{\rightharpoonup} | \xi |^2 \, W \quad \text{in } L^p ((0, T), \mathcal{M} (\mathbb{R}^d \times \mathbb{R}^d)),
\end{gather*}
and thus
\begin{gather*}
    W_{\varepsilon_n}^H \underset{n \rightarrow \infty}{\longrightarrow} W \qquad \text{in } L^p ((0, T), \mathcal{M} (\mathbb{R}^d \times \mathbb{R}^d)).
\end{gather*}
In particular, 
thanks to \eqref{whepsmom0},
\begin{equation*}
    \int_{\mathbb{R}^d} W_{\varepsilon_n}^H (t,x,\xi) \diff \xi = |f_{\varepsilon_n} (t,.) |^2 * \gamma_{\varepsilon_n} (x) \underset{n \rightarrow \infty}{\longrightarrow} \rho \qquad \text{in } L^p ((0, T), \mathcal{M} (\mathbb{R}^d)).
\end{equation*}
But $(f_{\varepsilon_n})_n$ is uniformly bounded in $L^\infty\left((0,T), \mathcal{F}(H^1) (\mathbb{R}^d) \right)$, thus we get up to a further subsequence
\begin{equation*}
    |f_{\varepsilon_n} |^2 \underset{n \rightarrow \infty}{\longrightarrow} \tilde{\rho} \qquad \text{in } L^p ((0, T), \mathcal{M} (\mathbb{R}^d)),
\end{equation*}
for some $\tilde{\rho} \in L^\infty ((0, T), \mathcal{M} (\mathbb{R}^d))$. In particular, it is also known that
\begin{equation*}
    |f_{\varepsilon_n} (t,.) |^2 * \gamma_{\varepsilon_n} (x) \underset{n \rightarrow \infty}{\longrightarrow} \tilde{\rho} \qquad \text{in } L^p ((0, T), \mathcal{M} (\mathbb{R}^d)).
\end{equation*}
Therefore $\tilde{\rho} = \rho$, hence the whole sequence $(|f_{\varepsilon_n}|^2)_n$ converges (there is no need of further subsequence) and 
especially
\begin{equation*}
    \lVert f_{\varepsilon_n} (t)\rVert_{L^2}^2 = \int_{\mathbb{R}^d} |f_{\varepsilon_n} (t,x)|^2 \diff x \underset{n \rightarrow \infty}{\rightharpoonup} \int_{\mathbb{R}^d} \tilde{\rho} (t, \diff x) \qquad \text{in } L^p ((0, T)). \qedhere
\end{equation*}
\end{proof}

\section{Semiclassical limit of the Logarithmic Schrödinger equation}
\label{section_semiclassical_limit_log_nls}

In this section, we prove Theorems \ref{th_cauchy_log_nls_eps}, \ref{main_th_log_nls_eps} and \ref{main_th_wigner} except the convergence rates in Wasserstein distance, which will be done in the next section. First, a brief proof of Theorem \ref{th_cauchy_log_nls_eps} and a longer one for Theorem \ref{main_th_log_nls_eps} are given. Using those results along with the properties of section \ref{section_wigner}, the semiclassical limit is then performed and gives the first part of the proof of the latter.

\subsection{Proof of theorems \ref{th_cauchy_log_nls_eps} and \ref{main_th_log_nls_eps}}

The proof of Theorem \ref{th_cauchy_log_nls_eps} is very easy and follows from a simple change of variable: $u_\varepsilon$ is a solution to \eqref{log_nls_eps} if and only if $\tilde{u}_\varepsilon (t,x) = u_\varepsilon (\varepsilon t, \varepsilon x)$ is solution to \eqref{log_nls} (with initial data $u_{\varepsilon,\textnormal{in}} ( \varepsilon x )$). Therefore, we can use \cite[Theorem~1.5.]{carlesgallagher} and it leads to the conclusion with some additional and obvious computations.
For Theorem \ref{main_th_log_nls_eps}, the first part of the proof is actually a slight and simple adaptation of the proof of \cite[Theorem~1.7.]{carlesgallagher}.

\subsubsection{Rescaling and estimates}

Writing \eqref{log_nls_eps} in terms of $v_\varepsilon$ yields
\begin{equation*}
    i \varepsilon \, \partial_t v_\varepsilon + \frac{\varepsilon^2}{2 \tau(t)^2} \Delta_y v_\varepsilon = \lambda v_\varepsilon \ln{ \left\lvert \frac{v_\varepsilon}{\gamma} \right\rvert^2} - \lambda \, \left(d \ln{\tau(t)} - 2 \, \ln{\frac{\lVert u_{\varepsilon,\textnormal{in}}\rVert_{L^2}}{\lVert \gamma\rVert_{L^2}}} \right) v_\varepsilon.
\end{equation*}
The last term is totally harmless, as it can be removed by changing $v_\varepsilon$ into $v_\varepsilon \, e^{- i \frac{\theta}{\varepsilon}}$ where
\begin{equation*}
    \theta = \theta (t) := \lambda d \int_0^t \ln{\tau(s)} \, ds - 2 \lambda t \ln{\frac{\lVert u_{\varepsilon,\textnormal{in}}\rVert_{L^2}}{\lVert \gamma\rVert_{L^2}}}.
\end{equation*}
Thus, we obtain the system
\begin{equation}
    i \varepsilon \, \partial_t v_\varepsilon + \frac{\varepsilon^2}{2 \tau(t)^2} \Delta_y v_\varepsilon = \lambda v_\varepsilon \ln{ \left\lvert \frac{v_\varepsilon}{\gamma} \right\rvert^2}, \qquad \qquad
    v_\varepsilon (0, x) = \frac{\lVert \gamma\rVert_{L^2}}{\lVert u_{\varepsilon,\textnormal{in}}\rVert_{L^2}} \, u_{\varepsilon,\textnormal{in}}.  \label{mod_log_nls}
\end{equation}


We define the modified total energy and kinetic energy with semiclassical constant and the relative entropy:
\begin{gather*}
    \mathcal{E}^\varepsilon_{\text{kin}} (t) := \frac{\varepsilon^2}{2 \, \tau(t)^2} \lVert \nabla v_\varepsilon \rVert_{L^2}^2, \qquad
    \mathcal{E}_{\text{ent}}^\varepsilon (t) := \int_{\mathbb{R}^d} |v_\varepsilon(t,y)|^2 \ln{\left\lvert \frac{v_\varepsilon (t,y)}{\gamma (y)} \right\rvert^2} \diff y, \\
    \mathcal{E}^\varepsilon := \mathcal{E}^\varepsilon_{\text{kin}} + \lambda \, \mathcal{E}_{\text{ent}}^\varepsilon.
\end{gather*}
Then there holds
\begin{equation}
    \Dot{\mathcal{E}^\varepsilon} = - 2 \frac{\Dot{\tau} (t)}{\tau (t)} \mathcal{E}_{\text{kin}}^\varepsilon, \label{mod_energy_t}
\end{equation}
Following the ideas of \cite{carlesgallagher}, we should now have estimates which should depend only on $\mathcal{E}^\varepsilon (0)$. However, writing
\begin{equation*}
    \mathcal{E}_{\text{ent}}^\varepsilon (t) = \int_{\mathbb{R}^d} |v_\varepsilon(t,y)|^2 \ln{\left\lvert v_\varepsilon (t,y) \right\rvert^2} \diff y + \int_{\mathbb{R}^d} |y|^2 \, |v_\varepsilon(t,y)|^2 \diff y,
\end{equation*}
it is obvious that $\mathcal{E}^\varepsilon \leq \tilde{E}_\varepsilon (., v_\varepsilon)$ and in particular $\mathcal{E}^\varepsilon (0) \leq \tilde{E}_\varepsilon^0 (v_{\varepsilon,in})$. Actually, if we separate the positive and negative parts of the entropy in the modified total energy thanks to
\begin{equation*}
    \int |v_\varepsilon| \ln |v_\varepsilon|^2 = \int_{|v_\varepsilon|>1} |v_\varepsilon|^2 \ln |v_\varepsilon|^2 + \int_{|v_\varepsilon|\leq1} |v_\varepsilon|^2 \ln |v_\varepsilon|^2,
\end{equation*}
we can define
\begin{align*}
    \mathcal{E}^\varepsilon_+ &:= \mathcal{E}^\varepsilon_{\text{kin}} + \lambda \int_{|v_\varepsilon|>1} |v_\varepsilon|^2 \ln |v_\varepsilon|^2 + \lambda \int |y|^2 \, |v_\varepsilon|^2 \geq 0, \\
    \mathcal{E}_-^\varepsilon &:= - \lambda \int_{|v_\varepsilon|\leq1} |v_\varepsilon| \ln |v_\varepsilon|^2 \geq 0.
\end{align*}
Then, with the definition of $\tilde{E}_\varepsilon$ in \eqref{def_tilde_E_eps_0}, it is clear that
\begin{gather*}
    \tilde{E}_\varepsilon \approx \mathcal{E}^\varepsilon_+ + \mathcal{E}^\varepsilon_- \geq \mathcal{E}^\varepsilon_+ \geq \mathcal{E}^\varepsilon = \mathcal{E}^\varepsilon_+ - \mathcal{E}^\varepsilon_-.
\end{gather*}

We already know that $\mathcal{E}^\varepsilon$ is bounded since it is decreasing and non-negative thanks to the Csisz\'ar-Kullback inequality, which reads (see \cite[Theorem~8.2.7]{soblog})
\begin{equation*}
    \mathcal{E}_{\text{ent}}^\varepsilon (t) \geq \frac{1}{2 \lVert \gamma^2\rVert_{L^1(\mathbb{R}^d)}} \left\lVert \, |v_\varepsilon|^2 (t) - \gamma^2 \right\rVert_{L^1(\mathbb{R}^d)}^2.
\end{equation*}
Actually, the following lemma states not only the boundedness of $\tilde{E}_\varepsilon$ but also some integrability property for the $\dot{H}^1$ norm, which are \eqref{enestschr_th} and \eqref{enresschr_th}.

\begin{lem}
\label{lem_en_est_schr}
With the previous notations, there exists a continuous non-decreasing function 
$C: [0, \infty) \rightarrow [0, \infty)$ depending only on $\lambda$ and $d$ such that for all $t \geq 0$ and for all $\varepsilon > 0$,
\begin{gather*}
    \tilde{E}_\varepsilon \big(t,v(t)\big)
    \leq C \hspace{-0.5mm} \left(\tilde{E}_\varepsilon^0 (v_{\varepsilon,in}) \right), \qquad
    \int_0^\infty \frac{\varepsilon^2 \, \Dot{\tau} (t)}{\tau^3 (t)} \lVert \nabla_y v_\varepsilon(t)\rVert_{L^2(\mathbb{R}^d)}^2 \diff t \leq C \hspace{-0.5mm} \left(\tilde{E}_\varepsilon^0 (v_{\varepsilon,in}) \right). 
\end{gather*}
\end{lem}

\begin{proof}
Using the fact that the modified energy is non-increasing, we have
\begin{align*}
    \mathcal{E}^\varepsilon_+ \leq \mathcal{E}^\varepsilon(0) + \mathcal{E}_-^\varepsilon.
\end{align*}
The last term can be controlled by
\begin{equation*}
    \mathcal{E}_-^\varepsilon \leq C_\delta \int_{\mathbb{R}^d} |v_\varepsilon|^{2-\delta},
\end{equation*}
for all $\delta \in (0,2)$. Moreover, we have the estimate
\begin{equation*}
    \int_{\mathbb{R}^d} |v_\varepsilon|^{2-\delta} \leq C_\delta \, \lVert v_\varepsilon \rVert_{L^2}^{2- \left( 1+\frac{d}{2} \right) \delta} \lVert y v_\varepsilon \rVert_{L^2}^{\frac{d \delta}{2}} = C_\delta \, \lVert \gamma \rVert_{L^2}^{2- \left( 1+\frac{d}{2} \right) \delta} \lVert y v_\varepsilon \rVert_{L^2}^{\frac{d \delta}{2}},
\end{equation*}
as soon as $0 < \delta < \frac{2}{d + 2}$ in the same way as \eqref{est_subcritical_power_to_mom}. Taking (for example) $\delta = \frac{1}{d + 2}$, this implies
\begin{equation*}
    \mathcal{E}_-^\varepsilon \leq C_d \, (\mathcal{E}_+^\varepsilon)^{\frac{d}{4(d + 2)}}, \qquad \text{and then} \qquad
    \mathcal{E}^\varepsilon_+ \leq \tilde{E}_\varepsilon^0 ( v_{\varepsilon,in} ) + C_d \, (\mathcal{E}^\varepsilon_+)^{\frac{d}{4(d + 2)}}.
\end{equation*}
Thus $\mathcal{E}^\varepsilon_+ \leq \Tilde{C} \hspace{-0.5mm} \left( \tilde{E}_\varepsilon^0 ( v_{\varepsilon,in} ) \right)$ for some continuous and non-decreasing function $\Tilde{C} : [0, \infty) \rightarrow [0, \infty)$ (independent of $t$ and $\varepsilon$) since $\frac{d}{4 (d + 2)} < 1$. There also holds $\mathcal{E}_-^\varepsilon \leq C_d \left( \Tilde{C} \hspace{-0.5mm} \left( \tilde{E}_\varepsilon^0 ( v_{\varepsilon,in} ) \right) \right)^\frac{d}{4 (d + 2)}$, and then \eqref{enestschr_th} for $C := \tilde{C} + C_d \, \Tilde{C}^\frac{d}{4 (d + 2)}$.

Last, \eqref{enresschr_th} follows from \eqref{mod_energy_t} and the fact that $\mathcal{E}^\varepsilon (t)$ is bounded uniformly in $t\geq 0$ by $C \left( \tilde{E}_\varepsilon^0 ( v_{\varepsilon,in} ) \right)$.
\end{proof}


\begin{rem}
The Csisz\'ar-Kullback inequality shows that, if we had $\mathcal{E}_\text{ent}^\varepsilon (t) \underset{t \rightarrow \infty}{\longrightarrow} 0$ (for example, $\mathcal{E}^\varepsilon (t) \underset{t \rightarrow \infty}{\longrightarrow} 0$), we would have $\left\lVert \, |v_\varepsilon|^2 (t) - \gamma^2 \right\rVert_{L^1(\mathbb{R}^d)}^2 \underset{t \rightarrow \infty}{\longrightarrow} 0$ and then strong convergence would follow, but we cannot reach this conclusion in the general case.
\end{rem}

\subsubsection{Convergence of some quadratic quantities} We now prove \eqref{momestschr_1_th}-\eqref{momestschr_2_th}, as stated in the next lemma.

\begin{lem}
\label{lem_conv_quadr_schr}
Under the assumptions of Theorem \ref{main_th_log_nls_eps}, the first two momenta converge: for all $t\geq1$ and all $\varepsilon>0$,
\begin{gather*}
    \int_{\mathbb{R}^d} y \, |v_\varepsilon (t,y)|^2 \diff y = \frac{1}{\tau (t)} \frac{\lVert \gamma^2\rVert_{L^1}}{\lVert u_{\varepsilon,\textnormal{in}}\rVert_{L^2}^2} (I_{1,0}^\varepsilon \, t + I_{2,0}^\varepsilon), \\
    \left | \int_{\mathbb{R}^d} |y|^2 \, |v_\varepsilon (t,y)|^2 \diff y - \int_{\mathbb{R}^d} |y|^2 \, \gamma^2 (y) \diff y \right\rvert \leq C \hspace{-0.5mm} \left(\tilde{E}_\varepsilon^0 (v_{\varepsilon,in}) 
    \right) \frac{\Dot{\tau}(t) + 1}{\Dot{\tau} (t)^2},
\end{gather*}
where
\begin{equation*}
    I_{1,0}^\varepsilon = \Im \varepsilon \int_{\mathbb{R}^d} \overline{u_{\varepsilon,\textnormal{in}}} \, \nabla u_{\varepsilon,\textnormal{in}} \diff y, \qquad I_{2,0}^\varepsilon = \int_{\mathbb{R}^d} y \, |u_{\varepsilon,\textnormal{in}}|^2 \diff y.
\end{equation*}
\end{lem}

\begin{proof}
Introducing
\begin{equation*}
    I_1^\varepsilon (t) := \Im \varepsilon \int_{\mathbb{R}^d} \overline{v_\varepsilon} (t,y) \, \nabla v_\varepsilon (t,y) \diff y, \qquad I_2^\varepsilon(t) := \int_{\mathbb{R}^d} y \, |v_\varepsilon (t,y)|^2 \diff y, \qquad \Tilde{I_2^\varepsilon} (t) := \tau (t) \, I_2^\varepsilon (t),
\end{equation*}
we compute
\begin{gather*}
    \Dot{I_1^\varepsilon} = -2 \lambda I_2^\varepsilon, \qquad \Dot{I_2^\varepsilon} = \frac{1}{\tau^2 (t)} I_1^\varepsilon, \qquad \Ddot{\Tilde{I_2^\varepsilon}} = 0.
\end{gather*}
Therefore \eqref{momestschr_1_th} easily follows from simple computations.
We now go back to the conservation of energy for $u_\varepsilon$,
\begin{equation*}
    \frac{\varepsilon^2}{2} \lVert \nabla u_\varepsilon (t) \rVert^2_{L^2} + \lambda \int_{\mathbb{R}^d} |u_\varepsilon (t,x)|^2 \ln |u_\varepsilon (t,x)|^2 \diff x = \frac{\varepsilon^2}{2} \left\lVert \nabla u_{\varepsilon,\textnormal{in}} \right\rVert_{L^2}^2 + \lambda \int_{\mathbb{R}^d} |u_{\varepsilon,\textnormal{in}}|^2 \ln |u_{\varepsilon,\textnormal{in}}|^2,
\end{equation*}
and translate this property into estimates on $v_\varepsilon$
\begin{multline*}
    \mathcal{E}_{\text{kin}} + \frac{\Dot{\tau}^2}{2} \int |y|^2 \, |v_\varepsilon|^2 - \varepsilon \frac{\Dot{\tau}}{\tau} \Im \int v_\varepsilon (t,y) \, y \, \overline{\nabla v_\varepsilon} (t,y) \diff y + \lambda \int |v_\varepsilon|^2 \ln |v_\varepsilon|^2 - \lambda d \, \lVert \gamma^2\rVert_{L^1} \ln \tau \\ = \frac{\varepsilon^2}{2} \left\lVert \nabla v_{\varepsilon, in} \right\rVert_{L^2}^2 + \lambda \int_{\mathbb{R}^d} |v_{\varepsilon,in}|^2 \ln |v_{\varepsilon,in}|^2,
\end{multline*}
Therefore, we obtain thanks to the previous estimate \eqref{enestschr_th} (along with a Cauchy-Schwarz inequality)
\begin{align*}
    \left\lvert \frac{\Dot{\tau}^2}{2} \int |y|^2 \, |v_\varepsilon|^2 - \lambda d \, \lVert \gamma^2\rVert_{L^1} \ln \tau \right\rvert &\leq \begin{multlined}[t][12cm] \left\lvert \varepsilon \frac{\Dot{\tau}}{\tau} \Im \int v_\varepsilon (t,y) \, y \, \overline{\nabla v_\varepsilon} (t,y) \diff y \right\rvert  \\ + \left\lvert \frac{\varepsilon^2}{2} \left\lVert \nabla v_{\varepsilon, in} \right\rVert_{L^2}^2 + \lambda \int_{\mathbb{R}^d} |v_{\varepsilon,in}|^2 \ln |v_{\varepsilon,in}|^2 - \lambda \int |v_\varepsilon|^2 \ln |v_\varepsilon|^2 \right\rvert
    \end{multlined} \\
    &\leq \dot{\tau} (t) \, \left\lVert y \, v_\varepsilon (t) \right\rVert_{L^2}  \frac{\varepsilon}{\tau (t)} \lVert \nabla v_\varepsilon (t) \rVert_{L^2} + C_0 \, \left[ \tilde{E}_\varepsilon^0 (v_{\varepsilon,in}) + C \left(\tilde{E}_\varepsilon^0 (v_{\varepsilon,in}) \right) \right] \\
    &\leq C \hspace{-0.5mm} \left(\tilde{E}_\varepsilon^0 (v_{\varepsilon,in}) \right) (\Dot{\tau}(t) + 1).
\end{align*}
Multiplying \eqref{def_tau} by $\dot{\tau}$ and integrating yields
\begin{equation*}
    \frac{\Dot{\tau}^2}{2} = 2 \lambda \ln \tau,
\end{equation*}
which gives in the above inequality for all $t>1$
\begin{equation*}
    \left\lvert \int |y|^2 \, |v_\varepsilon|^2 - \frac{d}{2} \lVert \gamma^2\rVert_{L^1} \right\rvert \leq C \hspace{-0.5mm} \left(\tilde{E}_\varepsilon^0 (v_{\varepsilon,in}) \right) \frac{\Dot{\tau} (t) + 1}{\Dot{\tau}^2 (t)},
\end{equation*}
and then we can conclude thanks to the identity $\frac{d}{2} \lVert \gamma^2\rVert_{L^1} = \int |y|^2 \, \gamma^2 (y) \diff y$.
\end{proof}

\subsubsection{Equations on quadratic observables}

Finally, we get two equations involving the density and the density of angular momentum defined by
\begin{gather*}
    \rho_\varepsilon := |v_\varepsilon|^2, \qquad
    J_\varepsilon := \Im ( \varepsilon \, \overline{v_\varepsilon} \, \nabla v_\varepsilon ).
\end{gather*}
They satisfy in $\mathcal{D}' ((0, \infty) \times \mathbb{R}^d)$
\begin{gather}
    \partial_t \rho_\varepsilon + \frac{1}{\tau^2 (t)} \nabla \cdot J_\varepsilon = 0, \qquad
    \partial_t J_\varepsilon + \lambda \, \nabla \rho_\varepsilon + 2 \lambda \, y \, \rho_\varepsilon = \frac{\varepsilon^2}{4 \, \tau^2 (t)} \Delta \nabla \rho_\varepsilon - \frac{\varepsilon^2}{\tau^2(t)} \nabla \cdot ( \Re ( \nabla v_\varepsilon \otimes \overline{\nabla v_\varepsilon} ) ).
    \label{1steq_schr}
\end{gather}

\begin{rem}
In the same way as in \cite{carlesgallagher}, we can already conclude the weak convergence (in $L^1$) to $\gamma^2$ of $\rho_\varepsilon = | v_\varepsilon |^2$.
\end{rem}

\begin{rem} \label{rem_proof_eq_log_nls}
The three most important equations are given by \eqref{mod_energy_t} and \eqref{1steq_schr}. Even though we derive them from the equation \eqref{mod_log_nls} on $v_\varepsilon$ in the same way as in \cite{carlesgallagher}, it could have been directly derived from some equations for $u_\varepsilon$: the conservation of the mass and the energy and some identities for $|u_\varepsilon|^2$ and $\Im ( \varepsilon \, \overline{u_\varepsilon} \, \nabla u_\varepsilon )$ similar to \eqref{1steq_schr} (and some other estimates which arise from them, like the conservation of the angular momentum, the variation of the second momentum and the variation of $\int_{\mathbb{R}^d} x \Im ( \varepsilon \, \overline{u_\varepsilon} \, \nabla u_\varepsilon ) \diff x$):
\begin{equation}
    \begin{gathered}
        \partial_t ( |u_\varepsilon|)^2 + \nabla \cdot \Im ( \varepsilon \, \overline{u_\varepsilon} \, \nabla u_\varepsilon ) = 0, \\
        \partial_t \left( \Im ( \varepsilon \, \overline{u_\varepsilon} \, \nabla u_\varepsilon ) \right) + \lambda \, \nabla ( |u_\varepsilon|^2 ) = \frac{\varepsilon^2}{4} \Delta \nabla ( |u_\varepsilon|^2 ) - \varepsilon^2 \, \nabla \cdot ( \Re ( \nabla u_\varepsilon \otimes \overline{\nabla u_\varepsilon} ) ).
    \end{gathered}
        \label{eq_schr_for_u_eps}
\end{equation}
This is an important remark in view of Section \ref{section_KIE}.
\end{rem}

\subsection{Proof of Corollary \ref{cor_sob_norm_growth}}

Again, this proof is extremely similar to that in \cite{carlesgallagher}. In the energy for $u_\varepsilon$, write the potential energy in terms of $v_\varepsilon$.
\begin{align*}
    \int_{\mathbb{R}^d} |u_\varepsilon (t)|^2 \ln |u_\varepsilon (t)|^2 &= - d \left( \ln \tau (t) + \ln \left( \frac{\lVert u_{\varepsilon,\textnormal{in}}\rVert_{L^2}^2}{\lVert \gamma^2\rVert_{L^1}} \right) \right) \lVert u_{\varepsilon,\textnormal{in}}\rVert_{L^2}^2 + \frac{\lVert u_{\varepsilon,\textnormal{in}}\rVert_{L^2}^2}{\lVert \gamma^2\rVert_{L^1}} \int |v_\varepsilon (t)|^2 \ln |v_\varepsilon (t)|^2 \\
    &= - d \lVert u_{\varepsilon,\textnormal{in}}\rVert_{L^2}^2 \ln \tau (t) + \textnormal{O} (1).
\end{align*}
The conservation of the energy for $u_\varepsilon$ yields
\begin{equation*}
    \varepsilon^2 \, \lVert \nabla u_\varepsilon \rVert_{L^2}^2 \underset{t \rightarrow \infty}{\sim} 2 \lambda d \lVert u_{\varepsilon,\textnormal{in}}\rVert_{L^2}^2 \ln \tau (t).
\end{equation*}
Now fix $0<\delta<1$. By interpolation, we readily have
\begin{equation*}
    \varepsilon^\delta \lVert u_\varepsilon (t) \rVert_{\dot{H}^\delta} \lesssim \lVert u_\varepsilon \rVert_{L^2}^{1-\delta} \, ( \varepsilon^2 \, \lVert u_\varepsilon (t) \rVert_{\dot{H}^1}^\delta \lesssim (\ln t)^\frac{\delta}{2}.
\end{equation*}
For the other inequality, we recall the lemma used in \cite{carlesgallagher} without semiclassical constant. However, in our context, it is better to recall it with semiclassical constant.
\begin{lem}[{\cite[Lemma~1.5.]{Alazard_Carles}}]
    There exists a constant $C > 0$ such that for all $\varepsilon \in (0,1]$, for all $\delta \in [0,1]$, for all $u \in H^1 (\mathbb{R}^d)$ and for all $w \in \dot{W}^{1,\infty} (\mathbb{R}^d)$,
    \begin{equation*}
        \lVert \, |w|^\delta u \rVert_{L^2} \leq \varepsilon^\delta \lVert u \rVert_{\dot{H}^\delta} + \lVert u \rVert_{L^2}^{1-\delta} \, \lVert (\varepsilon \nabla - iw) u \rVert_{L^2}^{\delta} + C \, \varepsilon^\frac{\delta}{2} \, (1 + \lVert \nabla w \rVert_{L^\infty}) \, \lVert u \rVert_{L^2}.
    \end{equation*}
\end{lem}

Applying this lemma with $u_\varepsilon (t)$ and
\begin{equation*}
    w (t,x) := \frac{\dot{\tau} (t)}{\tau (t)} \, x,
\end{equation*}
we get for all $t \geq 0$
\begin{equation*}
    \dot{\tau} (t)^\delta \, \lVert \, |y|^\delta \, v_\varepsilon (t) \rVert_{L^2} \leq \lVert u_\varepsilon \rVert_{\dot{H}^\delta} + \left\lVert \frac{\varepsilon}{\tau (t)} \nabla v_\varepsilon (t) \right\rVert_{L^2}^\delta \, \frac{\lVert u_{\varepsilon, \textnormal{in}} \rVert_{L^2}}{\lVert \gamma \rVert_{L^2}^\delta} + C \, \varepsilon^\frac{\delta}{2} \, \left( 1 + \frac{\dot{\tau} (t)}{\tau (t)} \right) \lVert u_{\varepsilon, \textnormal{in}} \rVert_{L^2}.
\end{equation*}
The result readily follows: all the terms of the right hand side are bounded but the first one, and the behaviour of the left hand side is given by the convergence in Wasserstein distance $W_2$ which implies (since $\delta \in (0,1)$)
\begin{equation*}
    \int_{\mathbb{R}^d} |y|^{2\delta} \, |v_\varepsilon (t,y)|^2 \diff y \underset{t \rightarrow \infty}{\longrightarrow} \int_{\mathbb{R}^d} |y|^{2\delta} \, \gamma^2 (y) \diff y.
\end{equation*}

\subsection{First part of the proof of Theorem \ref{main_th_wigner}}

From now on, $C_0$ denotes a positive constant (which may change from line to line) independent of $t$ and $\varepsilon$ and we assume that \eqref{in_data_assumption} is satisfied. 

\subsubsection{Convergence of the Wigner Transforms and first properties}

First, we proved that $(v_\varepsilon)_{\varepsilon > 0}$ satisfies \eqref{assumption_3} thanks to \eqref{enestschr_th}, hence we can apply Lemma \ref{lem_wigner} for $(v_\varepsilon)_{\varepsilon > 0}$ and also for $(u_\varepsilon)_{\varepsilon > 0}$ for all $T>0$ because $(u_\varepsilon)_{\varepsilon > 0}$ 
satisfies \eqref{assumption_3} thanks to \eqref{rescaling} and the first assumption of \eqref{in_data_assumption}.
By an argument of diagonal extraction, it leads to a subsequence (still denoted $\varepsilon$) and two measures $W = W (t, x, \xi)$ and $\Tilde{W} = \Tilde{W} (t, y, \eta)$ in $L^\infty ((0, \infty), \mathcal{M}(\mathbb{R}^d \times \mathbb{R}^d))$ such that for every $p \in [1, \infty)$
\begin{gather}
    W_{\varepsilon} \underset{\varepsilon \rightarrow 0}{\rightharpoonup} W, \qquad \Tilde{W}_{\varepsilon} \underset{\varepsilon \rightarrow 0}{\rightharpoonup} \Tilde{W}, \qquad \text{in } L^p_\text{loc} ((0, \infty), \mathcal{A}'), \notag \\
    \rho_\varepsilon = | v_\varepsilon |^2 \underset{\varepsilon \rightarrow 0}{\longrightarrow} \tilde{\rho} := \int_{\mathbb{R}^d} \tilde{W} (., ., \diff \eta) \qquad \text{in } L^p_\text{loc} \left((0, \infty), \mathcal{M} (\mathbb{R}^d)\right), \notag \\
    \iint_{\mathbb{R}^d \times \mathbb{R}^d} |y|^2 \, \tilde{W} (t, \diff y, \diff \eta) \leq C_0, \qquad
    \iint_{\mathbb{R}^d \times \mathbb{R}^d} |\eta|^2 \, \tilde{W} (t, \diff y, \diff \eta) \leq C_0 \, \tau(t)^2. \label{est_2nd_mom_proved}
\end{gather}
But \eqref{enestschr_th} also gives that $\rho_\varepsilon$ is uniformly bounded in $L^\infty \left( (0,\infty), L \log L (\mathbb{R}^d) \right)$. Therefore, $\tilde{\rho} \in L^\infty \left( (0,\infty), L \log L (\mathbb{R}^d) \right)$.
It remains to prove that $\Tilde{\rho} \in \mathcal{C} ([0, \infty), \mathcal{P}_1 (\mathbb{R}^d))$. Come back to the equation for $\partial_t \rho_\varepsilon$ in \eqref{1steq_schr}:
\begin{equation*}
    \partial_t \rho_\varepsilon + \frac{1}{\tau^2 (t)} \nabla \cdot J_\varepsilon = 0 \qquad \text{in } \mathcal{D}',
\end{equation*}
where we recall $J_\varepsilon = \Im ( \varepsilon \, \overline{v_\varepsilon} \nabla v_\varepsilon)$. We also recall that $\frac{1}{\tau (t)} J_\varepsilon$ is bounded in $L^\infty((0,\infty), L^1_1 (\mathbb{R}^d))$ uniformly in $\varepsilon > 0$ thanks to \eqref{enestschr_th} and a Cauchy-Schwarz inequality. Therefore, thanks to Kantorovich-Rubinstein duality, $\pi^{-\frac{d}{2}} \, \rho_\varepsilon$ is equicontinuous with values in $\mathcal{P}_1 (\mathbb{R}^d)$ (endowed with the Wasserstein metric $\mathcal{W}_1$). Moreover, since $\rho_\varepsilon$ is bounded in $L^\infty \left( (0, \infty), L^1_2 \cap L \log L (\mathbb{R}^d) \right)$, de la Vallée-Poussin and Dunford-Pettis theorems yield the compactness of $\{ \pi^{-\frac{d}{2}} \, \rho_\varepsilon (t), \varepsilon > 0 \}$ in $\mathcal{P}_1 (\mathbb{R}^d)$ for all $t\geq0$.
Hence, by Ascoli theorem, $\{ \pi^{-\frac{d}{2}} \, \rho_\varepsilon\}$ is a compact set in $\mathcal{C} \left([0, T], \mathcal{P}_1 (\mathbb{R}^d) \right)$ for all $T>0$. Thus we get not only $\pi^{-\frac{d}{2}} \, \Tilde{\rho} \in \mathcal{C}([0, \infty), \mathcal{P}_1 (\mathbb{R}^d))$ but also
\begin{gather*}
    \pi^{-\frac{d}{2}} \, \rho_\varepsilon \underset{\varepsilon \rightarrow 0}{\longrightarrow} \pi^{-\frac{d}{2}} \, \tilde{\rho} \qquad \text{in } \mathcal{C} \left([0, T], \mathcal{P}_1 (\mathbb{R}^d) \right) \text{ for all } T > 0. \label{uniform_conv_density}
\end{gather*}
Moreover, the identity $\tilde{\rho}(t, \mathbb{R}^d) = \pi^\frac{d}{2}$, satisfied for a.e. $t \in (0, \infty)$ thanks to Lemma \ref{lem_wigner}, is actually satisfied for \textit{all} $t \in [0, \infty)$.


Lastly, an easy computation leads to the relation \eqref{wigner_rel}, substituting $u_\varepsilon$ by $v_\varepsilon$ thanks to \eqref{rescaling} and performing a simple change of variable.
In this relation \eqref{wigner_rel}, two terms already converges: $\tilde{W}_\varepsilon$ converges to the non-null measure $\tilde{W}$ and $W_\varepsilon$ converges to the measure $W$. Therefore, thanks to this relation, it is easy to prove that $\lVert u_{\varepsilon,\textnormal{in}}\rVert_{L^2}^2$ converges. Thus we can pass to the limit in the relation between the two Wigner Transforms.

\subsubsection{Estimates on the momenta}

We already proved the estimates \eqref{est_2nd_mom_proved}. Moreover, in the same way, \eqref{enresschr_th} can be translated into a property on the Husimi Transform thanks to \eqref{whepsmom1}:
\begin{equation*}
    \int_0^\infty \frac{\Dot{\tau} (t)}{\tau^3 (t)} \iint_{\mathbb{R}^d \times \mathbb{R}^d} |\eta|^2 \, \Tilde{W}^H_{\varepsilon} (t,\diff y,\diff \eta) \diff t \leq C_0.
\end{equation*}
A slight modification of the proof of Lemma \ref{lem_wigner} shows that this estimate still holds after passing to the limit, so that we get \eqref{est_WM_th}.
For \eqref{1st_mom_density_WM_th}, we proved that $\rho_\varepsilon$ converges locally uniformly in time to $\tilde{\rho}$ and has a second momentum bounded uniformly in $t > 0$ and $\varepsilon > 0$. Therefore, a usual argument shows that
\begin{equation*}
    \int_{\mathbb{R}^d} y \, \rho_{\varepsilon} (t,y) \diff y \underset{\varepsilon \rightarrow 0}{\longrightarrow} \int_{\mathbb{R}^d} y \, \tilde{\rho} (t,y) \diff y \qquad \text{locally uniformly in } t.
\end{equation*}
However, the term on the left-hand side has already been computed in \eqref{momestschr_1_th}, hence the affine function $\frac{\lVert \gamma^2\rVert_{L^1}}{\lVert u_{\varepsilon,\textnormal{in}} \rVert_{L^2}^2} \, (I_{1,0}^\varepsilon \, t + I_{2,0}^\varepsilon)$ converges locally uniformly in $t$. Thus, we conclude that both $\frac{\lVert \gamma^2\rVert_{L^1}}{\lVert u_{\varepsilon,\textnormal{in}} \rVert_{L^2}^2} \, I_{1,0}^\varepsilon$ and $\frac{\lVert \gamma^2\rVert_{L^1}}{\lVert u_{\varepsilon,\textnormal{in}} \rVert_{L^2}^2} \, I_{2,0}^\varepsilon$ converge to some $C_1$ and $C_2$, and then we obtain \eqref{1st_mom_density_WM_th}. This completes the first part of the proof.

\subsection{Convergence of some other quantities}

Actually, we would like to pass to the limit in the two identities in \eqref{1steq_schr}. For this, there are still two quantities which should converge: $J_\varepsilon$ and $\varepsilon^2 \, \Re \left( \nabla v_\varepsilon \otimes \overline{\nabla v_\varepsilon} \right)$. First, we recall the estimates found for those two quantities thanks to \eqref{enestschr_th} and \eqref{enresschr_th} (up to a Cauchy-Schwarz inequality for $J_\varepsilon$)
\begin{gather}
    \frac{1}{\tau (t)} \int_{\mathbb{R}^d} (1 + |y|) \, |J_\varepsilon (t,y)| \diff y \leq C_0 \quad \text{for all } t \geq 0 \text{ and } \varepsilon > 0, \notag \\
    \int_0^\infty \frac{\Dot{\tau} (t)}{\tau(t)} \, \left( \frac{1}{\tau (t)} \int_{\mathbb{R}^d} (1 + |y|) \, |J_\varepsilon (t,y)| \diff y \right)^2 \diff t \leq C_0, \label{est_J_eps_int_time}  \\
    \frac{\varepsilon^2}{\tau^2 (t)} \left\lVert \Re ( \nabla v_\varepsilon (t) \otimes \overline{\nabla v_\varepsilon (t)} ) \right\rVert_{L^1} \leq C_0 \qquad \text{for all } t \geq 0 \text{ and } \varepsilon > 0, \notag \\
    \int_0^\infty \frac{\Dot{\tau}(t)}{\tau(t)} \, \int_{\mathbb{R}^d} \left\lvert \frac{\varepsilon^2}{\tau^2 (t)} \Re ( \nabla v_\varepsilon (t) \otimes \overline{\nabla v_\varepsilon (t)} ) \right\rvert \diff y \diff t \leq C_0. \label{est_nabla_v_eps_int_time}
\end{gather}
Moreover, $J_\varepsilon$ and $\varepsilon^2 \, \Re \left( \nabla v_\varepsilon \otimes \overline{\nabla v_\varepsilon} \right)$ are related to the Husimi Transform respectively through \eqref{wheps1stmom} and \eqref{whepsmom2}.
An analysis similar to that for the density and for the second momentum for the Wigner Measure in the proof of Lemma \ref{lem_wigner} shows that for all $p > 1$:
\begin{gather*}
    \frac{1}{\tau (t)} \, J_\varepsilon \underset{\varepsilon \rightarrow 0}{\longrightarrow} \mu_0 := \frac{1}{\tau (t)} \int_{\mathbb{R}^d} \eta \, \tilde{W} (t, y, \diff \eta) \quad \text{in } L^p_\text{loc} ((0, \infty), \mathcal{M}^s (\mathbb{R}^d)^d), \\
    \frac{\varepsilon^2}{\tau^2 (t)} \, \Re \left( \nabla v_\varepsilon \otimes \overline{\nabla v_\varepsilon} \right) \underset{\varepsilon \rightarrow 0}{\rightharpoonup} \nu_0 := \frac{1}{\tau^2 (t)} \int_{\mathbb{R}^d} \eta \otimes \eta \, \tilde{W} (t, y, \diff \eta) \quad \text{in } L^p_\text{loc} \left( (0, \infty), \mathcal{M}^s \left( \mathbb{R}^d \right)^{d \times d} \right),
\end{gather*}
where $\mathcal{M}^s (\mathbb{R}^d)$ is the set of finite signed measure on $\mathbb{R}^d$. 
In particular,
\begin{gather}
    \underset{t > 0}{\supess} \int_{\mathbb{R}^d} (1 + |y|) \, |\mu_0| (t, \diff y) +
    \int_0^\infty \frac{\Dot{\tau} (t)}{\tau(t)} \left( \int_{\mathbb{R}^d} (1 + |y|) \, |\mu_0| (t, \diff y) \right)^2 \diff t \leq C_0, \label{est_mu} \\
    \underset{t > 0}{\supess} |\nu_0| (t, \mathbb{R}^d) +
    \int_0^\infty \frac{\Dot{\tau}(t)}{\tau(t)} \, |\nu_0| (t, \mathbb{R}^d) \diff t \leq C_0, \label{est_nu}
\end{gather}
where $|\mu_0|$ (resp. $|\nu_0|$) is the absolute variation of $\mu_0$ (resp $\nu_0$), and \eqref{1steq_schr} becomes
\begin{gather}
    \partial_t \tilde{\rho} + \frac{1}{\tau (t)} \nabla \cdot \mu_0 = 0, \qquad
    \partial_t ( \tau (t) \, \mu_0 ) + \lambda \, \nabla \tilde{\rho} + 2 \lambda \, y \, \tilde{\rho} = - \nabla \cdot \nu_0, \qquad \text{ in } \mathcal{D}' ((0,\infty) \times \mathbb{R}^d).
    \label{eq_WM}
\end{gather}

\begin{rem} \label{rem_continuity_J_0}
The latter equation along with the estimates \eqref{est_WM_th}, \eqref{est_mu} and \eqref{est_nu} gives some new estimate for $\tau (t) \, \mu_0$, due to the fact that $\partial_t ( \tau (t) \, \mu_0 ) \in L^\infty \left( (0, \infty), W^{-1-\delta, 1} (\mathbb{R}^d) \right)$ for all $\delta>0$, uniformly in $\delta$:
\begin{equation*}
    \lVert \partial_t ( \tau (t)  \, \mu_0 ) \rVert_{L^\infty_t W^{-1-\delta,1}_y} \leq \lambda \lVert \tilde{\rho} \rVert_{L^\infty_t L^1_y} + 2 \lambda \lVert y \, \tilde{\rho} \rVert_{L^\infty_t L^1_y} + \underset{t > 0}{\supess} |\nu_0| (t, \mathbb{R}^d) \leq C_0.
\end{equation*}
In particular, $J_0 := \tau (t) \, \mu_0 \in \mathcal{C} \left( [0, \infty), W^{-1-\delta, 1} \cap \mathcal{M}^s (\mathbb{R}^d) \right)$ for all $\delta > 0$, so that $\mu_0$ is actually defined for \textit{all} $t \geq 0$. We can also derive that
\begin{equation*}
    \lVert J_0 (t_0 + t, y) - J_0 (t, y) \rVert_{W^{-1-\delta, 1}_y} \leq C_0 \, t_0, \qquad \forall t \geq 0, \, \forall t_0 \geq 0.
\end{equation*}
But $J_0 (t) \in \mathcal{M}^s (\mathbb{R}^d) \subset W^{-1,1} (\mathbb{R}^d)$ for all $t \geq 0$. Therefore, for all $t, t_0 \geq 0$, we can take the limit $\delta \rightarrow 0$ and we get
\begin{equation*}
    \lVert J_0 (t_0 + t, y) - J_0 (t, y) \rVert_{W^{-1, 1}_y} \leq C_0 \, t_0, \qquad \forall t \geq 0, \, \forall t_0 \geq 0.
\end{equation*}
In particular, this leads to:
\begin{equation*}
    \lVert \mu_0 (t) \rVert_{W^{-1, 1}} \leq C_0 \, \frac{1 + t}{\tau (t)}  \underset{t \rightarrow \infty}{\longrightarrow} 0, \qquad \forall t \geq 0.
\end{equation*}

\end{rem}

\section{Fokker-Planck equation and convergence rate in Wasserstein distance}
\label{section_FP}

\subsection{From Schrödinger to Fokker-Planck}

We define $\rho_0 := \tilde{\rho}$, $\mu_\varepsilon := \frac{1}{\tau (t)} \, J_\varepsilon$ and $\nu_\varepsilon := \frac{\varepsilon^2}{\tau^2 (t)} \, \Re \left( \nabla v_\varepsilon \otimes \overline{\nabla v_\varepsilon} \right)$ for all $\varepsilon > 0$, so that we can write \eqref{1steq_schr} and \eqref{eq_WM} in a single generalized system for all $\varepsilon \geq 0$ (which also holds in $\mathcal{D}' (\mathbb{R} \times \mathbb{R}^d)$):
\begin{gather}
    \partial_t \rho_\varepsilon + \frac{1}{\tau (t)} \nabla \cdot \mu_\varepsilon = 0, \qquad
    \partial_t (\tau (t) \mu_\varepsilon) + \lambda \, \nabla \rho_\varepsilon + 2 \lambda \, y \, \rho_\varepsilon = \frac{\varepsilon^2}{4 \, \tau^2 (t)} \Delta \nabla \rho_\varepsilon - \nabla \cdot \nu_\varepsilon.
    \label{generalized_eq_schr_WM}
\end{gather}
In a similar way as in \cite[Theorem~1.7.]{carlesgallagher}, combining those two equations leads to
\begin{equation*}
    \partial_t ( \tau^2 \, \partial_t \rho_\varepsilon ) = \lambda \, L \rho_\varepsilon - \frac{\varepsilon^2}{4 \, \tau^2 (t)} \Delta^2 \rho_\varepsilon + \nabla \cdot \left( \nabla \cdot \nu_\varepsilon \right),
\end{equation*}
where $L = \Delta + \nabla \cdot (2 y \, .)$ is a Fokker-Planck operator. On the other hand,
\begin{equation*}
    \partial_t ( \tau^2 \, \partial_t \rho_\varepsilon ) = \tau^2 \, \partial_t^2 \rho_\varepsilon + 2 \dot{\tau} \tau \, \partial_t \rho_\varepsilon.
\end{equation*}
Since $\tau^2 \ll (\dot{\tau} \tau)^2$, it is natural to change scales in time and define
\begin{equation}
    s = \int \frac{\lambda}{\dot{\tau} \tau} = \int \frac{\Ddot{\tau}}{2 \, \dot{\tau}} = \frac{1}{2} \ln \dot{\tau} (t). \label{change_time_var}
\end{equation}
From this we have (using the notation $f(t) = \check{f} (s(t))$ for the change of time variable)
\begin{equation}
    \partial_s \check{\rho}_\varepsilon - \frac{2 \lambda}{\check{\Dot{\tau}}^2} \partial_s \check{\rho}_\varepsilon + \frac{\lambda}{\check{\Dot{\tau}}^2} \partial_s^2 \check{\rho}_\varepsilon = L \check{\rho}_\varepsilon - \frac{\varepsilon^2}{4 \, \lambda \, \check{\tau}^2 (s)} \Delta^2 \check{\rho}_\varepsilon + \frac{1}{\lambda} \nabla \cdot \left( \nabla \cdot \check{\nu}_\varepsilon \right),
    \label{eqtrhos_schr}
\end{equation}

Discarding formally negligible terms leads to the Fokker-Planck equation without source terms
\begin{equation*}
    \partial_s \check{\rho}_\varepsilon = L \check{\rho}_\varepsilon,
\end{equation*}
for which it is well-known (see for instance \cite[Corollary~2.17.]{fokkerplank}) that in large times the solution
converges strongly in $L^1$ to an element of the kernel of $L$, hence a Gaussian. Notice
that the convergence is exponentially fast in $s$ variable, so coming back into $t$ variable
produces a logarithmic decay (which is exactly what we are expecting) due to the estimate
\begin{equation*}
    s = \frac{1}{4} \ln \ln t + \text{o}(1) \qquad \text{as } t \rightarrow \infty.
\end{equation*}
In particular, translating the properties of convergence \eqref{conv_wass_dist} and \eqref{conv_wass_dist_wigner} in terms of $s$ leads to
\begin{equation}
    \mathcal{W}_1 \left( \pi^{-\frac{d}{2}} \, \check{\rho}_\varepsilon (s), \pi^{-\frac{d}{2}} \, \gamma^2 \right) \leq C_0 \, e^{-2s}, \qquad \forall s > 0. \label{conv_FP_wass_s_var}
\end{equation}
It is worth mentioning that exponential convergence also occurs in 2-Wasserstein distance for Fokker-Planck equations without source terms (see for instance \cite{bolley-gentil-guillin_wassFK, villani2008optimal}). In particular, for our particular Fokker-Planck operator, such a result reads as follows:
\begin{lem}[\cite{bolley-gentil-guillin_wassFK, villani2008optimal}] \label{lem_wass_conv_FP_without_source_term}
    For any $f_{\textnormal{in}} \in \mathcal{P}_2 \cap L \log L \, (\mathbb{R}^d)$, the solution $f$ to $\partial_t f = L f$ with $f(0) = f_{\textnormal{in}}$ satisfies:
    \begin{equation*}
        \mathcal{W}_2 \left( f(t), \pi^{-\frac{d}{2}} \, \gamma^2 \right) \leq e^{-2t} \, \mathcal{W}_2 \left( f_{\textnormal{in}}, \pi^{-\frac{d}{2}} \, \gamma^2 \right), \qquad \forall t \geq 0.
    \end{equation*}
\end{lem}
Therefore, the $s$ variable must be better suited for our study. 
The following lemma computes $\check{\Dot{\tau}}$ and $\check{\tau}$, which will be needed in the rest of the paper.

\begin{lem}
    With the previous notations, for all $s \in \mathbb{R}$:
    \begin{gather*}
        \check{\tau} (s) = \exp \left[ \frac{e^{4s}}{4 \lambda} \right], \qquad
        \check{\Dot{\tau}} (s) = e^{2s}.
    \end{gather*}
\end{lem}

\begin{proof}
The second identity is easy to state thanks to \eqref{change_time_var}. Then the same change of variable allows us to compute:
\begin{gather*}
    \frac{\diff \check{\tau}}{\diff s} (s) = \left( \frac{\tau \, \Dot{\tau}}{\lambda} \frac{\diff \tau}{\diff t} \right) \left( t(s) \right) = \left( \frac{\tau \, \Dot{\tau}^2}{\lambda} \right) \left( t(s) \right) = \frac{e^{4s}}{\lambda} \check{\tau} (s).
\end{gather*}
This yields the first identity thanks to the fact that $\underset{s \rightarrow -\infty}{\lim} \check{\tau} (s) = \underset{t \rightarrow 0^+}{\hspace{-1mm}\lim} \tau (t) = \tau (0) = 1$.
\end{proof}

Actually, we prove a slightly better result which may be adapted for other situations (for instance for Section \ref{section_KIE}).

\begin{lem}
\label{main_lem_FP_wass}
Let $g_1 \in L^\infty \left( (0, \infty), \mathcal{M}^s (\mathbb{R}^d)^d \right)$, $g_2 \in \mathcal{C}_b \left( [0, \infty), \mathcal{M}^s (\mathbb{R}^d)^d \right) \cap L^\infty \left( (0, \infty), \mathcal{M}^s_1 (\mathbb{R}^d)^d \right)$, \hfill \break $g_3 \in L^\infty \left( (0, \infty), \mathcal{M}^s (\mathbb{R}^d) \right)$ and $g_4 \in L^\infty \left( (0, \infty), \mathcal{M}^s (\mathbb{R}^d)^{d \times d} \right)$ such that \vspace{-2.5mm}
\begin{multline}
    G := \lVert g_2 \rVert_{L^\infty_t \mathcal{M}} + \lVert e^{6t} \, g_3 \rVert_{L^\infty_t \mathcal{M}} + \lVert g_4 \rVert_{L^\infty_t \mathcal{M}} + \underset{t > 0}{\supess} \int_{\mathbb{R}^d} |y| \, |g_2| (t, \diff y) \\ + \left( \int_0^\infty \left( e^{2t} \, |g_1| (t, \mathbb{R}^d) \right)^2 \diff t \right)^\frac{1}{2} + \left( \int_0^\infty \left( e^{2t} \, |g_2| (t, \mathbb{R}^d) \right)^2 \diff t \right)^\frac{1}{2} + \int_0^\infty e^{4t} \left\lvert g_4 \right\rvert (t,\mathbb{R}^d) \diff t < \infty. \label{def_G}
\end{multline}
Let $f_{\textnormal{in}} \in \mathcal{P}_2 \cap L \log L (\mathbb{R}^d)$ and suppose there exists $f := f(t,y) \in L^\infty ((0, \infty), \mathcal{P}_2 (\mathbb{R}^d)) \cap \mathcal{C} ([0, \infty), L^1_w (\mathbb{R}^d))$ satisfying
\begin{equation}
    \partial_t f = L f + e^{-2t} \, \nabla \cdot g_1 + \partial_t (e^{-2t} \, \nabla \cdot g_2) + \Delta^2 g_3 + \nabla \cdot ( \nabla \cdot g_4 ), \qquad f(0) = f_{\textnormal{in}}. \label{eq_FP_source_terms}
\end{equation}
Then there exists $C > 0$ such that
\begin{equation*}
    \mathcal{W}_1 ( f(t), \pi^{-\frac{d}{2}} \gamma^2 ) \leq C \, (1 + G + \lVert |y|^2 f \rVert_{L^\infty_t L^1_y}) \, e^{-2t} \qquad \forall t > 0.
\end{equation*}
\end{lem}

This result shows that if we already have some estimates for the function solution to the Fokker-Planck equation with source terms, and if the source terms are negligible enough, then the main behaviour coming from the Fokker-Planck equation \textit{without} source terms still holds for this function. It is actually related to the very particular form of the Fokker-Planck operator we are considering. In the same way as above with the transformation from \eqref{generalized_eq_schr_WM} to \eqref{eqtrhos_schr}, such a result may be expressed with a system similar to \eqref{generalized_eq_schr_WM}.

\begin{lem}
\label{2nd_lem_FP_wass}
Let $\lambda > 0$, $h_1 \in \mathcal{C}_b \left( (0, \infty), \mathcal{M}^s (\mathbb{R}^d)^d \right) \cap L^\infty \left( (0, \infty), \mathcal{M}^s_1 (\mathbb{R}^d)^d \right)$, $h_2 \in L^\infty \left( (0, \infty), \mathcal{M}^s (\mathbb{R}^d) \right)$ and $h_3 \in L^\infty \left( (0, \infty), \mathcal{M}^s (\mathbb{R}^d)^{d \times d} \right)$ such that \vspace{-3mm}
\begin{multline}
    G_0 := \lVert h_1 \rVert_{L^\infty \mathcal{M}} + \lVert h_2 \rVert_{L^\infty \mathcal{M}} + \lVert h_3 \rVert_{L^\infty \mathcal{M}} + \underset{t > 0}{\supess} \int_{\mathbb{R}^d} |y| \, |h_1| (t, \diff y) \\ + \left( \int_0^\infty \frac{\dot{\tau} (t)}{\tau (t)} \, \left( |h_1| (t, \mathbb{R}^d) \right)^2 \diff t \right)^\frac{1}{2} + \int_0^\infty \frac{\dot{\tau} (t)}{\tau (t)} \, \left\lvert h_3 \right\rvert (t,\mathbb{R}^d) \diff t < \infty. \label{def_G0}
\end{multline}
Suppose there exists $f := f(t,y) \in L^\infty ((0, \infty), \mathcal{P}_2 \cap L \log L (\mathbb{R}^d)) \cap \mathcal{C} ((0, \infty), L^1_w (\mathbb{R}^d))$ satisfying
\begin{equation}
\begin{gathered}
    \partial_t f + \frac{1}{\tau (t)} \nabla \cdot h_1 = 0, \qquad 
    \partial_t ( \tau (t) h_1 ) + \lambda \, \nabla f + 2 \lambda \, y \, f = \frac{1}{\tau (t)^2} \Delta \nabla h_2 - \nabla \cdot h_3. \label{eq_2nd_lem_FP_wass}
\end{gathered}
\end{equation}
Then there exists $C > 0$ depending only on $\lambda$ such that
\begin{equation*}
    \mathcal{W}_1 ( f(t), \pi^{-\frac{d}{2}} \gamma^2 ) \leq C \, \frac{1 + G_0 + \lVert |y|^2 f \rVert_{L^\infty_t L^1_y}}{\dot{\tau} (t)}, \qquad \forall t > 1.
\end{equation*}
\end{lem}

\begin{rem}
    The assumption $f \in \mathcal{C} ((0, \infty), L^1_w (\mathbb{R}^d))$ can be removed since it easily follows from the first equation in \eqref{eq_2nd_lem_FP_wass} and the fact that $h_1 \in L^\infty \left( (0, \infty), \mathcal{M}^s (\mathbb{R}^d)^d \right)$.
\end{rem}

\begin{proof}
    In the same way as above, combining both equations in \eqref{eq_2nd_lem_FP_wass} with the change of time variable $s = \frac{1}{2} \ln \dot{\tau} (t)$ (with the same notation for this change of variable in the functions) leads to the equation
    \begin{equation}
        \partial_s \check{f} - \frac{2 \lambda}{\check{\Dot{\tau}}^2} \partial_s \check{f} + \frac{\lambda}{\check{\Dot{\tau}}^2} \partial_s^2 \check{f} = L \check{f} - \frac{1}{\lambda \, \check{\tau}^2 (s)} \Delta^2 \check{h}_2 + \frac{1}{\lambda} \nabla \cdot \check{h}_3. \label{eq_2nd_lem_with_FP_op}
    \end{equation}
    The first equation of \eqref{eq_2nd_lem_FP_wass} reads in terms of $s$
    \begin{equation*}
        \partial_s \check{f} + \frac{\check{\Dot{\tau}} (s)}{\lambda} \, \nabla \cdot \check{h}_1 = 0.
    \end{equation*}
    Substituting $\partial_s \check{f}$ in the second and third term of the left-hand side of \eqref{eq_2nd_lem_with_FP_op}, we compute
    \begin{equation*}
        \frac{1}{\check{\dot{\tau}}^2} \partial_s ( \check{\dot{\tau}} \, \check{h}_1 ) = e^{-4s} \partial_s (e^{2s} \, \check{h}_1 ) = \partial_s ( e^{-2s} \, \check{h}_1 ) + 4 e^{-2s} \, \check{h}_1,
    \end{equation*}
    and so
    \begin{equation*}
        \partial_s \check{f} =  L \check{f} + 2 \, e^{-2s} \, \nabla \cdot \check{h}_1 + \partial_s ( e^{-2s} \, \nabla \cdot \check{h}_1 ) - \frac{1}{\lambda \, \check{\tau} (s)^2} \Delta^2 \check{h}_2 + \frac{1}{\lambda} \nabla \cdot \left( \nabla \cdot \check{h}_3 \right).
    \end{equation*}
    Hence we can apply Lemma \ref{main_lem_FP_wass} with $g_1 = 2 \, g_2 = \check{h}_1$, $g_3 = - \frac{1}{\lambda \, \check{\tau} (s)^2} \, \check{h}_2$ and $g_4 = \frac{1}{\lambda} \check{h}_3$ since the translation of \eqref{def_G0} into the $s$ variable implies $G \leq C \, G_0$ for some $C$ depending only on $\lambda$. The inequality we get from its conclusion leads to the expected result when coming back into the $t$ variable.
\end{proof}

The results \eqref{conv_wass_dist} and \eqref{conv_wass_dist_wigner} follow then as a simple corollary.

\begin{cor}
\label{main_cor_wasserstein}
Given any $\lambda > 0$ and $(u_{\varepsilon,\textnormal{in}})_{\varepsilon > 0}$ satisfying \eqref{in_data_assumption}, define $u_\varepsilon$ and $v_\varepsilon$ provided by Theorems \ref{th_cauchy_log_nls_eps} and \ref{main_th_log_nls_eps} and set $\rho_\varepsilon := |v_\varepsilon|^2$ for all $\varepsilon > 0$. For $\varepsilon = 0$, define also $\rho$ as the density of a Wigner Measure of the sequence $(u_\varepsilon)_{\varepsilon>0}$ given by Theorem \ref{main_th_wigner} and set $\rho_0 = \rho$.

Then there exists $C>0$ depending only on $d$ and $\lambda$ such that for all $\varepsilon \in [0, 1]$,
\begin{gather*}
    \mathcal{W}_1 \left( \pi^{-\frac{d}{2}} \rho_\varepsilon (t), \pi^{-\frac{d}{2}} \gamma^2 \right) \leq \frac{C}{\sqrt{\ln t}}, \qquad \forall t \geq 2.
\end{gather*}
\end{cor}

\begin{proof}
    The estimates \eqref{est_J_eps_int_time}-\eqref{est_nu} read in the $s$ variable:
    \begin{equation*}
        \int_0^\infty \left( e^{2s} \, |\mu_\varepsilon| (s, \mathbb{R}^d) \right)^2 \diff s + \int_0^\infty e^{4s} \left\lvert \nu_\varepsilon \right\rvert (s,\mathbb{R}^d) \diff s \leq C_0, \qquad \forall \varepsilon \geq 0.
    \end{equation*}
    Since \eqref{generalized_eq_schr_WM} holds we can apply Lemma \ref{2nd_lem_FP_wass} with (up to a factor $\pi^{-\frac{d}{2}}$) $f = \rho_\varepsilon$, $h_1 = \mu_\varepsilon$, $h_2 = \frac{\varepsilon^2}{4} \, \check{\rho}_\varepsilon$ and $h_3 = \nu_\varepsilon$ where $G_0$ (defined in \eqref{def_G0}) is uniformly bounded in $\varepsilon$ thanks to the previous result along with the estimates already proven in Theorems \ref{main_th_log_nls_eps} and \ref{main_th_wigner}. We also know that the second momentum (in space) of $\rho_\varepsilon$ is bounded uniformly in time and in $\varepsilon$.
    The result leads to \eqref{conv_FP_wass_s_var} which establishes the formula when coming back to the $t$ variable.
\end{proof}

\subsection{The harmonic Fokker-Planck operator}

The Fokker-Planck operator $L = \Delta + \nabla \cdot (2y \, .)$ is very special and well-known, due in particular to its links with the heat equation. Its form allows to compute explicitly its kernel and therefore get better estimates for the solution. Those estimates will be helpful in order to compute some convergence rate.

The fact that the kernel can be computed explicitly comes from the fact that taking the Fourier Transform in space of the Fokker-Planck operator transforms it into a simple transport operator with a source term $- |\xi|^2 \, \hat{f}$ which leads to a simple first order ODE when applying the method of characteristics, with the notation $\hat{f}$ for the spatial Fourier Transform of any function $f = f(s,y)$. 
This operator is also related to the heat equation. Indeed, if $H = H(t,x)$ is a solution to the heat equation $ \partial_t H = \frac{1}{2} \Delta H$, then $f = f(t,x)$ defined by
\begin{equation}
    f(t,x) = e^{2dt} \, H \hspace{-0.5mm} \left( \frac{e^{4t} - 1}{2}, e^{2t} \, x \right), \qquad \forall t \geq 0, \forall x \in \mathbb{R}^d, \label{rescaling_heat_to_FP}
\end{equation}
is a solution to the harmonic Fokker-Planck equation $\partial_t f = L f$. The inverse change of variable allows to pass from the heat equation to the harmonic Fokker-Planck equation in the same way. Its kernel is therefore easy to compute.

\begin{lem}
    The kernel $\mathcal{K} = \mathcal{K} (t,x,\xi)$ of the harmonic Fokker-Planck semi-group is given by
    \begin{equation*}
        \mathcal{K} (t,x,y) := \pi^{-\frac{d}{2}} \left( 1 - e^{-4t} \right)^{-\frac{d}{2}} \, \gamma^2 \hspace{-0.5mm} \left( (x - e^{-2t} y) \left( 1 - e^{-4t} \right)^{-\frac{1}{2}} \right).
    \end{equation*}
\end{lem}

\begin{proof}
    For any $f_{\textnormal{in}} \in \mathcal{S} ( \mathbb{R}^d )$, we want to compute the solution $f$ to $\partial_t f = L f$ with initial data $f(0) = f_{\textnormal{in}}$. The function $H$ defined by the rescaling \eqref{rescaling_heat_to_FP} is solution of the heat equation with initial data $H(0) = f_{\textnormal{in}}$, therefore it is known that for all $t>0$ and $x\in\mathbb{R}^d$
    \begin{equation*}
        H (t,x) = \frac{1}{(2 \pi t)^\frac{d}{2}} \int_{\mathbb{R}^d} \exp \left( - \frac{|x - y|^2}{2t} \right) f_{\textnormal{in}} (y) \diff y.
    \end{equation*}
    The result follows from some basic computations.
\end{proof}

The kernel for the harmonic Fokker-Planck equation is of course very similar to that for the heat equation. In particular, for all $t > 0$ and all $x \in \mathbb{R}^d$, $\mathcal{K} (t,x,.) \in \mathcal{S} (\mathbb{R}^d)$, so there is a huge regularization in the same way as for the heat equation. Moreover, if $e^{tL}$ is not a convolution (which is the case for the heat equation), it is not far from this feature since $\mathcal{K} (t)$ depends only on $x - e^{-2t} y$. In particular, we get for all $n \in \mathbb{N}$, $I \in \{1,...,d\}^n$, $t > 0$ and $x,y \in \mathbb{R}^d$,
\begin{equation*}
    \partial_{y_I} \mathcal{K} (t,x,y) = (-1)^n e^{-2nt} \, \partial_{x_I} \mathcal{K} (t,x,y), \quad \text{and} \quad \int_{\mathbb{R}^d} \left\lvert D_{x}^n \mathcal{K} (t,x,y) \right\rvert \diff x = \frac{\left( 1 - e^{-4t} \right)^{-\frac{n}{2}}}{\pi^\frac{d}{2}} \int_{\mathbb{R}^d} \left\lvert D^n (\gamma^2) (x) \right\rvert \diff x.
\end{equation*}
    %
There is also another identity we will need later:
\begin{equation}
    \int_{\mathbb{R}^d} |x - e^{-2t} y| \, \left\lvert D_{x}^n \mathcal{K} (t,x,y) \right\rvert \diff x = \frac{\left( 1 - e^{-4t} \right)^{-\frac{n-1}{2}}}{\pi^\frac{d}{2}} \int_{\mathbb{R}^d} |x| \left\lvert D^n (\gamma^2) (x) \right\rvert \diff x. \label{est_x_K_L1}
\end{equation}
The first two identities are crucial for the next lemma.

\begin{lem}
    Given any $f_0 \in \mathcal{M}^s (\mathbb{R}^d)$, $n\in \mathbb{N}$ and $I \in \{1,...,d\}^n$, $f(t) = e^{t L} ( \partial_I^n f_0 )$ is a $W^{\infty,1}$ function for all $t>0$, and for all $m \in \mathbb{N}$ we have:
    \begin{equation*}
        f(t) = e^{-2nt} \, \partial_I^n \hspace{-0.5mm} \left( e^{tL} f_0 \right) \qquad \text{and} \qquad
        \lVert f (t) \rVert_{\dot{W}^{m-n,1}} \leq \frac{e^{-2nt}}{\left( 1 - e^{-4t} \right)^{\frac{m}{2}}} \frac{\lVert \gamma^2 \rVert_{\dot{W}^{m,1}}}{\pi^\frac{d}{2}} \, |f_0| (\mathbb{R}^d).
    \end{equation*}
\end{lem}

\begin{proof}
    With $n$ integrations by parts, the previous identity for $\partial_{y_I} \mathcal{K} (t,x,.)$ and the 
    Lebesgue theorem, we get for all $t>0$ and $x \in \mathbb{R}^d$:
    \begin{equation*}
        f(t,x) = \langle \mathcal{K} (t,x,.), \partial_I f_0 \rangle = (-1)^n \langle \partial_{y_I} \mathcal{K} (t,x,.), f_0 \rangle = e^{-2nt} \int_{\mathbb{R}^d} \partial_{x_I} \mathcal{K} (t,x,y) \, f_0 (\diff y) = e^{-2nt} \, \partial_I \hspace{-0.5mm} \left( e^{tL} f_0 \right) (x).
    \end{equation*}
    The estimate readily comes with the fact that, with the 
    Lebesgue theorem again and the second identity:
    \begin{align*}
        \lVert \partial_I \hspace{-0.5mm} \left( e^{tL} f_0 \right) \rVert_{\dot{W}^{m-n,1}} &\leq \lVert D^m_x \hspace{-0.5mm} \left( e^{tL} f_0 \right) \rVert_{L^1} = \int_{\mathbb{R}^d} \left\lvert \int_{\mathbb{R}^d} D_x^m \mathcal{K} (t,x,y) \, f_0 (\diff y) \right\rvert \diff x \\
        &\leq  \int_{\mathbb{R}^d} \int_{\mathbb{R}^d} \left\lvert D_{x}^n \mathcal{K} (t,x,y) \right\rvert \diff x \, |f_0| (\diff y) = \left( 1 - e^{-4t} \right)^{-\frac{n}{2}} \lVert \gamma^2 \rVert_{\dot{W}^{m,1}} \int_{\mathbb{R}^d} |f_0| (\diff y). \qedhere
    \end{align*}
\end{proof}

In particular, for $m=-n+1$, the bound is integrable in time, which shows that integrating in time leads to a better regularity than the source term. It is also not far from being integrable for $m=-n+2$, since we get a bound in $t^{-1}$, but of course we cannot reach this regularity. However, some kind of cut off in the integral will lead to an interesting bound in order to get as close to this regularity as possible.

\begin{lem} \label{main_lem_FP_est}
Given any $h \in L^\infty \left( (0, T), \mathcal{M}^s (\mathbb{R}^d) \right)$ for some $T>0$, $n \in \mathbb{N}$ and $I \in \{1,...,d\}^n$, there exists a unique solution $f \in \mathcal{C} \left( [0,T], W^{-n+1, 1} (\mathbb{R}^d) \right)$ 
to:
\begin{equation}
    \partial_t f = L f + \partial_I^n h \qquad \text{in $\mathcal{D}'\left((0,T) \times \mathbb{R}^d \right)$}, \qquad \qquad f(0) = 0, \label{eq_FP_with_source_term}
\end{equation}
given for all $t \in [0,T]$ by:
\begin{equation}
    f(t) = \int_0^t e^{(t-u)L} ( \partial_I^n h (u) ) \diff u = \partial_I^n \int_0^t e^{-2n(t-u)} e^{(t-u)L} h (u) \diff u, \label{expression_sol_FP_source}
\end{equation}
where the last integral is to be understood as a Bochner integral.
Moreover, some estimates holds:
\begin{enumerate}
    \item It satisfies for all $t \in [0,T]$:
    \begin{align}
        \lVert f(t) \rVert_{\dot{W}^{-n,1}} &\leq e^{-2nt} \int_0^t e^{2nu} \, |h| \hspace{-0.5mm} \left( u, \mathbb{R}^d \right) \diff u, \label{FP_source_-n_est} \\
        \begin{split}
            \lVert f(t) \rVert_{\dot{W}^{-n+1,1}} &\leq \frac{d}{2} e^{-2nt} \int_0^t \frac{e^{2nu}}{\left( 1 - e^{-4(t-u)} \right)^\frac{1}{2} } \, |h| \hspace{-0.5mm} \left( u, \mathbb{R}^d \right) \diff u \\ &\leq \frac{d}{2} e^{-2(n-1)t} \int_0^t \frac{e^{-2u}}{\left( 1 - e^{-4u} \right)^\frac{1}{2} } \diff u \, \left\lVert e^{2(n-1)u} |h| \hspace{-0.5mm} \left( u, \mathbb{R}^d \right) \right\rVert_{L^\infty_u}.
        \end{split} \label{FP_source_-n+1_est}
    \end{align}
    \item For all $T>t>S>0$, $f(t) = f_{1,S} (t) + f_{2,S} (t) + f_3 (t)$ where $f_{1,S} (t) \in W^{\infty,1} (\mathbb{R}^d)$, $f_{2,S} (t) \in \dot{W}^{-n+1,1} (\mathbb{R}^d)$ and $f_3 (t) \in \dot{W}^{-n+3,1} (\mathbb{R}^d)$ are given by
    \begin{gather*}
        f_{1,S} (t) = \int_0^S e^{-4(t-u)} \, e^{(t-u)L} (\partial_I^n h (u))  \diff u, \qquad
        f_{2,S} (t) = \int_S^t e^{-4(t-u)} \, e^{(t-u)L} (\partial_I^n h (u) ) \diff u, \\
        f_3 (t) = \int_0^t \left( 1 - e^{-4(t-u)} \right) e^{(t-u)L} (\partial_I^n h (u) ) \diff u,
    \end{gather*}
    and satisfy:
    \begin{gather}
        \hspace{-8mm}    \lVert f_{1,S} (t) \rVert_{\dot{W}^{-n+2,1}} \leq \frac{\lVert \gamma^2 \rVert_{\dot{W}^{2,1}}}{2 \pi^\frac{d}{2}} e^{-2nt} \left[ \frac{1}{e^{4t} -e^{4S}} - \frac{1}{e^{4t} - 1} \right]^\frac{1}{2} \left( \int_0^{S} \left( e^{2(n+1)u} \, |h| \hspace{-0.5mm} \left( u, \mathbb{R}^d \right) \right)^2 \diff u \right)^\frac{1}{2}, \label{est_f1S} \\
        \lVert f_{2,S} (t) \rVert_{\dot{W}^{-n+1,1}} \leq \frac{d}{4} e^{-2(n+1)t} \left( e^{4t} - e^{4S} \right)^\frac{1}{2} \left\lVert e^{2nu} \, |h| \hspace{-0.5mm} \left( u, \mathbb{R}^d \right) \right\rVert_{L^\infty_u}, \label{est_f2S} \\
        \lVert f_3 (t) \rVert_{\dot{W}^{-n+2,1}} \leq \frac{\lVert \gamma^2 \rVert_{\dot{W}^{2,1}}}{2 \pi^\frac{d}{2}} e^{-2nt} \left( 1 - e^{-4t} \right)^\frac{1}{2} \left( \int_{0}^t \left( e^{2(n+1)u} \, |h| \hspace{-0.5mm} \left( u, \mathbb{R}^d \right) \right)^2 \diff u \right)^\frac{1}{2}. \label{est_f3S}
    \end{gather}
    \item If $h \in L^\infty \left( (0, T), \mathcal{M}_1^s (\mathbb{R}^d) \right)$, then for all $t \in [0,T]$,
    \begin{multline}
        \hspace{-3mm} \lVert x \, f (t) \rVert_{\dot{W}^{-n+1,1}_x} \leq \frac{d}{2} e^{-2nt} \int_0^t \frac{e^{-2u}}{\left( 1 - e^{-4u} \right)^\frac{1}{2}} \diff u \left\lVert e^{2nu} \int_{\mathbb{R}^d} |x| |h| \hspace{-0.5mm} \left( u, \diff x \right) \right\rVert_{L^{\infty}_u} \\ + \left[ \frac{\left\lVert |.| \nabla \gamma^2 \right\rVert_{L^1}}{\pi^\frac{d}{2}} + n - 1 \right] e^{-2nt} \int_0^t e^{2nu} \, |h| \hspace{-0.5mm} \left( u, \mathbb{R}^d \right) \diff u. \label{est_x_f_FP}
    \end{multline}
\end{enumerate}
\end{lem}

\begin{proof}
    The first part is easy to prove thanks to the previous remarks and the usual way to deal with the source term thanks to the semigroup of an evolution equation. For the estimates, the first inequality easily follows from \eqref{expression_sol_FP_source} along with the previous estimates:
    \begin{align*}
        \lVert f(t) \rVert_{\dot{W}^{-n,1}} &\leq \left\lVert \int_0^t e^{-2n(t-u)} e^{(t-u)L} h (u) \diff u \right\rVert_{L^1} \leq \int_0^t e^{-2n(t-u)} \left\lVert e^{(t-u)L} h (u) \right\rVert_{L^1} \diff u \\
        &\leq \int_0^t e^{-2n(t-u)} \, |h| \hspace{-0.5mm} \left( u, \mathbb{R}^d \right) \diff u.
    \end{align*}
    In the same way for the second estimate:
    \begin{align*}
        \lVert f(t) \rVert_{\dot{W}^{-n+1,1}} &\leq \left\lVert \nabla \int_0^t e^{-2n(t-u)} e^{(t-u)L} h (u) \diff u \right\rVert_{L^1} \leq \int_0^t e^{-2n(t-u)} \left\lVert e^{(t-u)L} h (u) \right\rVert_{\dot{W}^{1,1}} \diff u \\
        &\leq \int_0^t \frac{e^{-2n(t-u)}}{\left( 1 - e^{-4(t-u)} \right)^\frac{1}{2} } \frac{\lVert \gamma^2 \rVert_{\dot{W}^{1,1}}}{\pi^\frac{d}{2}}  \, |h| \hspace{-0.5mm} \left( u, \mathbb{R}^d \right)  \diff u.
    \end{align*}
    As for the second part, similar computations may be done, so that for $f_{1,S}$:
    \begin{align*}
        \lVert f_{1,S} (t) \rVert_{\dot{W}^{-n+2,1}} &\leq \int_0^S \frac{e^{-(4+2n)(t-u)}}{1 - e^{-4(t-u)}} \frac{\lVert \gamma^2 \rVert_{\dot{W}^{2,1}}}{\pi^\frac{d}{2}} \, |h| \hspace{-0.5mm} \left( u, \mathbb{R}^d \right) \diff u \\
        &\leq \frac{\lVert \gamma^2 \rVert_{\dot{W}^{2,1}}}{\pi^\frac{d}{2}} e^{-2(n+1)t} \int_0^S \frac{e^{-2(t-u)}}{1 - e^{-4(t-u)}} \, e^{2(n+1)u} \, |h| \hspace{-0.5mm} \left( u, \mathbb{R}^d \right) \diff u \\
        &\leq \frac{\lVert \gamma^2 \rVert_{\dot{W}^{2,1}}}{\pi^\frac{d}{2}} e^{-2(n+1)t} \left( \int_0^S \frac{e^{-4(t-u)}}{(1 - e^{-4(t-u)})^2} \diff u \right)^\frac{1}{2} \left( \int_0^S \left( e^{2(n+1)u} \, |h| \hspace{-0.5mm} \left( u, \mathbb{R}^d \right) \right)^2 \diff u \right)^\frac{1}{2},    \end{align*}
    and we find \eqref{est_f1S} when we compute
    \begin{equation*}
        \int_0^S \frac{e^{-4(t-u)}}{(1 - e^{-4(t-u)})^2} \diff u = \frac{1}{4} \left[ \frac{1}{1 - e^{-4 (t-S)}} - \frac{1}{1 - e^{-4t}} \right].
    \end{equation*}
    In the same way for $f_{2,S}$, it yields
    \begin{align*}
        \lVert f_{2,S} (t) \rVert_{\dot{W}^{-n+1,1}} &\leq \frac{d}{2} e^{-2nt} \int_S^t \frac{e^{-4(t-u)}}{(1 - e^{-4(t-u)})^\frac{1}{2}} \, e^{-2nu} \, |h| \hspace{-0.5mm} \left( u, \mathbb{R}^d \right) \diff u \\
        &\leq \frac{d}{2} e^{-2nt} \int_S^t \frac{e^{-4(t-u)}}{(1 - e^{-4(t-u)})^\frac{1}{2}} \diff u \left\lVert e^{2nu} \, |h| \hspace{-0.5mm} \left( u, \mathbb{R}^d \right) \right\rVert_{L^\infty_u},
    \end{align*}
    which is exactly \eqref{est_f2S} when we compute the remaining integral. Then, for $f_3$, it is easy to check that
    \begin{align*}
        \lVert f_3 (t) \rVert_{\dot{W}^{-n+2,1}} &\leq \int_0^t e^{-(4+2n)(t-u)} \, \frac{\lVert \gamma^2 \rVert_{\dot{W}^{2,1}}}{\pi^\frac{d}{2}} \, |h| \hspace{-0.5mm} \left( u, \mathbb{R}^d \right) \diff u \\
        &\leq \frac{\lVert \gamma^2 \rVert_{\dot{W}^{2,1}}}{\pi^\frac{d}{2}} e^{-2(n+1)t} \left( \int_0^t e^{-4(t-u)} \diff u \right)^\frac{1}{2} \left( \int_0^t \left( e^{2(n+1)u} \, |h| \hspace{-0.5mm} \left( u, \mathbb{R}^d \right) \right)^2 \diff u \right)^\frac{1}{2},
    \end{align*}
    and therefore \eqref{est_f3S}.
    
    The third part is a bit more tricky. For all $t > 0$ we define $f_{4} (t) \in \dot{W}^{-n+1,1} (\mathbb{R}^d)$ and $f_{5} (t) \in \dot{W}^{-n+2, 1} (\mathbb{R}^d)$ by:
    \begin{gather*}
        f_{4} (t) = \partial_{\tilde{I}}^{n-1} \left( x \, \partial_{i_n} \int_0^t e^{-2n(t-u)} e^{(t-u) L} h (u) \diff u \right), \qquad
        f_{5} (t) = \left[ x, \partial_{\tilde{I}}^{n-1} \right] \partial_{i_n} \int_0^t e^{-2n(t-u)} e^{(t-u) L} h (u) \diff u,
    \end{gather*}
    where $\tilde{I} = (i_1, ..., i_{n-1})$ with $I = (i_1, ..., i_n)$. It is obvious that $x \, f (t) = f_{4} (t) + f_5 (t)$. Moreover, $f_5 (t)$ is easy to estimate due to the fact that $\left[ x, \partial_{\tilde{I}}^{n-1} \right]$ can be readily computed, which leads to:
    \begin{equation*}
        \lVert f_{5,S} (t) \rVert_{\dot{W}^{-n+1,1}} \leq (n-1) \left\lVert \int_0^t e^{-2n(t-u)} e^{(t-u) L} h (u) \diff u \right\rVert_{L^1} \leq (n-1) \int_0^t e^{-2n(t-u)} \, |h| \hspace{-0.5mm} \left( u, \mathbb{R}^d \right) \diff u.
    \end{equation*}
    For $f_4 (t)$, first of all, to get an estimate in $\dot{W}^{-n+1,1}$, we only need to focus on $\tilde{f}_4 (t)$ in $L^1$ where
        \begin{gather*}
        \tilde{f}_{4} (t) = x \, \partial_{i_n} \int_0^t e^{-2n(t-u)} e^{(t-u) L} h (u) \diff u,
    \end{gather*}
    since $f_4 (t) = \partial_{\tilde{I}}^{n-1} \tilde{f}_4 (t)$. For this, we will estimate
    \begin{equation*}
        f_6 (t) = x \int_{\mathbb{R}^d} \partial_{x_i} \mathcal{K} (t, x, y) \, f_0(\diff y) =  \int_{\mathbb{R}^d} (x - e^{-2t} y) \, \partial_{x_i} \mathcal{K} (t, x, y) \, f_0(\diff y) + e^{-2t} \int_{\mathbb{R}^d} \partial_{x_i} \mathcal{K} (t, x, y) \, y \, f_0(\diff y).
    \end{equation*}
    Using the expression of $\mathcal{K}$, the first term on the right-hand side can be estimated thanks to \eqref{est_x_K_L1}:
    \begin{equation*}
        \left\lVert \int_{\mathbb{R}^d} (x - e^{-2t} y) \, \partial_{x_i} \mathcal{K} (t, x, y) \, f_0(\diff y) \right\rVert_{L^1} \leq \pi^{-\frac{d}{2}} \, |f_0| (\mathbb{R}^d) \int |x| |\nabla \gamma^2 |.
    \end{equation*}
    The second term is also easy to estimate:
    \begin{equation*}
        \left\lVert \int_{\mathbb{R}^d} \partial_{x_i} \mathcal{K} (t, x, y) \, y \, f_0(\diff y) \right\rVert_{L^1} \leq \frac{\left( 1 - e^{-4t} \right)^{-\frac{1}{2}}}{\pi^\frac{d}{2}} \lVert \gamma^2 \rVert_{\dot{W}^{1,1}} \, \int_{\mathbb{R}^d} |y| \, |f_0| (\diff y).
    \end{equation*}
    Coming back to $\tilde{f}_4 (t)$, those estimates lead to
    \begin{align*}
        \lVert \tilde{f}_4 (t) \rVert_{L^1} &\leq \pi^{-\frac{d}{2}} \int_0^t e^{-2n(t-u)} \left[ |h| (u, \mathbb{R}^d) \, \lVert |.| \nabla \gamma^2 \rVert_{L^1} + \frac{e^{-2(t-u)}}{(1 - e^{-4(t-u)})^{\frac{1}{2}}} \lVert \gamma^2 \rVert_{\dot{W}^{1,1}} \, \int_{\mathbb{R}^d} |y| \, |h| (u, \diff y) \right] \diff u \\
        &\leq \frac{\lVert |.| \nabla \gamma^2 \rVert_{L^1}}{\pi^\frac{d}{2}} e^{-2nt} \int_0^t e^{2nu} |h| (u, \mathbb{R}^d) \diff u + \frac{d}{2} e^{-2nt} \int_0^t \frac{e^{-2(t-u)}}{(1 - e^{-4(t-u)})^{\frac{1}{2}}} \, e^{2nu} \int_{\mathbb{R}^d} |y| \, |h| (u, \diff y) \diff u \\
        &\leq \frac{\lVert |.| \nabla \gamma^2 \rVert_{L^1}}{\pi^\frac{d}{2}} e^{-2nt} \int_0^t e^{2nu} |h| (u, \mathbb{R}^d) \diff u + \frac{d}{2} e^{-2nt} \int_0^t \frac{e^{-2v}}{(1 - e^{-4v})^{\frac{1}{2}}} \diff v \left\lVert e^{2nu} \int_{\mathbb{R}^d} |y| \, |h| (u, \diff y) \right\rVert_{L^\infty_u},
    \end{align*}
    which is exactly what we need to get \eqref{est_x_f_FP} when putting all back together.
\end{proof}

We dealt with the spatial derivative in the source term. Actually, thanks to the linearity of this equation, we can also deal with a time-derivative, subject to a slightly higher regularity for the source term.

\begin{cor} \label{cor_FP_with_source_time_deriv}
Let $h \in L^\infty \left( (0, T), \mathcal{M}_1 (\mathbb{R}^d) \right) \cap \mathcal{C} \left( [0, T), \mathcal{M} (\mathbb{R}^d) \right)$ for some $T>0$, $n \in \mathbb{N}$ and $I \in \{1,...,d\}^n$. Then there exists a unique solution $f \in L^\infty \left( (0,T), W^{-n-1, 1} (\mathbb{R}^d) \right)$ of the Fokker-Planck equation with source term:
\begin{equation*}
    \partial_t f = L f + \partial_t \partial_I^n h \qquad \text{in $\mathcal{D}'\left((0,T) \times \mathbb{R}^d \right)$}, \qquad \qquad f(0) = 0.
\end{equation*}
It is given by the identity $f=\Delta g + \nabla \cdot (2y \, g) + \partial_I^n h - e^{-2nt} \partial_I^n ( e^{tL} h (0) )$ where $g$ is the unique solution in $\mathcal{C} \left( (0,T), W^{-n+1, 1} (\mathbb{R}^d) \right)$ to \eqref{eq_FP_with_source_term} 
\end{cor}

\begin{proof}
    Suppose that such an $f$ exists.
    Define
    \begin{equation*}
        g (t) = \int_0^t \left( f(u) + e^{uL} (\partial_I^n h(0)) \right) \diff u \in W^{1,\infty} \left( (0, T), W^{-n-1,1} (\mathbb{R}^d) \right),
    \end{equation*}
    so that $\partial_t g = f + e^{uL} (\partial_I^n h(0))$ and thus:
    \begin{equation*}
        \partial_t \left( \partial_t g - L g - \partial_I^n h \right) = 0, \qquad g(0) = 0.
    \end{equation*}
    Moreover, $\partial_t g (0) = \partial_I^n h (0)$ and $L g (0) = 0$ so that
    \begin{equation*}
        \partial_t g - L g - \partial_I^n h = 0, \qquad g(0) = 0.
    \end{equation*}
    The result obviously follows.
\end{proof}

\subsection{Proof of Lemma \ref{main_lem_FP_wass}}

\subsubsection{Duality and regularization}

Lemma \ref{main_lem_FP_est} provides interesting estimates in view of Lemma \ref{main_lem_FP_wass}. Indeed, there are many source terms in \eqref{eq_FP_source_terms} with different regularities, and we can apply for each of them one of the previous estimates with different $n$ by linearity of the Fokker-Planck operator $L$. However, it is obvious that we will not be able to reach (at least at first) a non-negative regularity for all the estimates, for instance because of the $\tau (t)^{-2} \, \Delta^{\hspace{-0.5mm} 2} g_3$ term for which we have $n = 4$ in Lemma \ref{main_lem_FP_est}, and the best estimate we can get is for $\dot{W}^{-3,1}$. Therefore, if we want to estimate in a higher regularity, we need to use duality and regularize the test function to fit the lower regularity for which we have the estimate (for instance with a convolution). We also need to check if this regularization suits the estimate, i.e. if we can get a nice convergence rate for the difference between the initial test function and the regularized one in $L^\infty$ thanks to the assumption that $f (t)$ is in $L^1$ uniformly in $t$. For example, if one wants to have a convergence rate in $L^1$ strong through this way, they would have to regularize an $L^\infty$ test function into a smoother function. However, approaching a general $L^\infty$-function by a regular function is not very convenient (if not doomed). Actually, there is a more suitable case: the Wasserstein (or Kantorovich–Rubinstein) distance $\mathcal{W}_1$. Indeed, such a distance has a dual representation: for any $\mu_1$, $\mu_2$ in $\mathcal{P}_1 (\mathbb{R}^d)$,
\begin{equation*}
    \mathcal{W}_{1} (\mu_1, \mu_2) = \sup \{ \int_{\mathbb{R}^d} \Phi \diff (\mu_1 - \mu_2), \Phi : \mathbb{R}^d \to \mathbb{R} \mbox{ continuous}, \mathrm{Lip} (\Phi) \leq 1 \}.
\end{equation*}
The fact that $\Phi$ is 1-Lip is suitable in order to regularize it, whereas the fact that $\Phi$ may be unbounded (but growing at most like an affine function) is not a big problem thanks to the assumption on the integrability of $f$ (in particular its uniformly bounded second momentum). 

Given any $\Phi : \mathbb{R}^d \to \mathbb{R}$ 1-Lip, before using the estimates in Lemma \ref{main_lem_FP_est}, we need to quantify the cost of its regularization into a smoother function. We will regularize it into a $\mathcal{C}_c^\infty$ function since it is not very difficult. Our first action is to regularize $\Phi$ into a $\mathcal{C}^\infty$ function by convolution with a smooth and suitable mollifier. Take some $\Psi \in \mathcal{S}(\mathbb{R}^d)$ such that $\Psi \geq 0$ and $\int_{\mathbb{R}^d} \Psi = 1$. For $\delta > 0$, define $\Psi^{\delta}$ by $\Psi^{\delta} (x) = \frac{1}{\delta^d} \Psi \left( \frac{x}{\delta} \right)$ for all $x \in \mathbb{R}^d$. Then it is known that $\tilde{\Phi}^\delta := \Phi * \Psi^{\delta}$ is a $\mathcal{C}^\infty$ function and satisfies:
\begin{enumerate}
    \item $\lVert \tilde{\Phi}^\delta - \Phi \rVert_{L^\infty} \leq \delta \, \lVert \, |.| \, \Psi \rVert_{L^1}$,
    \item $\mathrm{Lip} (\tilde{\Phi}^\delta) \leq 1$ and in a more general way, $\forall n \in \mathbb{N}, \lVert \tilde{\Phi}^\delta \rVert_{\Dot{W}^{1+n,\infty}} \leq \delta^{-n} \, \lVert \Psi \rVert_{\Dot{W}^{n,\infty}}$.
\end{enumerate}
In particular, the first estimate yields
\begin{equation*}
    \left\lvert \int_{\mathbb{R}^d} \tilde{\Phi}^\delta \diff (f(t) - \pi^{-\frac{d}{2}} \gamma^2) - \int_{\mathbb{R}^d} \Phi \diff (f(t) - \pi^{-\frac{d}{2}} \gamma^2) \right\rvert \leq 2 \delta \, \lVert \, |.| \, \Psi \rVert_{L^1 (\mathbb{R}^d)}.
\end{equation*}

Now, we want to apply a cut off to the function $\tilde{\Phi}^\delta$. It is in this step where the fact that the second momentum of $f(t)$ is bounded independently of $t$ is used. Take $\chi \in \mathcal{C}_c^\infty (\mathbb{R}^d)$ such that $\chi \equiv 1$ on $\mathcal{B}(0,1)$ and $0 \leq \chi \leq 1$. Define $\chi^{\delta}$ by $\chi^{\delta} (x) = \chi(\delta x)$. Then $\Phi^\delta := (\tilde{\Phi}^\delta - \tilde{\Phi}^\delta (0)) \chi^{\delta} \in \mathcal{C}^\infty_c (\mathbb{R}^d)$ and we get some similar properties from simple computations:
\begin{enumerate}
    \item Given any $x \in \mathbb{R}^d \setminus \{ 0 \}$, we get
    \begin{align*}
        \left\lvert \frac{(\tilde{\Phi}^\delta (x) - \tilde{\Phi}^\delta (0)) - \Phi^\delta (x)}{|x|^2} \right\rvert &= \left\lvert \frac{(\tilde{\Phi}^\delta (x) - \tilde{\Phi}^\delta (0)) (1 - \chi^{\delta} (x))}{|x|^2} \right\rvert = \frac{| \tilde{\Phi}^\delta (x) - \tilde{\Phi}^\delta (0) |}{|x|} \, \frac{| 1 - \chi (\delta x) |}{|x|} \\
        &\leq \frac{| 1 - \chi (\delta x) |}{|x|} \leq \delta,
    \end{align*}
    so that
    \begin{equation}
        \left\lVert \frac{(\tilde{\Phi}^\delta - \tilde{\Phi}^\delta (0)) - \Phi^\delta}{|.|^2} \right\rVert_{L^\infty} \leq \delta. \label{cut_off_1}
    \end{equation}
    \item We also get
    \begin{align*}
        \nabla \Phi^\delta (x) = \nabla \left( (\Phi - \Phi(0)) \, \chi^{2,\delta} \right) (x) &= \nabla \Phi (x) \, \chi^{2,\delta} (x) + (\Phi (x) - \Phi (0)) \, \nabla \chi^{2,\delta} (x) \\
        &= \nabla \Phi (x) \, \chi(\delta x) + \delta (\Phi (x) - \Phi (0)) \, \nabla \chi(\delta x),
        \intertext{and therefore}
        \left\lvert \nabla \left( (\Phi - \Phi(0)) \, \chi^{2,\delta} \right) (x) \right\rvert &\leq \chi(\delta x) + \delta |x| \,  \left\lvert \nabla \chi(\delta x) \right\rvert \\
        &\leq \lVert \chi + |.| \, | \nabla \chi| \, \rVert_{L^\infty}.
    \end{align*}
    Hence, $\Phi^\delta$ is Lip uniformly in $\delta$:
    \begin{equation*}
        \mathrm{Lip} \left( \Phi^\delta \right) \leq \lVert \chi + |.| \, | \nabla \chi| \, \rVert_{L^\infty}. \label{cut_off_2}
    \end{equation*}
    \item In the same way, computing the $n$-th derivative of $\Phi^\delta$ leads to the following property:
    \begin{equation}
        \forall n \in \mathbb{N}, \exists C_n > 0, \forall \delta \in (0,1], \lVert \Phi^\delta \rVert_{\dot{W}^{1+n,\infty}} \leq C_n \, \delta^{-n}. \label{cut_off_3}
    \end{equation}
\end{enumerate}

In particular, given ant $t \geq 0$, \eqref{cut_off_1} along with the fact that $\int_{\mathbb{R}^d} \tilde{\Phi}^\delta(0) \diff (f(t) - \gamma^2) = 0$ lead to
\begin{equation*}
    \left\lvert \int_{\mathbb{R}^d} \Phi^\delta \diff (f(t) - \gamma^2) - \int_{\mathbb{R}^d} \tilde{\Phi}^\delta \diff (f(t) - \gamma^2) \right\rvert \leq \delta \, \int_{\mathbb{R}^d} |y|^2 \, |f(t) - \gamma^2| (\diff y).
\end{equation*}
Therefore, the cost of the regularization of $\Phi$ into $\Phi^\delta$ is only proportional to $\delta$. In view of the convergence rate which must be reached (in $e^{-2t}$), we should define $\delta (t) = e^{-2t}$ and consider $\Phi^{\delta (t)}$. Therefore, the previous estimates yield
\begin{equation}
    \left\lvert \int_{\mathbb{R}^d} \Phi^{\delta (t)} \, (f (t) - \pi^{-\frac{d}{2}} \, \gamma^2) \diff y - \int_{\mathbb{R}^d} \Phi \, (f (t) - \pi^{-\frac{d}{2}} \, \gamma^2) \diff y \right\rvert \leq ( C_0 + \lVert \, |y|^2 f \, \rVert_{L^\infty_t L^1_y}) \, e^{-2t}. \label{est_diff_Phi_delta_t}
\end{equation}

\subsubsection{End of the proof}

It remains to estimate $\int_{\mathbb{R}^d} \Phi^{\delta (t)} \, (f (t) - \pi^{-\frac{d}{2}} \gamma^2) \diff y$. We now use the fact that $f$ satisfies \eqref{eq_FP_source_terms}. Define $f_1$, $f_2$, $f_3$ and $f_4$ to be the solutions to \eqref{eq_FP_with_source_term} respectively with source terms $e^{-2t} \, \nabla \cdot g_1$, $ e^{-2t} \, \nabla \cdot g_2$, $\Delta^2 g_3$ and $\nabla \cdot ( \nabla \cdot g_4 )$. Define also $f_0 = e^{tL} f_{\textnormal{in}}$. Therefore, $f$ may be written in terms of $f_0$, $f_1$, $f_2$, $f_3$ and $f_4$ thanks to Lemma \ref{main_lem_FP_est} and Corollary \ref{cor_FP_with_source_time_deriv}:
\begin{equation}
    f = f_0 + f_1 + f_3 + f_4 + \Delta f_2 + \nabla \cdot (2y \, f_2) + e^{-2t} \, \nabla \cdot g_2 - e^{-2t} \, \nabla \cdot ( e^{tL} g_2(0) ). \label{exp_f_proof_lem_FP}
\end{equation}
\begin{itemize}
    \item For $f_0$, thanks to Lemma \ref{lem_wass_conv_FP_without_source_term} and to the inequality $\mathcal{W}_1 \leq \mathcal{W}_2$, it is known that
    \begin{equation*}
        \mathcal{W}_1 \left( f_0(t), \pi^{-\frac{d}{2}} \, \gamma^2 \right) \leq e^{-2t} \, \mathcal{W}_2 \left( f_{\textnormal{in}}, \pi^{-\frac{d}{2}} \, \gamma^2 \right) \leq C_0 \, e^{-2t} \, (1 + \lVert \, |y|^2 f_{\textnormal{in}} \, \rVert_{L^1}), \qquad \forall t \geq 0.
    \end{equation*}
    Moreover, $\lVert \, |y|^2 f_{\textnormal{in}} \, \rVert_{L^1} \leq \lVert \, |y|^2 f \, \rVert_{L^\infty_t L^1_y}$ thanks to the assumption $f \in \mathcal{C} ((0,\infty), L^1_w (\mathbb{R}^d))$.
    Therefore, there holds for all $t\geq0$:
    \begin{align*}
        \left\lvert \int_{\mathbb{R}^d} \Phi^{\delta (t)} \, (f_0 (t) - \pi^{-\frac{d}{2}} \gamma^2) \diff y \right\rvert &\leq \mathrm{Lip} ( \Phi^{\delta (t)} ) \, \pi^\frac{d}{2} \, \mathcal{W}_1 \left( f_0 (t), \pi^{-\frac{d}{2}} \gamma^2 \right)
        \leq C_0 \, (1 + \lVert \, |y|^2 f \, \rVert_{L^\infty_t L^1_y}) \, e^{-2t}.
    \end{align*}
    Thus, it suffices to prove that any of the other terms in \eqref{exp_f_proof_lem_FP} integrated against $\Phi^{\delta (t)}$ goes to 0 with the same exponential convergence rate.
    \item For the last two terms, this convergence is easy to state:
    \begin{align*}
        \langle \Phi^{\delta (t)}, e^{-2t} \, \nabla \cdot ( e^{tL} g_2(0) ) \rangle &= - e^{-2t} \int_{\mathbb{R}^d} \nabla \Phi^{\delta (t)} \, e^{tL} g_2(0) \diff y, \\
        \left\lvert \langle \Phi^{\delta (t)}, e^{-2t} \, \nabla \cdot ( e^{tL} g_2(0) ) \rangle \right\rvert &\leq e^{-2t} \lVert \nabla \Phi^{\delta (t)} \rVert_{L^\infty} \, \lVert e^{tL} g_2(0) \rVert_{L^1} \leq C_0 \, |g_2| (0, \mathbb{R}^d) \, e^{-2t} \leq C_0 \, G \, e^{-2t},
    \end{align*}
    and in the same way
    \begin{align*}
        \left\lvert \langle \Phi^{\delta (t)}, e^{-2t} \, \nabla \cdot g_2 (t) \rangle \right\rvert &\leq e^{-2t} \lVert \nabla \Phi^{\delta (t)} \rVert_{L^\infty} \, |g_2|(t, \mathbb{R}^d) \leq C_0 \, G \, e^{-2t}.
    \end{align*}
    \item For $f_1$, we use \eqref{FP_source_-n_est} with $n=1$ to get for all $t \geq 0$:
    \begin{align*}
        \hspace{-5mm} \lVert f_1 (t) \rVert_{\dot{W}^{-1,1}} \leq e^{-2t} \int_0^t |g_1| (u, \mathbb{R}^d) \diff u \leq e^{-2t} \left( \int_0^t \left( e^{2u} \, |g_1| (u, \mathbb{R}^d) \right)^2 \diff u \right)^{\hspace{-0.5mm} \frac{1}{2}} \left( \int_0^t e^{-4u} \diff u \right)^{\hspace{-0.5mm} \frac{1}{2}} \leq \frac{G}{2} \, e^{-2t}.
    \end{align*}
    Therefore, for all $t \geq 0$,
    \begin{equation*}
        \left\lvert \langle \Phi^{\delta (t)}, f_1 (t) \rangle \right\rvert \leq \lVert \Phi^{\delta (t)} \rVert_{\dot{W}^{1,\infty}} \, \lVert f_1 (t) \rVert_{\dot{W}^{-1,1}} \leq C_0 \, G \, e^{-2t}.
    \end{equation*}
    \item For $f_4$, we use again \eqref{FP_source_-n_est} with $n=2$:
    \begin{equation*}
        \lVert f_4(t) \rVert_{\dot{W}^{-n,1}} \leq e^{-4t} \int_0^t e^{4u} \, |g_4| \hspace{-0.5mm} \left( u, \mathbb{R}^d \right) \diff u \leq G \, e^{-4t}.
    \end{equation*}
    We conclude for this term using the fact that $\lVert \Phi^{\delta (t)} \rVert_{\dot{W}^{2,\infty}} \leq C_0 \, e^{2t}$ thanks to \eqref{cut_off_3}:
    \begin{equation*}
        \left\lvert \langle \Phi^{\delta (t)}, f_4 (t) \rangle \right\rvert \leq \lVert \Phi^{\delta (t)} \rVert_{\dot{W}^{2,\infty}} \, \lVert f_4 (t) \rVert_{\dot{W}^{-2,1}} \leq C_0 \, G \, e^{-2t}.
    \end{equation*}
    \item For $f_3$, we use the second inequality in \eqref{FP_source_-n+1_est} with $n=4$ along with the fact that $e^{-2u} (1-e^{-4u})^{-\frac{1}{2}}$ is integrable on $(0, \infty)$:
    \begin{align*}
        \hspace{-5mm} \lVert f_3 (t) \rVert_{\dot{W}^{-3,1}} \leq C_0 \, e^{-6t} \left\lVert e^{6u} \, |g_3| \hspace{-0.5mm} \left( u, \mathbb{R}^d \right) \right\rVert_{L^\infty_u} \leq C_0 \, G \, e^{-6t}, \qquad \forall t \geq 0.
    \end{align*}
    Property \eqref{cut_off_3} shows that $\lVert \phi^{\delta (t)} \rVert_{\dot{W}^{3,\infty}} \leq C_0 \, e^{4t}$, and thus
    \begin{equation*}
        \left\lvert \langle \Phi^{\delta (t)}, f_3 (t) \rangle \right\rvert \leq \lVert \phi^{\delta (t)} \rVert_{\dot{W}^{3,\infty}} \, \lVert f_3 (t) \rVert_{\dot{W}^{-3,1}} \leq C_0 \, G \, e^{-2t}.
    \end{equation*}
    \item As for $\nabla \cdot (2y f_2)$, we use \eqref{est_x_f_FP} with $n=1$, so that for all $t \geq 0$:
    \begin{align*}
        \lVert \nabla \cdot (2y f_2) \rVert_{\dot{W}^{-1,1}} = 2 \lVert |y| f_2 \rVert_{L^1} \leq C_0 \, e^{-2t} \left[ \left\lVert \int_{\mathbb{R}^d} |x| |g_2| \hspace{-0.5mm} \left( u, \diff x \right) \right\rVert_{L^{\infty}_u} + \int_0^t |g_2| \hspace{-0.5mm} \left( u, \mathbb{R}^d \right)  \right] \leq C_0 \, G \, e^{-2t},
    \end{align*}
    thanks to the fact that
    \begin{equation*}
        \int_0^t |g_2| \hspace{-0.5mm} \left( u, \mathbb{R}^d \right) \leq \left( \int_0^t \left( e^{2u} \, |g_2| \hspace{-0.5mm} \left(u, \mathbb{R}^d \right) \right)^2 \diff u \right)^{\hspace{-0.5mm} \frac{1}{2}} \left( \int_0^t e^{-4u} \diff u \right)^{\hspace{-0.5mm} \frac{1}{2}} \leq C_0 \, G.
    \end{equation*}
    Thus,
    \begin{equation*}
        \left\lvert \langle \Phi^{\delta (t)}, \nabla \cdot (2y f_2 (t)) \rangle \right\rvert \leq \lVert \Phi^{\delta (t)} \rVert_{\dot{W}^{1,\infty}} \, \lVert \nabla \cdot (2y f_2 (t)) \rVert_{\dot{W}^{-1,1}} \leq C_0 \, G \, e^{-2t}.
    \end{equation*}
    \item Lastly, we will use the decomposition used in the part 2 of Lemma \ref{main_lem_FP_est} for $\Delta f_2$: for some $S>0$, $f_2 (t) = f_2^{1,S} (t) + f_2^{2,S} (t) + f_2^{3} (t)$ and with \eqref{est_f1S}-\eqref{est_f3S} for $n=1$,
    \begin{gather*}
        \hspace{-8mm}    \lVert \Delta f_2^{1,S} (t) \rVert_{\dot{W}^{-1,1}} = \lVert f_2^{1,S} (t) \rVert_{\dot{W}^{1,1}} \leq C_0 \, e^{-2t} \left[ \frac{1}{e^{4t} -e^{4S}} - \frac{1}{e^{4t} - 1} \right]^\frac{1}{2} \left( \int_0^{S} \left( e^{2u} \, |g_2| \hspace{-0.5mm} \left( u, \mathbb{R}^d \right) \right)^2 \diff u \right)^\frac{1}{2}, \\
        \lVert \Delta f_2^{2,S} (t) \rVert_{\dot{W}^{-2,1}} = \lVert f_2^{2,S} (t) \rVert_{L^1} \leq C_0 \, e^{-4t} \left( e^{4t} - e^{4S} \right)^\frac{1}{2} \left\lVert |g_2| \hspace{-0.5mm} \left( u, \mathbb{R}^d \right) \right\rVert_{L^\infty_u}, \\
        \lVert \Delta f_2^{3} (t) \rVert_{\dot{W}^{-1,1}} = \lVert f_2^{3} (t) \rVert_{\dot{W}^{1,1}} \leq C_0 \, e^{-2t} \left( \int_{0}^t \left( e^{2u} \, |g_2| \hspace{-0.5mm} \left( u, \mathbb{R}^d \right) \right)^2 \diff u \right)^\frac{1}{2}.
    \end{gather*}
    Therefore, those estimates yield
    \begin{align*}
        \left\lvert \langle \Phi^{\delta (t)}, \Delta f_2 (t) \rangle \right\rvert &\leq \lVert \Phi^{\delta (t)} \rVert_{\dot{W}^{1,\infty}} \left[ \lVert \Delta f_2^{1,S} (t) \rVert_{\dot{W}^{-1,1}} + \lVert \Delta f_2^{3,S} (t) \rVert_{\dot{W}^{-1,1}} \right] + \lVert \Phi^{\delta (t)} \rVert_{\dot{W}^{2,\infty}} \, \lVert \Delta f_2^{2,S} (t) \rVert_{\dot{W}^{-2,1}} \\
        &\leq C_0 \, G \, e^{-2t} \left[ (e^{4t} - e^{4S})^{-\frac{1}{2}} + (e^{4t} - e^{4S})^{\frac{1}{2}} \right].
    \end{align*}
    The convergence rate comes by optimizing over $S$, which is taking $S = S(t)$ such that $e^{4t} - e^{4S} = 1$.
\end{itemize}
Putting all together, we finally get that
\begin{equation*}
    \left\lvert \int_{\mathbb{R}^d} \Phi^{\delta (t)} \, (f (t) - \pi^{-\frac{d}{2}} \gamma^2) \diff y \right\rvert \leq C_0 \, (1 + G + \lVert \, |y|^2 f \, \rVert_{L^\infty_t L^1_y}) \, e^{-2t},
\end{equation*}
and the result by putting it back in \eqref{est_diff_Phi_delta_t}.
\hfill $\square$

\subsection{Proof of Corollaries \ref{cor_other_conv_log_nls} and \ref{cor_other_conv_wigner}}

The only thing that remains is the proof of Corollaries \ref{cor_other_conv_log_nls} and \ref{cor_other_conv_wigner}. Like already said, the convergence rate in $(1+\delta)$-Wasserstein distance for $\delta \in (0,1)$ (in both Corollaries) follows from a simple Hölder inequality and the bounds of the second momentum of both $| v_\varepsilon (t) |^2$ and $\tilde{\rho} (t)$ found in Theorem \ref{main_th_log_nls_eps} and \ref{main_th_wigner} respectively. On the other hand, the convergence rate in $W^{-1+\delta, 1}$ can be proved through the following lemma and the inequality $\lVert \, \rVert_{\dot{W}^{-1,1}} \leq \mathcal{W}_1$.
\begin{lem}
    Given any $\delta \in (0,1)$ and any $f \in L^1 (\mathbb{R}^d)$, there holds
    \begin{equation*}
        \lVert f \rVert_{\dot{W}^{-1+\delta, 1}} \leq C_0 \, \lVert f\rVert_{\dot{W}^{-1,1}}^{1-\delta} \, \lVert f\rVert_{L^1}^\delta.
    \end{equation*}
\end{lem}

\begin{proof}
    Let $g \in \dot{W}^{1-\delta, \infty} (\mathbb{R}^d) = \mathcal{C}^{0,1-\delta} (\mathbb{R}^d)$ and define $g_\eta = g * \gamma_\eta$ where $\gamma_\eta = \gamma_\eta (x) = (\pi \eta)^{-\frac{d}{2}} e^{-\frac{|x|^2}{\eta}}$ for all $\eta > 0$. Then, for any $x \in \mathbb{R}^d$,
    \begin{equation*}
        | g(x) - g_\eta (x) | \leq \int_{\mathbb{R}^d} | g(x) - g(x-y) | \, \gamma_\eta (y) \diff y \leq \int_{\mathbb{R}^d} \lVert g \rVert_{\mathcal{C}^{0,1-\delta}} \, | y |^{1-\delta} \, \gamma_\eta (y) \diff y \leq \eta^{1-\delta} \, \lVert g \rVert_{\mathcal{C}^{0,1-\delta}} \lVert \, |.|^{1-\delta} \, \gamma \, \rVert_{L^1}.
    \end{equation*}
    Moreover, $\nabla g_\eta = g * \nabla \gamma_\eta$, so that for any $x \in \mathbb{R}^d$,
    \begin{align*}
        | \nabla g_\eta (x) | = \left\lvert \int_{\mathbb{R}^d} (g(x-y) - g(x)) \nabla \gamma_\eta (y) \diff y \right\rvert \leq \lVert g \rVert_{\mathcal{C}^{0,1-\delta}} \int_{\mathbb{R}^d} | y |^{1-\delta} \, | \nabla \gamma_\eta (y) | \diff y \leq \eta^{-\delta} \, \lVert g \rVert_{\mathcal{C}^{0,1-\delta}} \lVert \, |.|^{1-\delta} \nabla \gamma \, \rVert_{L^1}.
    \end{align*}
    Therefore, we get for all $\eta > 0$:
    \begin{align*}
        \left\lvert \int_{\mathbb{R}^d} f(x) \, g(x) \diff x \right\rvert &\leq \left\lvert \int_{\mathbb{R}^d} f(x) \, (g(x) - g_\eta (x)) \diff x \right\rvert + \left\lvert \int_{\mathbb{R}^d} f(x) \, g_\eta(x) \diff x \right\rvert \\
        &\leq \lVert f \rVert_{L^1} \, \lVert g - g_\eta \rVert_{L^\infty} + \lVert f\rVert_{\dot{W}^{-1,1}} \, \lVert \nabla g_\eta \rVert_{L^\infty} \leq C_0 \, \lVert g \rVert_{\dot{W}^{1-\delta, \infty}} ( \eta^{1-\delta} \, \lVert f\rVert_{L^1} + \eta^{-\delta} \, \lVert f\rVert_{\dot{W}^{-1,1}}),
    \end{align*}
    which yields
    \begin{equation*}
        \lVert f \rVert_{\dot{W}^{-1+\delta, 1}} \leq C_0 \, ( \eta^{1-\delta} \, \lVert f\rVert_{L^1} + \eta^{-\delta} \, \lVert f\rVert_{\dot{W}^{-1,1}}),
    \end{equation*}
    and the result by optimizing in $\eta$.
\end{proof}

\section{Kinetic Isothermal Euler system}
\label{section_KIE}

\subsection{Discussion on its formal properties}
\label{subsec_disc_KIE}

We recall the Kinetic Isothermal Euler system \eqref{def_KIE}:
\begin{align*}
    \partial_t f + \xi \cdot \nabla_x f - \lambda \, \nabla_x \left( \ln{\rho} \right) \cdot \nabla_\xi f = 0,
\end{align*}
where $\lambda > 0$ and $\rho (t,x) = \int f(t,x,\diff \xi)$.
A solution $f=f(t,x,\xi)$ of such a Vlasov equation should be a non-negative measure in $x$ and $\xi$ for every (or a.e.) $t$.

This equation is a non-linear Vlasov-type equation with potential $\ln \rho$. In particular, it is a transport equation with null-divergence transport. The formal properties of this kind of equations should be guaranteed, i.e. the conservation of the mass and the energy like for the Schrödinger equation:
\begin{gather*}
    \frac{\diff}{\diff t} \left( \iint_{\mathbb{R}^d \times \mathbb{R}^d} f(t, \diff x, \diff \xi) \right) = 0, \\
    \frac{\diff }{\diff t} \left( \frac{1}{2} \iint_{\mathbb{R}^d \times \mathbb{R}^d} |\xi|^2 \, f(t, \diff x, \diff \xi) + \lambda \int_{\mathbb{R}^d} \rho (t,x) \ln \rho (t,x) \diff x \right) = 0.
\end{gather*}
The second equation is very interesting. Indeed, it transforms the highly singular non-linearity of the equation \eqref{def_KIE} $\nabla_x \left( \ln{\rho} \right)$ into better suited properties on $\rho$. Moreover, if we want $\int_{\mathbb{R}^d} \rho (t) \ln \rho (t) \diff x$ to be well-defined, we shall require $\rho$ to be in $L \log L$ and in particular in $L^1$, which is similar to the previous properties found for the Wigner Measure.
Furthermore, we should also have some other (formal) properties coming from (formal) computations, for example for $\rho$ or also for $J (t,x) := \int_{\mathbb{R}^d} \xi \, f(t,x,\diff \xi)$:
\begin{gather}
    \partial_t \rho (t,x) + \nabla_x \cdot J \, (t, x) = 0, \label{eq_rho_KIE} \\
    \partial_t J (t,x) + \nabla_x \cdot \int_{\mathbb{R}^d} \xi \otimes \xi \, f(t, x, \diff \xi) + \lambda \, \nabla_x \rho (t,x) = 0. \label{eq_xi_f_KIE}
\end{gather}
Those two equations look like \eqref{iso_eul_sys}.
In particular, if we consider time-dependent mono-kinetic solutions to \eqref{def_KIE}, then \eqref{eq_rho_KIE} and \eqref{eq_xi_f_KIE} give exactly \eqref{iso_eul_sys}.
Furthermore, they yield
\begin{gather*}
    \frac{\diff }{\diff t} \left( \iint_{\mathbb{R}^d \times \mathbb{R}^d} |x|^2 \, f(t, \diff x, \diff \xi) \right) = 2 \iint_{\mathbb{R}^d \times \mathbb{R}^d} x \cdot \xi \, f(t, \diff x, \diff \xi), \label{cons_mass_KIE} \\
    \frac{\diff }{\diff t} \left( \iint_{\mathbb{R}^d \times \mathbb{R}^d} x \cdot \xi \, f(t, \diff x, \diff \xi) \right) = \iint_{\mathbb{R}^d \times \mathbb{R}^d} |\xi|^2 \, f(t, \diff x, \diff \xi) + \lambda \int_{\mathbb{R}^d} \rho(t,x) \diff x. \label{cons_en_KIE}
\end{gather*}
All those properties are totally formal. However, a good framework for \eqref{def_KIE} should get those properties, which means that all those terms should be well-defined (in some sense). Thus, intuitively, the solution $f$ should be at least in $L^\infty_\text{loc} ((0, \infty), \mathcal{M} \Sigma_{\log} \cap \mathcal{M}_2)$ where:
\begin{align*}
    \mathcal{M}\Sigma_{\log} &= \{\mu \in \mathcal{M}(\mathbb{R}^d_x \times \mathbb{R}^d_\xi), \int_{\mathbb{R}^d_\xi} \mu(x, \diff \xi) \in L^1 \cap L \, \log L \, (\mathbb{R}^d) \}, \\
    \mathcal{M}_2 &= \{\mu \in \mathcal{M}(\mathbb{R}^d_x \times \mathbb{R}^d_\xi), \iint_{\mathbb{R}^d \times \mathbb{R}^d} (|x|^2 + |\xi|^2) \diff \mu < \infty \} .
\end{align*}
Again, from \eqref{eq_rho_KIE} and \eqref{eq_xi_f_KIE}, we can also prove some continuity for $\rho$ and $J$. Indeed, \eqref{eq_rho_KIE} implies that $\partial_t \rho \in L^\infty_\text{loc} \left((0, \infty), W^{-1-\delta, 1} (\mathbb{R}^d) \right)$ for all $\delta > 0$ uniformly in $\delta$. Since $\rho \in L^\infty_\text{loc} \left((0, \infty), L^1_2 \cap L \, \log L (\mathbb{R}^d) \right)$, the previous property leads to $\rho \in W^{1, \infty}_\text{loc} \left((0, \infty), W^{-1, 1} (\mathbb{R}^d) \right)$ and also $\pi^{-\frac{d}{2}} \, \rho \in \mathcal{C} \left((0, \infty), \mathcal{P}_1 (\mathbb{R}^d) \right)$. As for $J$, similar arguments as in Remark \ref{rem_continuity_J_0} apply and lead to $J \in \mathcal{C}^{0,1}_\text{loc} \left((0, \infty), W^{-1, 1} (\mathbb{R}^d)^d \right) \cap \mathcal{C} \left((0, \infty), \mathcal{M}^s (\mathbb{R}^d)^d \right)$.

Actually, \eqref{eq_rho_KIE} and \eqref{eq_xi_f_KIE} are very similar to \eqref{eq_schr_for_u_eps}. Moreover, we also have conservation of the mass and the energy similar to those for the logarithmic Schrödinger equation. Finally, we have seen that the rescaling \eqref{rescaling} is translated into the identity \eqref{wigner_rel}, therefore it is natural to consider a rescaling for the solution of the Kinetic Isothermal Euler system to $\tilde{f} = \tilde{f} (t,y,\eta)$ defined by:
\begin{equation*}
    f (t, x, \xi) = \frac{f_0 (\mathbb{R}^d \times \mathbb{R}^d)}{\lVert \gamma^2 \rVert_{L^1}} \, \Tilde{f} \left(t, \frac{x}{\tau (t)}, \tau(t) \, \xi - \Dot{\tau}(t) \, x \right).
\end{equation*}
Thus, we can perform arguments similar to that in the proof of Theorem \ref{main_th_log_nls_eps}:
\begin{itemize}
    \item We define the density of particles and the density of angular momentum:
    \begin{gather*}
        \tilde{\rho} (t,y) := \int_{\mathbb{R}^d} \tilde{f}(t,y,\diff \eta) \in L^\infty_{loc} \left( (0, \infty), L^1_2 \cap L \, \log L \, (\mathbb{R}^d) \right), \\
        \tilde{J} (t,y) := \int_{\mathbb{R}^d} \eta \, \tilde{f}(t,y,\diff \eta) \in L^\infty_{loc} \left( (0, \infty), \mathcal{M}^s_1 (\mathbb{R}^d)^d \right),
    \end{gather*}
    where $\mathcal{M}^s_1 (\mathbb{R}^d)$ is the set of signed finite measure with bounded first momentum.
    \item We also define the modified kinetic energy, the relative entropy and the modified total energy:
    \begin{gather*}
        \mathcal{E}_\text{kin} (t) := \frac{1}{2 \, \tau (t)^2} \iint_{\mathbb{R}^d \times \mathbb{R}^d} |\eta|^2 \, \tilde{f} (t, \diff y, \diff \eta), \qquad
        \mathcal{E}_\text{ent} (t) := \int_{\mathbb{R}^d} \tilde{\rho} (t,y)\ln \frac{\rho (t,y)}{\gamma^2 (y)} \diff y, \\
        \mathcal{E} := \mathcal{E}_\text{kin} + \lambda \, \mathcal{E}_\text{ent}.
    \end{gather*} 
    \item Then, in the same way as in Remark \ref{rem_proof_eq_log_nls}, there holds
    \begin{gather}
        \dot{\mathcal{E}} = - 2 \frac{\dot{\tau} (t)}{\tau (t)} \, \mathcal{E}_\text{kin}, \notag \\
        \partial_t \tilde{\rho} + \frac{1}{\tau^2 (t)} \nabla \cdot \tilde{J} = 0, \qquad
        \partial_t \tilde{J} + \lambda \, \nabla \tilde{\rho} + 2 \lambda \, y \, \tilde{\rho} = - \frac{1}{\tau^2(t)} \nabla \cdot \int_{\mathbb{R}^d} \eta \otimes \eta \, \tilde{f}(t, y, \diff \eta), \qquad \text{in } \mathcal{D}'. \label{eq_rho_tilde_KIE}
    \end{gather}
    \item Write
    \begin{gather*}
        \mathcal{E}_+ := \mathcal{E}_{\text{kin}} + \lambda \int_{\tilde{\rho}>1} \tilde{\rho} \ln \tilde{\rho} + \lambda \int |y|^2 \, \tilde{\rho} \geq 0, \qquad \mathcal{E}_- := - \lambda \int_{\tilde{\rho}<1} \tilde{\rho} \ln \tilde{\rho} \geq 0,
    \end{gather*}
    so that
    \begin{gather*}
        \mathcal{E} = \mathcal{E}_+ - \mathcal{E}_- \leq \mathcal{E} (0), \qquad \mathcal{E}_- \leq C_0 \, \left( \mathcal{E}_+ \right)^\frac{d}{2 (d+2)}.
    \end{gather*}
    Similar arguments as in Lemma \ref{lem_en_est_schr} apply to this case, showing that $\mathcal{E}_+$ is bounded which leads to the estimates
    \begin{gather*}
        \int_{\mathbb{R}^d} \left( 1 + |y|^2 + \left\lvert \ln \tilde{\rho} \right\rvert \right) \, \tilde{\rho} \diff y + \frac{1}{\tau (t)^2} \iint_{\mathbb{R}^d \times \mathbb{R}^d} |\eta|^2 \, \tilde{f} (t, \diff y, \diff \eta) \leq C_0, \qquad \forall t \geq 0, \\
        \int_0^\infty \frac{\dot{\tau} (t)}{\tau (t)^3} \iint_{\mathbb{R}^d \times \mathbb{R}^d} |\eta|^2 \, \tilde{f} (t, \diff y, \diff \eta) \leq C_0.
    \end{gather*}
    \item Those estimates along with the system \eqref{eq_rho_tilde_KIE} show that we can apply Lemma \ref{2nd_lem_FP_wass} with (up to a factor $\pi^{-\frac{d}{2}}$) $f = \tilde{\rho}$, $h_1 = \tau (t)^{-1} \tilde{J}$, $h_2 = 0$ and $h_3 = \frac{1}{\tau (t)^2} \int_{\mathbb{R}^d} \eta \otimes \eta \, \tilde{f} (t,y, \diff \eta)$. Therefore, we get in a similar way:
    \begin{gather*}
        \mathcal{W}_1 \left( \pi^{-\frac{d}{2}} \, \tilde{\rho} (t), \pi^{-\frac{d}{2}} \, \gamma^2 \right) \leq \frac{C_0}{\sqrt{\ln t}} \qquad \forall t \geq 2.
    \end{gather*}
    \item Introducing
    \begin{equation*}
        I_1 (t) = \int_{\mathbb{R}^d} \tilde{J} (t, \diff y), \qquad I_2(t) = \int_{\mathbb{R}^d} y \, \tilde{\rho} (t,y) \diff y, \qquad \Tilde{I_2} = \tau \, I_2,
    \end{equation*}
    computations similar to those in the proof of Theorem \ref{main_th_log_nls_eps} yield
    \begin{gather*}
        \Dot{I_1} = -2 \lambda I_2, \qquad \Dot{I_2} = \frac{1}{\tau^2 (t)} I_1, \qquad \Ddot{\Tilde{I_2}} = 0, \qquad I_2 (t) = \frac{1}{\tau (t)} \left( I_1 (0) \, t + I_2 (0) \right) \underset{t \rightarrow \infty}{\longrightarrow} 0 = \int y \, \gamma^2 (y) \diff y.
    \end{gather*}
    Moreover, as soon as $I_1 (0) \neq 0$, there holds
    \begin{equation*}
         I_2 (t) \underset{t \rightarrow \infty}{\sim} \frac{I_1 (0)}{2 \sqrt{\lambda \ln t}}.
    \end{equation*}
    In the same way, from the conservation of the energy for $f$ by translating it into estimates on $\tilde{f}$, we derive for all $t \geq 2$:
    \begin{equation*}
        \left\lvert \int_{\mathbb{R}^d} |y|^2 \, \tilde{\rho} (t,y) \diff y - \int_{\mathbb{R}^d} |y|^2 \, \gamma^2 (y) \diff y \right\rvert \leq \frac{C_0}{\sqrt{\ln t}}.
    \end{equation*}
\end{itemize}

It is interesting to see that the Wigner Measure found in Theorem \ref{main_th_wigner} satisfy most of those properties. The only thing we could not prove is the convergence of the second momentum of the density, pointed out in Remark \ref{rem_conv_2nd_mom_density_WM}. If a good framework were found for \eqref{def_KIE} and if we could show the fact that the Wigner Measure satisfy \eqref{def_KIE} in this sense, we would (probably) be able to prove also the convergence of this momentum.

\begin{rem}
    $\nabla_x (\ln{\rho} (t))$ is actually weakly defined $\rho (t)$-a.e.: indeed, for every $\phi \in W^{1,\infty}(\mathbb{R}^d)$,
    \begin{equation*}
        \int \nabla_x (\ln{\rho}) (t, .) \, \phi \, d\rho(t) = - \int \rho (t, x) \, \nabla \phi (x) \, dx = - \int \nabla \phi \, d\rho(t).
    \end{equation*}
    In the same way, the term $\nabla_x ( \ln \rho ) \cdot \nabla_\xi f$ is weakly well-defined as soon as $\rho(t) \in W^{1,1}$ because for every $\phi \in L^\infty (\mathbb{R}^d_x, W^{1,\infty}(\mathbb{R}^d_\xi))$
    \begin{align*}
        \langle \nabla_x ( \ln \rho ) (t, x) \cdot \nabla_\xi f (t, x, \xi), \phi(x,\xi) \rangle_{(x,\xi)} &= \langle \nabla_x ( \ln \rho ) (t, x) \, f (t, x, \xi), \nabla_\xi \phi(x,\xi) \rangle_{(x,\xi)} \\
        &= \Big \langle \nabla_x ( \ln \rho ) (t, x),  \big\langle f (t, x, \xi), \nabla_\xi \phi(x,\xi) \big \rangle_\xi \Big \rangle_{x},
    \end{align*}
    with the last term well-defined because:
    \begin{align*}
        \iint_{\mathbb{R}^d \times \mathbb{R}^d} \left\lvert \nabla_x ( \ln \rho ) (t, x) \, f (t, x, \xi) \cdot \nabla_\xi \phi(x,\xi) \right\rvert \, dx d\xi &\leq \int_{\mathbb{R}^d_x} \left\lvert \nabla_x ( \ln \rho ) (t, x) \right\rvert \int_{\mathbb{R}^d_\xi} \left\lvert f (t, x, \xi) \, \nabla_\xi \phi(x,\xi) \right\rvert \, d\xi \, dx \\
        &\leq \int_{\mathbb{R}^d_x} \left\lvert \nabla_x ( \ln \rho ) (t, x) \right\rvert \int_{\mathbb{R}^d_\xi} f (t, x, \xi) \left\lVert \nabla_\xi \phi \right\rVert_{L^\infty} \, d\xi \, dx \\
        &\leq \int_{\mathbb{R}^d_x} \left\lvert \nabla_x ( \ln \rho ) (t, x) \right\rvert \rho (t, x) \left\lVert \nabla_\xi \phi \right\rVert_{L^\infty} \, dx \\
        &\leq \int_{\mathbb{R}^d_x} \left\lvert \nabla_x \rho (t, x) \right\rvert \, \left\lVert \nabla_\xi \phi \right\rVert_{L^\infty} \, dx \, < \infty.
    \end{align*}

Such remarks might help in order to find a real formalization of the equation, but this is not our goal here. However, we could not prove any $W^{1,1}$ regularity for $\rho$, whether for the Wigner Measure or with an estimate in the previous discussion.
\end{rem}

\subsection{Explicit solutions} \label{sec_explicit_sol_KIE}

Actually, there exists a particular case in which the Wigner Measure can be computed explicitly and is a solution to \eqref{def_KIE}: the Gaussian case, providing \textit{Gaussian-monokinetic} solutions to \eqref{def_KIE}. It happens when all the initial data for \eqref{log_nls_eps} are Gaussian up to a quadratic complex oscillation. This result was proved by R. Carles and A. Nouri:

\begin{theorem}[{\cite[Theorem~1.1.]{carlesnouri}} and its proof]
\label{th_gauss_wigner}
Let $\lambda$, $\rho_*$, $\sigma_0 > 0$ and $\omega_0$, $p_0 \in \mathbb{R}$. Set
\begin{equation*}
    \rho_{\textnormal{in}}(x) = \rho_* e^{-\sigma_0 x^2}, \qquad \phi_{\textnormal{in}} = \omega_0 \frac{x^2}{2} + p_0 x, \qquad v_{in}(x) = \phi_{\textnormal{in}}' (x),
\end{equation*}
and consider the solution $\tau_0 \in \mathcal{C}^\infty(\mathbb{R}^+)$ to the ordinary differential equation
\begin{equation*}
    \Ddot{\tau}_0 = \frac{2 \lambda \sigma_0}{\tau_0}, \qquad \tau_0 (0) = 1, \qquad \Dot{\tau}_0 (0) = \omega_0.
\end{equation*}
Set
\begin{equation*}
    \rho (t,x) = \frac{\rho_*}{\tau_0(t)} e^{-\sigma_0 \frac{(x - p_0 t)^2}{\tau_0(t)^2}}, \qquad v(t,x) = \frac{\Dot{\tau}_0 (t)}{\tau_0(t)} (x - p_0 t) + p_0
\end{equation*}
and consider $u_\varepsilon$ the solution to \eqref{log_nls_eps} with initial data
\begin{equation*}
    u_{\varepsilon,\textnormal{in}} (x) = \sqrt{\rho_{\textnormal{in}} (x)} \, e^{i \, \frac{\phi_{\textnormal{in}}(x)}{\varepsilon}} \in \mathcal{F}(H^1) \cap H^1 (\mathbb{R}),
\end{equation*}
provided by Theorem \ref{th_cauchy_log_nls_eps}.
Then the Wigner Transform $W_\varepsilon (t)$ of $(u_\varepsilon (t))_{\varepsilon>0}$ weakly converges (in terms of measures) when $\varepsilon \rightarrow 0$ for all $t \geq 0$ to the finite measure
\begin{equation*}
    W (t, \diff x, \diff \xi) = \rho(t,x) \diff x \otimes \delta_{\xi = v(t,x)},
\end{equation*}
solution to \eqref{def_KIE} with $W (0,\diff x,\diff \xi) = \rho_{\textnormal{in}}(x) \diff x \otimes \delta_{\xi = v_{in}(x)}$ because $(\rho, v)$ is solution to \eqref{iso_eul_sys}.
\end{theorem}

The proof relies on the fact that the solution $u_{\varepsilon}$ to \eqref{log_nls_eps} can actually be computed explicitly in this case, the Wigner Measure then readily follows from some computations. It is interesting to see that the initial data is a WKB state, satisfying \eqref{WKB_state_assumption}, and for this case this feature still holds for all time, recovering a (time-dependent) monokinetic measure for the Wigner Measure. Moreover, another interesting feature is the fact that the density (either for $u_\varepsilon$ or for the Wigner Measure $W$) never vanishes, and even more: $\nabla ( \log \rho_\varepsilon )$ is actually well defined for $\varepsilon \geq 0$ as an affine function in $x$, which is why we can say that this measure is solution to \eqref{def_KIE}. Such a feature is very exceptional and cannot be extended to the general case, in particular for \eqref{log_nls_eps}. However, we extend this class of solutions to \eqref{def_KIE} with a new class of explicit solutions, which are Gaussian in $x$ multiplied by an $x$-dependent Gaussian in $\xi$ for all time, stated in Theorem \ref{main_th_gauss}. We call them \textit{Gaussian-Gaussian} solutions, by opposition with the previous Gaussian-monokinetic solutions.

\subsection{Proof of Theorem \ref{main_th_gauss}}

The main step of this proof is to prove the part \ref{part_2_main_th_gauss} of Theorem \ref{main_th_gauss}. Indeed, the computations that will be done can be done reversely, or in another way one can prove directly by some easy computations that \eqref{expr_gauss_gauss_sol_KIE} is a solution to \eqref{def_KIE}. We must also prove that $c_1$ solution to \eqref{gaussian1_th} is $C^\infty (\mathbb{R}^+)$, but this has already been done in \cite{carlesgallagher}.

With the notations and assumptions of the part \ref{part_2_main_th_gauss} of Theorem \ref{main_th_gauss}, we compute:
\begin{align*}
    \partial_t f (t,x,\xi) &=
    \begin{multlined}[t][13cm]
        \left[ - \frac{\Dot{c}_1(t)}{c_1(t)} - \frac{\partial_t c_2 (t,x)}{c_2 (t,x)} + 2 \, \frac{\Dot{c}_1(t) \, (x - b_1(t))^2}{c_1(t)^3} + 2 \, \frac{\Dot{b}_1 (t) \cdot (x - b_1 (t))}{c_1 (t)^2} \right. \\
        \left. + 2 \, \frac{\partial_t c_2 (t,x) \, (\xi - b_2 (t,x))^2}{c_2(t,x)^3} + 2 \, \frac{\partial_t b_2 (t,x) \cdot (\xi - b_2 (t,x))}{c_2(t,x)^2} \right] f(t,x,\xi),
    \end{multlined} \\
    \partial_x f (t,x,\xi) &= \left[ - 2 \, \frac{x - b_1(t)}{c_1(t)^2} + 2 \, \partial_x b_2 (t,x) \frac{\xi - b_2(t,x)}{c_2(t,x)^2} + 2 \, \partial_x c_2 (t,x) \frac{(\xi - b_2(t,x))^2}{c_2(t,x)^3} \right] f(t,x,\xi), \\
    \partial_\xi f (t,x,\xi) &= - 2 \, \frac{\xi - b_2 (t,x)}{c_2(t,x)^2} f(t,x,\xi).
\end{align*}
We also obviously get $\rho (t,x) = \frac{1}{\sqrt{\pi} \, c_1(t)} e^{- \frac{(x - b_1(t))^2}{c_1(t)^2}}$, therefore it is easy to compute:
\begin{equation*}
    \partial_x ( \ln \rho ) (t,x) = - 2 \, \frac{x - b_1 (t)}{c_1(t)}.
\end{equation*}

Plugging all those identities into \eqref{def_KIE} leads to 
\begin{multline*}
    0 = \left[ - \frac{\Dot{c}_1(t)}{c_1(t)} - \frac{\partial_t c_2 (t,x)}{c_2 (t,x)} + 2 \, \frac{\Dot{c}_1(t) \, (x - b_1(t))^2}{c_1(t)^3} + 2 \, \frac{\Dot{b}_1 (t) \cdot (x - b_1 (t))}{c_1 (t)^2} + 2 \, \frac{\partial_t c_2 (t,x) \, (\xi - b_2 (t,x))^2}{c_2(t,x)^3} \right. \\
    \begin{aligned}
        &+ 2 \, \frac{\partial_t b_2 (t,x) \cdot (\xi - b_2 (t,x))}{c_2(t,x)^2} - 2 \, \xi \, \frac{x - b_1(t)}{c_1(t)^2} + 2 \, \partial_x b_2 (t,x) \frac{(\xi - b_2(t,x)) \, \xi}{c_2(t,x)^2} + 2 \, \partial_x c_2 (t,x) \frac{\xi \, (\xi - b_2(t,x))^2}{c_2(t,x)^3} \\ &\left. - 4 \lambda \, \frac{\xi - b_2 (t,x)}{c_2(t,x)^2} \, \frac{x - b_1 (t)}{c_1(t)} \right] f(t,x,\xi),
    \end{aligned}
\end{multline*}
which is of the form
\begin{equation*}
    P(t,x,\xi) f(t,x,\xi) = 0.
\end{equation*}
where $P$ is a function such that for every $(t,x)$, $P(t,x,.)$ is polynomial of degree at most 3. Since $f (t,x,\xi) > 0$ for every $(t,x,\xi)$, there holds $P = 0$ and therefore for every $(t,x)$, the coefficients of the polynomial function $P(t,x,.)$ are zero. In particular, the coefficient of highest degree is $2 \, \partial_x c_2 (t,x) \, c_2(t,x)^{-3}$, which yields
\begin{equation*}
    \partial_x c_2 (t,x) = 0, \qquad \text{ for all } t \in [0, T), x \in \mathbb{R},
\end{equation*}
and thus $c_2$ does not depend on $x$. We now take a more suitable basis to get zero coefficients for the polynomial function $\xi \mapsto P(t,x,\xi)$ of degree at most 2: $\big ((\xi - b_2(t,x))^2, \xi - b_2(t,x), 1 \big )$. Again, the coefficients in this basis are all zero, which yields for $(\xi - b_2(t,x))^2$:
\begin{equation*}
    2 \frac{\Dot{c}_2 (t)}{c_2 (t)^3} + 2 \frac{\partial_x b_2 (t,x)}{c_2(t,x)^2} = 0.
\end{equation*}
This equation leads to
\begin{equation*}
    \partial_x b_2 (t,x) = - \frac{\Dot{c}_2 (t)}{c_2(t)},
\end{equation*}
and then, there exists a function $p_0 = p_0 (t)$ such that:
\begin{equation}
    b_2 (t,x) = - \frac{\Dot{c}_2 (t)}{c_2(t)} \, x + p_0 (t), \qquad \text{for all } t \in [0, T), x \in \mathbb{R}. \label{expression_b2}
\end{equation}
The assumption on the regularity of $b_2$ shows that $p_0 \in \mathcal{C}^1([0,T))$. But then, we also get:
\begin{equation*}
    \Dot{c}_2 (t) = c_2 (t) \, ( p_0(t) - b_2(t, 1)) \in \mathcal{C}^1 ([0,T)).
\end{equation*}
Therefore, $c_2 \in \mathcal{C}^2 ([0,T))$.
Now, examining the coefficient for $(\xi - b_2(t,x))$, we get
\begin{equation*}
    2 \, \frac{\partial_t b_2 (t,x)}{c_2(t)^2} - 2 \, \frac{x - b_1(t)}{c_1(t)^2} + 2 \, \frac{b_2 (t,x) \, \partial_x b_2 (t,x)}{c_2 (t)^2} - 4 \lambda \, \frac{x - b_1(t)}{c_1(t)^2 \, c_2(t)^2} = 0, \qquad \text{for all } (t,x).
\end{equation*}
In terms of $\partial_t b_2$, this reads
\begin{align*}
    \partial_t b_2 (t,x) &= \left( 1 + \frac{2 \lambda}{c_2 (t)^2} \right) \frac{c_2 (t)^2}{c_1 (t)^2} \, (x - b_1(t)) - b_2 (t,x) \, \partial_x b_2(t,x) \\
    &= \left[ \left( 1 + \frac{2 \lambda}{c_2 (t)^2} \right) \frac{c_2 (t)^2}{c_1 (t)^2} - \frac{\Dot{c}_2 (t)^2}{c_2(t)^2} \right] \, x - \left( 1 + \frac{2 \lambda}{c_2 (t)^2} \right) \frac{c_2 (t)^2}{c_1 (t)^2} \, b_1(t) + \frac{\Dot{c}_2 (t)}{c_2(t)} p_0 (t).
\end{align*}
However, differentiating \eqref{expression_b2} with respect to $t$ gives:
\begin{align*}
    \partial_t b_2 (t,x) = \left( - \frac{\Ddot{c}_2 (t)}{c_2 (t)} + \frac{\Dot{c}_2 (t)^2}{c_2 (t)^2} \right) x + \Dot{p}_0 (t).
\end{align*}
This yields the following system of equations for all $t \geq 0$:
\begin{gather}
  \left( 1 + \frac{2 \lambda}{c_2 (t)^2} \right) \frac{c_2 (t)^2}{c_1 (t)^2} = - \frac{\Ddot{c}_2 (t)}{c_2 (t)} + 2 \, \frac{\Dot{c}_2 (t)^2}{c_2 (t)^2}, \label{1st_eq_sys_gauss} \\
  \Dot{p}_0 (t) = - \left( 1 + \frac{2 \lambda}{c_2 (t)^2} \right) \frac{c_2 (t)^2}{c_1 (t)^2} \, b_1(t) + \frac{\Dot{c}_2 (t)}{c_2(t)} p_0 (t). \notag
\end{gather}
In particular, the second equation shows that $\Dot{p}_0 \in \mathcal{C}^1 ([0,T))$ (since the right-hand side is) and is actually an ordinary differential equation of order 1. The solution is well-known as soon as we remark that $\frac{\Dot{c}_2}{c_2} = \frac{\text{d}}{\text{d}t} ( \ln c_2)$ and reads:
\begin{equation*}
    p_0 (t) = c_2 (t) \left( C_0 - \int_0^t \left( 1 + \frac{2 \lambda}{c_2 (s)^2} \right) \frac{c_2 (s)^2}{c_1 (s)^2} \, \frac{b_1(s) }{c_2 (s)} \diff s \right)
\end{equation*}
and thanks to \eqref{1st_eq_sys_gauss}, we can expand it:
\begin{align*}
    p_0 (t) &= c_2 (t) \left( C_0 - \int_0^t \left( - \frac{\Ddot{c}_2 (s)}{c_2 (s)} + 2 \, \frac{\Dot{c}_2 (s)^2}{c_2 (s)^2} \right) \frac{b_1(s) }{c_2 (s)} \diff s \right) \\
    &= c_2 (t) \left( C_0 + \int_0^t \frac{\text{d}^2}{\text{d}s^2} \left( \frac{1}{c_2 (s)} \right) b_1(s) \diff s \right) \\
    &= c_2 (t) \, C_1 + \frac{\Dot{c}_2 (t)}{c_2 (t)} \, b_1 (t) - c_2 (t) \int_0^t \frac{\Dot{c}_2 (s)}{c_2 (s)^2} \, \Dot{b}_1(s) \diff s ,
\end{align*}
where $C_1 = C_0 - \frac{\Dot{c}_2 (0)}{c_2 (0)^2} \, b_1 (0)$ with an integration by parts.
Last, the constant in $\xi$ gives the following equation:
\begin{equation*}
    - \frac{\Dot{c}_1 (t)}{c_1 (t)} - \frac{\Dot{c}_2 (t)}{c_2 (t)} + 2 \, \frac{\Dot{c}_1 (t)}{c_1 (t)^3} \, (x - b_1(t))^2 + 2 \frac{\Dot{b}_1 (t)}{c_1(t)^2} \, (x - b_1 (t)) - 2 \, \frac{b_2 (t,x)}{c_1 (t)^2} (x - b_1(t)) = 0.
\end{equation*}
But since we know that $b_2$ is affine in $x$, the left-hand side is a polynomial function in $x$ of degree 2 for all $t \in [0,T)$. Therefore, the coefficients in every basis are null. This time, we take the basis: $\big ( (x - b_1(t))^2, x-b_1(t), 1 \big)$. For the first one and for the constants, we get
\begin{gather*}
    2 \, \frac{\Dot{c}_1 (t)}{c_1 (t)^3} + 2 \, \frac{\Dot{c}_2 (t)}{c_2 (t) \, c_1(t)^2} = 0, \qquad \text{and} \qquad
    - \frac{\Dot{c}_1 (t)}{c_1 (t)} - \, \frac{\Dot{c}_2 (t)}{c_2 (t)} = 0.
\end{gather*}
Those 2 equations actually reduce in a single one,  which is
\begin{equation*}
    \frac{\diff }{\diff  t} ( c_1 \, c_2 ) = 0,
\end{equation*}
and therefore, for all $t \in [0,T)$,
\begin{equation*}
    c_1(t) \, c_2 (t) = c_1 (0) \, c_2 (0) =: \Tilde{C} > 0.
\end{equation*}
We already know that $c_2$ is $\mathcal{C}^2$ and positive, therefore so is $c_1$.
Coming back to \eqref{1st_eq_sys_gauss}, we now have
\begin{equation*}
    \Ddot{c}_2 = 2 \, \frac{\Dot{c}_2^2}{c_2} - \frac{2 \lambda}{\Tilde{C}} c_2^3 - \frac{c_2^5}{\Tilde{C}},
\end{equation*}
which reads in terms of $c_1$
\begin{equation*}
    \Ddot{c}_1 = \frac{2 \lambda}{c_1} + \frac{\Tilde{C}^2}{c_1^3},
\end{equation*}
which is \eqref{gaussian1_th}.
Last, the final equation we have comes from the coefficient for $(x - b_1 (t))$:
\begin{equation*}
    2 \, \frac{\Dot{b}_1 (t)}{c_1 (t)^2} + 2 \, \frac{\Dot{c}_2 (t)}{c_1 (t)^2 \, c_2 (t)} \, b_1 (t) - 2 \frac{p_0 (t)}{c_1 (t)^2} = 0, \qquad \text{ for all } t \geq 0.
\end{equation*}
This leads to
\begin{equation*}
    \Dot{b}_1 = - \frac{\Dot{c}_2}{c_2} \, b_1 + p_0.
\end{equation*}
All the terms in the right-hand side are $\mathcal{C}^1([0,T))$, therefore so is $\Dot{b}_1$, which yields to the $\mathcal{C}^2$-regularity of $b_1$. Hence, we can again expand the expression for $p_0$ found previously with another integration by parts:
\begin{equation*}
    p_0 (t) = C_2 \, c_2 (t) + \frac{\Dot{c}_2 (t)}{c_2 (t)} \, b_1 (t) + \Dot{b}_1 (t) - c_2 (t) \int_0^t \frac{\Ddot{b}_1 (s)}{c_2 (s)} \, ds,
\end{equation*}
with $C_2 = C_1 + \frac{1}{c_2 (0)}$
Plugging this expression of $p_0$ into the expression of $\Dot{b}_1$ leads to
\begin{equation*}
    C_2 \, c_2 (t) = c_2 (t) \int_0^t \frac{\Ddot{b}_1 (s)}{c_2 (s)} \, ds.
\end{equation*}
Since $c_2 > 0$, we then obtain $C_2 = 0$ and $\frac{\Ddot{b}_1}{c_2} = 0$, which is $\Ddot{b}_1 = 0$. Thus, there exists $B_0, B_1$ constants such that 
\begin{equation*}
    b_1 = B_1 \, t + B_0,
\end{equation*}
and this gives the final expression for $p_0$ (and therefore for $b_2$):
\begin{equation*}
    p_0 (t) = ( B_1 \, t + B_0 ) \frac{\Dot{c}_2 (t)}{c_2 (t)} + B_1.
\end{equation*}
Putting all together leads to \eqref{gaussian1_th}-\eqref{gaussian2_th}, which yields the $\mathcal{C}^\infty$ feature of all the functions.

The last thing we need to check the convergence rate of $\tilde{\rho} (t)$ to $\gamma^2$ in $L^1$. For this, we can use again the Csisz\'ar-Kullback inequality, and compute with the expression of $\tilde{\rho} = \tau (t) \, \rho (t, \tau (t) y)$:
\begin{align*}
    \left\lVert \tilde{\rho}^2 (t) - \gamma^2 \right\rVert_{L^1}^2 &\leq 2 \lVert \gamma^2\rVert_{L^1} \int_{\mathbb{R}} \gamma^2 (y) \ln \frac{\gamma^2 (y)}{\tilde{\rho} (t,y)} \diff y \\
    &\leq 2 \sqrt{\pi} \int_\mathbb{R} \left[ \frac{(\tau (t) y - b_1 (t))^2}{c_1 (t)^2} - y^2 + \ln \frac{c_1 (t)}{\tau (t)} \right] e^{-y^2} \diff y \\
    &\leq 2 \pi \left[ \frac{1}{2} \left( 1 - \left( \frac{\tau (t)}{c_1 (t)} \right)^2 \right) + \ln \frac{c_1 (t)}{\tau (t)} + \frac{b_1 (t)^2}{c_1 (t)^2} \right].
\end{align*}
From \cite{carlesgallagher}, it is known that both $\tau (t)$ and $c_1 (t)$ have the same feature when $t \rightarrow \infty$:
\begin{equation*}
    \tau (t) = 2t \sqrt{\lambda \ln t} \left( 1 + \textnormal{O} \left( \frac{\ln \ln t}{\ln t} \right) \right) = c_1 (t).
\end{equation*}
Therefore, we get
\begin{gather*}
    1 - \left( \frac{\tau (t)}{c_1 (t)} \right)^2 = \textnormal{O} \left( \frac{\ln \ln t}{\ln t} \right), \qquad \ln \frac{c_1 (t)}{\tau (t)} = \textnormal{O} \left( \frac{\ln \ln t}{\ln t} \right).
\end{gather*}
Moreover, since $b_1 = B_1 \, t + B_0$, it is known that
\begin{gather*}
    \frac{b_1 (t)^2}{c_1 (t)^2} = \textnormal{O} \left( \frac{1}{\ln t} \right).
\end{gather*}
Putting everything together, we get \eqref{L1_conv_gauss_case}.
\hfill $\square$

\newpage

\appendix

\section{Proof of Proposition \ref{propmomht}}
\label{proofwhepsmom}

We now prove the points \ref{item_2_wigner} to \ref{item_4_wigner} of Proposition \ref{propmomht}. The part \ref{item_2_wigner} is proven in Section \ref{proof_item_2_wigner}, Section \ref{proof_item_3_wigner} is devoted to the proof of part \ref{item_3_wigner}, and finally we prove part \ref{item_4_wigner} in Section \ref{proof_item_4_wigner}.

\subsection{First part: proof of the second momentum in \texorpdfstring{$\xi$}{xi}}
\label{proof_item_2_wigner}

The proof of part \ref{item_2_wigner} of Proposition \ref{propmomht} is organized in 4 parts. First, we will prove the equality of $\int_{\mathbb{R}^d} |\xi|^2 \, W_\varepsilon^H (x, \xi) \diff \xi$ for $f_\varepsilon \in \mathcal{S} (\mathbb{R}^d)$ because we need better regularity for the interchange of integrals we will do. Then, we will generalize this result to the case $f_\varepsilon \in H^1$ by using an argument of continuity of a quadratic form and the fact that the integral is still well-defined even if $f_\varepsilon \in H^1$ because $|\xi|^2 \, W_\varepsilon^H (x, \xi) \geq 0$. Then we will be able to consider $\xi_i \xi_j \, W_\varepsilon^H (x, \xi)$ without any issue, and we will prove the equality involving it in the same way: first for $f_\varepsilon \in \mathcal{S} (\mathbb{R}^d)$, and then generalizing it for $f_\varepsilon \in H^1$ thanks to a continuity argument.

\subsubsection{Scalar second momentum: \texorpdfstring{$\mathcal{S} (\mathbb{R}^d)$}{Schwarz} case}

As $W_\varepsilon^H$ is non-negative, we can consider $\int_{\mathbb{R}^d} |\xi|^2 \, W_\varepsilon^H (x, \xi) \diff \xi$ without any issue. Moreover, we suppose here that $f_\varepsilon \in \mathcal{S} (\mathbb{R}^d)$.
Then:
\begin{align*}
    \int_{\mathbb{R}^d} |\xi|^2 \, W_\varepsilon^H (x, \xi) \diff \xi &= \int_{\mathbb{R}^d} |\xi|^2 \, (W_\varepsilon *_\xi \gamma_\varepsilon *_x \gamma_\varepsilon) (x, \xi) \diff \xi = \left( \int_{\mathbb{R}^d} |\xi|^2 \, W_\varepsilon *_\xi \gamma_\varepsilon \diff \xi \right) *_x \gamma_\varepsilon (x) .
\end{align*}
We check that the previous integral exchange is rigorous.
\begin{multline*}
    \left( \int_{\mathbb{R}^d} |\xi|^2 \, \left\lvert W_\varepsilon *_\xi \gamma_\varepsilon \right\rvert \diff \xi \right) * \gamma_\varepsilon (x) \\
    \begin{aligned} &= \left( \int_{\mathbb{R}^d} |\xi|^2 \, \left\lvert \mathcal{F}_{z \rightarrow \xi} \left( f_\varepsilon \left(.+ \frac{\varepsilon}{2} \, z \right) \, \overline{f_\varepsilon \left(.-\frac{\varepsilon}{2} \, z \right) }\right) *_\xi \mathcal{F}_{z \rightarrow \xi} \left( \exp{\left( - \varepsilon \frac{|z|^2}{4} \right)} \right) \right\rvert \diff \xi \right) * \gamma_\varepsilon (x) \\
    &= \left( \int_{\mathbb{R}^d} \left\lvert \mathcal{F}_{z \rightarrow \xi} \left( \Delta_z \left( f_\varepsilon \left(.+ \frac{\varepsilon}{2} \, z \right) \, \overline{f_\varepsilon \left(.-\frac{\varepsilon}{2} \, z \right) } \, \exp{\left( - \varepsilon \frac{|z|^2}{4} \right)} \right) \right) \right\rvert \diff \xi \right) * \gamma_\varepsilon (x) \\
    & \leq C_0 \left\lVert \Delta_z \left( f_\varepsilon \left(.+ \frac{\varepsilon}{2} \, z \right) \, \overline{f_\varepsilon \left(.-\frac{\varepsilon}{2} \, z \right) } \, \exp{\left( - \varepsilon \frac{|z|^2}{4} \right)} \right) \right\rVert_{W^{d+1,1}_z} \, *_x \gamma_\varepsilon (x) \\
    & \leq C_0 \left\lVert f_\varepsilon \left(.+ \frac{\varepsilon}{2} \, z \right) \, \overline{f_\varepsilon \left(.-\frac{\varepsilon}{2} \, z \right) } \, \exp{\left( - \varepsilon \frac{|z|^2}{4} \right)} \right\rVert_{W^{d+3,1}_z} \, *_x \gamma_\varepsilon (x) \\
    & \leq C_0 \left\lVert \exp{\left( - \varepsilon \frac{|z|^2}{4} \right)} \right\rVert_{W^{d+3,\infty}_z} \,  \left\lVert f_\varepsilon \left(.+ \frac{\varepsilon}{2} \, z \right) \, \overline{f_\varepsilon \left(.-\frac{\varepsilon}{2} \, z \right) } \right\rVert_{W^{d+3,1}_z} \, *_x \gamma_\varepsilon (x) \\
    & \leq C_0 \left\lVert \exp{\left( - \varepsilon \frac{|z|^2}{4} \right)} \right\rVert_{W^{d+1,\infty}_z} \,  \left\lVert f_\varepsilon \left(.+ \frac{\varepsilon}{2} \, z \right) \right\rVert_{H^{d+3}_z} \, \left\lVert \overline{f_\varepsilon \left(.-\frac{\varepsilon}{2} \, z \right) } \right\rVert_{H^{d+3}_z} \, *_x \gamma_\varepsilon (x) \\
    & \leq C_0 \, \varepsilon^{-d} \left\lVert \exp{\left( - \varepsilon \frac{|z|^2}{4} \right)} \right\rVert_{W^{d+1,\infty}_z} \, \left\lVert f_\varepsilon \right\rVert_{H^{d+3}}^2 \, \lVert \gamma_\varepsilon \rVert_{L^1} < \infty. \end{aligned}
\end{multline*}

\begin{rem}
This computation shows that we actually only need $f_\varepsilon \in H^{d+3}$.
\end{rem}

Now, come back to our first identity. We can compute in the way we want:
\begin{multline*}
    \int_{\mathbb{R}^d} |\xi|^2 \, W_\varepsilon^H (x, \xi) \diff \xi = \left( \int_{\mathbb{R}^d} |\xi|^2 \, W_\varepsilon *_\xi \gamma_\varepsilon \diff \xi \right) *_x \gamma_\varepsilon (x) \\
    \begin{aligned}
    &= - \left( \int_{\mathbb{R}^d} \mathcal{F}_{z \rightarrow \xi} \left( \Delta_z \left( f_\varepsilon \left(x+ \frac{\varepsilon}{2} \, z \right) \, \overline{f_\varepsilon \left(x-\frac{\varepsilon}{2} \, z \right) } \, \exp{\left( - \varepsilon \frac{|z|^2}{4} \right)} \right) \right) \diff \xi \right) *_x \gamma_\varepsilon (x) \\
    &= - \left[ \left. \Delta_z \left( f_\varepsilon \left(x+ \frac{\varepsilon}{2} \, z \right) \, \overline{f_\varepsilon \left(x-\frac{\varepsilon}{2} \, z \right) } \, \exp{\left( - \varepsilon \frac{|z|^2}{4} \right)} \right) \right] \right\rvert_{z=0} *_x \gamma_\varepsilon (x). \end{aligned}
\end{multline*}
Computing $\Delta_z \left( f_\varepsilon \left(x- \frac{\varepsilon}{2} \, z \right) \, \overline{f_\varepsilon \left(x+\frac{\varepsilon}{2} \, z \right) } \, \exp{\left( - \varepsilon \frac{|z|^2}{4} \right)} \right)$, we obtain
\begin{multline*}
    \left. \Delta_z \left( f_\varepsilon \left(x+ \frac{\varepsilon}{2} \, z \right) \, \overline{f_\varepsilon \left(x-\frac{\varepsilon}{2} \, z \right) } \, \exp{\left( - \varepsilon \frac{|z|^2}{4} \right)} \right) \right\rvert_{z=0} \\
\begin{aligned}
    &= \frac{\varepsilon^2}{4} \left[ \Delta f_\varepsilon \left(x \right) \, \overline{f_\varepsilon \left(x\right) } + f_\varepsilon \left(x\right) \, \overline{\Delta f_\varepsilon \left(x \right) } - 2 \, |\nabla f_\varepsilon (x)|^2 \right] - \frac{\varepsilon d}{2} |f_\varepsilon (x)|^2 \\
    &= \frac{\varepsilon^2}{4} \left[ \Delta \left( | f_\varepsilon |^2 \right) (x) - 4 \, |\nabla f_\varepsilon (x)|^2 \right] - \frac{\varepsilon d}{2} |f_\varepsilon (x)|^2.
\end{aligned}
\end{multline*}
Therefore, knowing that $\gamma_\varepsilon \in \mathcal{S} (\mathbb{R}^d)$, we can pass the $\Delta$ to the other side of the convolution and get \eqref{whepsmom1}.
Keeping in mind that $\gamma_\varepsilon \in \mathcal{S}$, integrating in $x$ yields \eqref{whepsmom}.

\subsubsection{Scalar second momentum: \texorpdfstring{$H^1$}{H1} case}

In the same way, we can still consider $\int_{\mathbb{R}^d} |\xi|^2 \, W_\varepsilon^H (x, \xi) \diff \xi$ even for $f_\varepsilon \in H^1$. However, it could still be equal to $+ \infty$. The first part will be to show that this is not the case.

For fixed $\varepsilon > 0$, take a sequence of functions $f_{\varepsilon, k}$ in $\mathcal{S} (\mathbb{R}^d)$ converging to $f_\varepsilon$ in $H^1$ when $k \rightarrow \infty$. Using the notation $W_{\varepsilon, k}$ (resp. $W_{\varepsilon,k}^H$) for the Wigner Transform (resp. the Husimi Transform) of the functions of the sequence, we first show that they converge uniformly to the Wigner Transform $W_\varepsilon$ (resp. the Husimi Transform $W_\varepsilon^H$) of $f_\varepsilon$.

For any $x, \xi \in \mathbb{R}^d$
\begin{align*}
    | W_{\varepsilon, k} (x, \xi) - W_\varepsilon (x, \xi) | 
    &\leq C_0 \left\lVert f_{\varepsilon, k} (x + \frac{\varepsilon z}{2}) \, \overline{f_{\varepsilon, k} (x - \frac{\varepsilon z}{2})} - f_\varepsilon (x + \frac{\varepsilon z}{2}) \, \overline{f_\varepsilon(x - \frac{\varepsilon z}{2})} \right\rVert_{L^1_z} \\
    &\begin{multlined}[11cm]
    \leq C_0 \left( \left\lVert f_{\varepsilon, k} (x + \frac{\varepsilon z}{2}) \, \left( \overline{f_{\varepsilon, k} (x - \frac{\varepsilon z}{2}) - f_\varepsilon(x - \frac{\varepsilon z}{2})} \right) \right\rVert_{L^1_z} \right. \\
    \left. + \left\lVert \left( f_{\varepsilon, k} (x + \frac{\varepsilon z}{2}) - f_\varepsilon (x + \frac{\varepsilon z}{2}) \right) \overline{f_\varepsilon(x - \frac{\varepsilon z}{2})} \right\rVert_{L^1_z} \right)
    \end{multlined} \\
    &\begin{multlined}[11cm]
    \leq C_0 \left( \left\lVert f_{\varepsilon, k} (x + \frac{\varepsilon z}{2}) \right\rVert_{L^2_z} \left\lVert f_{\varepsilon, k} (x - \frac{\varepsilon z}{2}) - f_\varepsilon(x - \frac{\varepsilon z}{2}) \right\rVert_{L^2_z} \right. \\
    \left. + \left\lVert f_{\varepsilon, k} (x + \frac{\varepsilon z}{2}) - f_\varepsilon (x + \frac{\varepsilon z}{2})  \right\rVert_{L^2_z} \left\lVert f_\varepsilon(x - \frac{\varepsilon z}{2}) \right\rVert_{L^2_z} \right)
    \end{multlined} \\
    &\leq C_\varepsilon \left( \left\lVert f_{\varepsilon, k} \right\rVert_{L^2} \left\lVert f_{\varepsilon, k}- f_\varepsilon \right\rVert_{L^2} + \left\lVert f_{\varepsilon, k} - f_\varepsilon  \right\rVert_{L^2} \left\lVert f_\varepsilon \right\rVert_{L^2} \right) \\
    &\leq C_\varepsilon \left\lVert f_{\varepsilon, k}- f_\varepsilon \right\rVert_{L^2}.
\end{align*}
Therefore, $W_{\varepsilon, k}$ converges uniformly to $W_\varepsilon$ with the estimate
\begin{equation*}
    \lVert W_{\varepsilon, k} - W_\varepsilon \rVert_{L^\infty} \leq C_\varepsilon \left\lVert f_{\varepsilon, k}- f_\varepsilon \right\rVert_{L^2},
\end{equation*}
and the same kind of estimate holds for the Husimi Transform:
\begin{equation*}
    \lVert W_{\varepsilon, k}^H - W_\varepsilon^H \rVert_{L^\infty} = \lVert \left( W_{\varepsilon, k} - W_\varepsilon \right) * G_\varepsilon \rVert_{L^\infty} \leq \lVert W_{\varepsilon, k} - W_\varepsilon \rVert_{L^\infty} \lVert G_\varepsilon\rVert_{L^1} \leq C_\varepsilon \left\lVert f_{\varepsilon, k}- f_\varepsilon \right\rVert_{L^2}.
\end{equation*}
Hence, Fatou's lemma for $|\xi|^2 W_{\varepsilon, k}^H (x, \xi)$ yields
\begin{equation*}
    \int_{\mathbb{R}^d} |\xi|^2 \, W_\varepsilon^H (x, \xi) \diff \xi \leq \underset{k \rightarrow \infty}{\liminf} \int_{\mathbb{R}^d} |\xi|^2 \, W_{\varepsilon,k}^H (x, \xi) \diff \xi.
\end{equation*}
The previous computation yields
\begin{equation*}
    \int_{\mathbb{R}^d} |\xi|^2 \, W_{\varepsilon,k}^H (x, \xi) \diff \xi = \varepsilon^2 \, |\nabla f_{\varepsilon, k}|^2 * \gamma_\varepsilon (x) - \frac{\varepsilon^2}{4} \, |f_{\varepsilon, k}|^2 * \Delta \gamma_\varepsilon (x) + \frac{\varepsilon d}{2} |f_{\varepsilon, k}|^2 * \gamma_\varepsilon (x).
\end{equation*}
But $f_{\varepsilon, k} \underset{k \rightarrow \infty}{\longrightarrow} f_\varepsilon$ in $H^1$, so $|\nabla f_{\varepsilon, k}|^2 \underset{k \rightarrow \infty}{\longrightarrow} |\nabla u_{\varepsilon}|^2$ and $|f_{\varepsilon, k}|^2 \underset{k \rightarrow \infty}{\longrightarrow} |f_\varepsilon|^2$ in $L^1$, therefore:
\begin{equation*}
    \int_{\mathbb{R}^d} |\xi|^2 \, W_\varepsilon^H (x, \xi) \diff \xi \leq \varepsilon^2 \, |\nabla f_\varepsilon|^2 * \gamma_\varepsilon (x) - \frac{\varepsilon^2}{4} \, |f_\varepsilon|^2 * \Delta \gamma_\varepsilon (x) + \frac{\varepsilon d}{2} |f_\varepsilon|^2 * \gamma_\varepsilon (x) < \infty.
\end{equation*}

Therefore, the map
\begin{align*}
    H^1 &\rightarrow \mathbb{R}^+ \\
    f_\varepsilon &\mapsto \int_{\mathbb{R}^d} |\xi|^2 \, W_\varepsilon^H (x, \xi) \diff \xi
\end{align*}
is well-defined for every $x \in \mathbb{R}^d$. Moreover, it is a non-negative quadratic form because $W_\varepsilon$ and then also $W_\varepsilon^H$ are quadratic. Furthermore, it is continuous thanks to the previous inequality which leads to
\begin{equation*}
    \int_{\mathbb{R}^d} |\xi|^2 \, W_\varepsilon^H (x, \xi) \diff \xi \leq C_\varepsilon \lVert f_\varepsilon \rVert_{H^1}^2.
\end{equation*}
Thus, the equality \eqref{whepsmom1}, which is true in $\mathcal{S} (\mathbb{R}^d)$ dense subspace in $H^1$, also holds in $H^1$.

\subsubsection{Vector second momentum: \texorpdfstring{$\mathcal{S}$}{Schwarz} case}

With the same assumptions, we can consider $\int_{\mathbb{R}^d} \xi_i \xi_j \, W_\varepsilon^H (x, \xi) \diff \xi$ as we now know that $\xi_i \xi_j \, W_\varepsilon^H (x, \xi)$ is integrable thanks to the previous identity, and in the same way, we have for $f_\varepsilon \in \mathcal{S} (\mathbb{R}^d)$ and for every $x\in \mathbb{R}^d$:
\begin{align*}
    \int_{\mathbb{R}^d} \xi_i \xi_j \, W_\varepsilon^H (x, \xi) \diff \xi &= \int_{\mathbb{R}^d} \xi_i \xi_j \, (W_\varepsilon *_\xi \gamma_\varepsilon *_x \gamma_\varepsilon) (x, \xi) \diff \xi \\
    &= \left( \int_{\mathbb{R}^d} \xi_i \xi_j \, W_\varepsilon *_\xi \gamma_\varepsilon \diff \xi \right) *_x \gamma_\varepsilon (x),
\end{align*}
the interchange of integral is rigorous with the same kind of estimate as previously. Moreover, we readily compute
\begin{multline*}
    \left[ \left. \partial_{z_i} \partial_{z_j} \left( f_\varepsilon \left(x+ \frac{\varepsilon}{2} \, z \right) \, \overline{f_\varepsilon \left(x-\frac{\varepsilon}{2} \, z \right) } \, \exp{\left( - \varepsilon \frac{|z|^2}{4} \right)} \right) \right] \right\rvert_{z=0} \\
    \begin{aligned}
        &= \frac{\varepsilon^2}{4} \left[ \partial_i \partial_j f_\varepsilon \left(x\right) \, \overline{f_\varepsilon \left(x\right) } + f_\varepsilon \left(x\right) \, \overline{\partial_i \partial_j f_\varepsilon \left(x\right) } - \partial_i f_\varepsilon \left(x \right) \, \overline{\partial_j f_\varepsilon \left(x \right) } - \partial_j f_\varepsilon \left(x \right) \, \overline{\partial_i f_\varepsilon \left(x \right) } \right] - \frac{\varepsilon \delta_{ij}}{2} f_\varepsilon \left(x\right) \, \overline{f_\varepsilon \left(x\right) } \\
        &= \frac{\varepsilon^2}{4} \left[ \partial_i \partial_j \left( |f_\varepsilon|^2 \right) \left(x\right) - 4 \, \Re \left( \partial_i f_\varepsilon \left(x \right) \, \overline{\partial_j f_\varepsilon \left(x \right) } \right) \right] - \frac{\varepsilon \delta_{ij}}{2} |f_\varepsilon(x)|^2.
    \end{aligned}
\end{multline*}
Therefore, in the same way as in the previous first section, we get \eqref{whepsmom01} and \eqref{whepsmom02}.

\subsubsection{Vector second momentum: \texorpdfstring{$H^1$}{H1} case}

The generalization of this equality is similar to the end of the previous generalization for the scalar second momentum. The map
\begin{align*}
    H^1 &\rightarrow \mathbb{R} \\
    f_\varepsilon &\mapsto \int_{\mathbb{R}^d} \xi_i \xi_j \, W_\varepsilon^H (x, \xi) \diff \xi
\end{align*}
is a well-defined, continuous quadratic form thanks to the previous equality for the scalar second momentum. Then, the identities found for $f_\varepsilon \in \mathcal{S} (\mathbb{R}^d)$ also hold for $f_\varepsilon \in H^1$.

\subsection{Second part: first momentum in \texorpdfstring{$\xi$}{xi}}
\label{proof_item_3_wigner}

We know that $\int_{\mathbb{R}^d} |\xi|^2 \, W_\varepsilon^H (x, \xi) \diff \xi < \infty$ by the previous proof and also that $\int_{\mathbb{R}^d} W_\varepsilon^H (x, \xi) \diff \xi < \infty$, therefore we can consider $\int_{\mathbb{R}^d} \xi \, W_\varepsilon^H (x, \xi) \diff \xi$. Then:
\begin{align*}
    \int_{\mathbb{R}^d} \xi \, W_\varepsilon^H (x, \xi) \diff \xi &= \int_{\mathbb{R}^d} \xi \, (W_\varepsilon *_\xi \gamma_\varepsilon *_x \gamma_\varepsilon) (x, \xi) \diff \xi \\
    &= \left( \int_{\mathbb{R}^d} \xi \, W_\varepsilon *_\xi \gamma_\varepsilon \diff \xi \right) *_x \gamma_\varepsilon,
\end{align*}
the integral exchange being rigorous for $f_\varepsilon \in \mathcal{S} (\mathbb{R}^d)$ with the same kind of computation as before, which infers that:
\begin{align*}
    \int_{\mathbb{R}^d} \xi \, W_\varepsilon^H (x, \xi) \diff \xi &= \left( - i \left. \nabla_z \left( f_\varepsilon \left(.+ \frac{\varepsilon}{2} \, z \right) \, \overline{f_\varepsilon \left(.-\frac{\varepsilon}{2} \, z \right) } \, \exp{\left( - \varepsilon \frac{|z|^2}{4} \right)} \right) \right\rvert_{z=0} \right) *_x \gamma_\varepsilon \\
    &= \varepsilon \Im \left( \nabla f_\varepsilon \, \overline{f_\varepsilon} \right) * \gamma_\varepsilon (x),
\end{align*}
and therefore \eqref{wheps1stmom} for $f_\varepsilon \in \mathcal{S} (\mathbb{R}^d)$, \eqref{wheps1stmom2} being obvious by integrating this result. The conclusion for the general case runs as before.

\subsection{Third part: second momentum in \texorpdfstring{$x$}{x}}
\label{proof_item_4_wigner}

In the same way, since $W_\varepsilon^H$ is non-negative, we have, thanks to Proposition \ref{prop_int_HT},
\begin{align*}
    \iint_{\mathbb{R}^d \times \mathbb{R}^d} |x|^2 \, W^H_\varepsilon (x, \xi) \diff x \diff \xi &= \int_{\mathbb{R}^d} |x|^2 \left( \int_{\mathbb{R}^d} W^H_\varepsilon (x, \xi) \diff \xi \right) \diff x \\
    &= \int_{\mathbb{R}^d} |x|^2 \, |f_\varepsilon|^2 * \gamma_\varepsilon (x) \diff x \\
    &= \iint_{\mathbb{R}^d \times \mathbb{R}^d} |x|^2 \, |f_\varepsilon (x-y)|^2 * \gamma_\varepsilon (y) \diff y \diff x.
\end{align*}
Therefore, we can easily compute
\begin{align*}
    \iint_{\mathbb{R}^d \times \mathbb{R}^d} |x|^2 \, W^H_\varepsilon (x, \xi) \diff x \diff \xi &= \iint_{\mathbb{R}^d \times \mathbb{R}^d} (|x-y|^2 + 2 \, (x-y) \cdot y + |y|^2) \, |f_\varepsilon (x-y)|^2 * \gamma_\varepsilon (y) \diff y \diff x \\
    & \begin{multlined} = \left( \int_{\mathbb{R}^d} |x|^2 \, |f_\varepsilon (x)|^2 \diff x \right) \, \lVert \gamma_\varepsilon \rVert_{L^1} + 2 \left( \int_{\mathbb{R}^d} x \, |f_\varepsilon (x)|^2 \diff x \right) \cdot \left( \int_{\mathbb{R}^d} y \, \gamma_\varepsilon(y) \diff y \right) \\ + \lVert f_\varepsilon\rVert_{L^2}^2 \int_{\mathbb{R}^d} |y|^2 \, \gamma_\varepsilon (y) \diff y
    \end{multlined} \\
    & =  \left\lVert |x|^2 |f_\varepsilon (x)|^2 \right\rVert_{L^1} + \frac{\varepsilon d}{2} \, \lVert f_\varepsilon\rVert_{L^2}^2.
\end{align*}
thanks to the properties on the momenta of $\gamma_\varepsilon$.

\bibliographystyle{abbrv}
\bibliography{sample}

\begin{thebibliography}{10}

\bibitem{Alazard_Carles}
T.~Alazard and R.~Carles.
\newblock Loss of regularity for supercritical nonlinear {S}chr{\"o}dinger
  equations.
\newblock {\em Mathematische Annalen}, 343(2):397--420, Feb 2009.

\bibitem{soblog}
C.~An{\'e}, S.~Blach{\`e}re, D.~Chafai, P.~Foug{\`e}res, I.~Gentil, F.~Malrieu,
  C.~Roberto, and G.~Scheffer.
\newblock {\em {Sur les in{\'e}galit{\'e}s de Sobolev logarithmiques}}.
\newblock {Soci{\'e}t{\'e} Math{\'e}matique de France}, 2000.

\bibitem{fokkerplank}
A.~Arnold, P.~Markowich, G.~Toscani, and A.~Unterreiter.
\newblock On convex sobolev inequalities and the rate of convergence to
  equilibrium for fokker-planck type equations.
\newblock {\em Communications in Partial Differential Equations},
  26(1-2):43--100, 2001.

\bibitem{athanassoulis-paul}
A.~Athanassoulis and T.~Paul.
\newblock {Strong phase-space semiclassical asymptotics}.
\newblock {\em {SIAM Journal on Mathematical Analysis}}, 43(5):2116--2149,
  2011.

\bibitem{nonlin_wave_mec}
I.~Bia{\l}ynicki-Birula and J.~Mycielski.
\newblock Nonlinear wave mechanics.
\newblock {\em Ann. Physics}, 100(1-2):62--93, 1976.

\bibitem{bolley-gentil-guillin_wassFK}
F.~Bolley, I.~Gentil, and A.~Guillin.
\newblock Convergence to equilibrium in {W}asserstein distance for
  {F}okker-{P}lanck equations.
\newblock {\em J. Funct. Anal.}, 263(8):2430--2457, 2012.

\bibitem{inco_white_light_log}
H.~Buljan, A.~\v{S}iber, M.~Solja\v{c}i\'{c}, T.~Schwartz, M.~Segev, and D.~N.
  Christodoulides.
\newblock Incoherent white light solitons in logarithmically saturable
  noninstantaneous nonlinear media.
\newblock {\em Phys. Rev. E (3)}, 68(3):036607, 6, 2003.

\bibitem{Carles_book}
R.~Carles.
\newblock {\em {Semi-classical analysis for nonlinear Schrödinger equations}}.
\newblock World Scientific, Singapore, 2008.

\bibitem{Carles_scattering_2016}
R.~Carles.
\newblock On semi-classical limit of nonlinear quantum scattering.
\newblock {\em Ann. Sci. \'{E}c. Norm. Sup\'{e}r. (4)}, 49(3):711--756, 2016.

\bibitem{Wigner_fails}
R.~Carles, C.~Fermanian-Kammerer, N.~J. Mauser, and H.~P. Stimming.
\newblock On the time evolution of {W}igner measures for {S}chr\"{o}dinger
  equations.
\newblock {\em Commun. Pure Appl. Anal.}, 8(2):559--585, 2009.

\bibitem{carlesgallagher}
R.~Carles and I.~Gallagher.
\newblock {Universal dynamics for the defocusing logarithmic Schrodinger
  equation}.
\newblock {\em {Duke Mathematical Journal}}, 167(9):1761--1801, 2018.

\bibitem{carles-miller}
R.~Carles and L.~Miller.
\newblock Semiclassical nonlinear schrödinger equations with potential and
  focusing initial data.
\newblock {\em Osaka J. Math.}, 41(3):693--725, 09 2004.

\bibitem{carlesnouri}
R.~Carles and A.~Nouri.
\newblock {Monokinetic solutions to a singular Vlasov equation from a
  semiclassical perspective}.
\newblock {\em {Asymptotic Analysis}}, 102:99--117, 2017.

\bibitem{cazenave-haraux}
T.~Cazenavec and A.~Haraux.
\newblock Equations d'\'evolution avec non lin\'earit\'e logarithmique.
\newblock {\em Annales de la facult\'e de sciences de Toulouse, 5e s\'erie},
  II:21--55, 1980.

\bibitem{Combescure_Robert__Coherent_states}
M.~Combescure and D.~Robert.
\newblock {\em Coherent states and applications in mathematical physics}.
\newblock Theoretical and Mathematical Physics. Springer, Dordrecht, 2012.

\bibitem{D'av_Mont_Squa_lognls}
P.~d'Avenia, E.~Montefusco, and M.~Squassina.
\newblock On the logarithmic {S}chr\"{o}dinger equation.
\newblock {\em Commun. Contemp. Math.}, 16(2):1350032, 15, 2014.

\bibitem{log_nls_magma_transp}
S.~{De Martino}, M.~Falanga, C.~Godano, and G.~Lauro.
\newblock Logarithmic schrödinger-like equation as a model for magma
  transport.
\newblock {\em EPL (Europhysics Letters)}, 63(3):472, 2003.

\bibitem{Griffin-Pickering_Iacobelli_VPME_to_KIE}
M.~Griffin-Pickering and M.~Iacobelli.
\newblock Singular limits for plasmas with thermalised electrons.
\newblock https://arxiv.org/abs/1811.03693, 11 2018.

\bibitem{Guerrero_Lopez_Nieto_H1_solv_lognls}
P.~Guerrero, J.~L. L\'{o}pez, and J.~Nieto.
\newblock Global {$H^1$} solvability of the 3{D} logarithmic {S}chr\"{o}dinger
  equation.
\newblock {\em Nonlinear Anal. Real World Appl.}, 11(1):79--87, 2010.

\bibitem{Gerard9091}
P.~Gérard.
\newblock Mesures semi-classiques et ondes de bloch.
\newblock {\em Séminaire Équations aux dérivées partielles
  (Polytechnique)}, pages 1--19, 1990-1991.

\bibitem{gerard-mark-mauser}
P.~Gérard, P.~A. Markowich, N.~J. Mauser, and F.~Poupaud.
\newblock Homogenization limits and wigner transforms.
\newblock {\em Communications on Pure and Applied Mathematics}, 50(4):323--379,
  Apr. 1997.

\bibitem{Hagedorn_semiclassical_III}
G.~A. Hagedorn.
\newblock Semiclassical quantum mechanics. {III}. {T}he large order asymptotics
  and more general states.
\newblock {\em Ann. Physics}, 135(1):58--70, 1981.

\bibitem{Hagedorn_Joye}
G.~A. Hagedorn and A.~Joye.
\newblock Exponentially accurate semiclassical dynamics: propagation,
  localization, {E}hrenfest times, scattering, and more general states.
\newblock {\em Ann. Henri Poincar\'{e}}, 1(5):837--883, 2000.

\bibitem{Han-Kwan_Iacobelli_quasineutral_lim_VP}
D.~Han-Kwan and M.~Iacobelli.
\newblock The quasineutral limit of the {V}lasov-{P}oisson equation in
  {W}asserstein metric.
\newblock {\em Commun. Math. Sci.}, 15(2):481--509, 2017.

\bibitem{log_nls_nuclear_physics}
E.~F. Hefter.
\newblock Application of the nonlinear schr\"odinger equation with a
  logarithmic inhomogeneous term to nuclear physics.
\newblock {\em Phys. Rev. A}, 32:1201--1204, Aug 1985.

\bibitem{quantal_damped_motion}
E.~Hernández and B.~Remaud.
\newblock General properties of gausson-conserving descriptions of quantal
  damped motion.
\newblock {\em Physica A: Statistical Mechanics and its Applications},
  105(1):130 -- 146, 1981.

\bibitem{solitons_log_med}
W.~Kr\'olikowski, D.~Edmundson, and O.~Bang.
\newblock Unified model for partially coherent solitons in logarithmically
  nonlinear media.
\newblock {\em Phys. Rev. E}, 61:3122--3126, Mar 2000.

\bibitem{PaulLions}
P.-L. Lions and T.~Paul.
\newblock Sur les mesures de wigner.
\newblock {\em Revista Matemática Iberoamericana}, 9(3):553--618, 1993.

\bibitem{villani2008optimal}
C.~Villani.
\newblock {\em Optimal Transport: Old and New}.
\newblock Grundlehren der mathematischen Wissenschaften. Springer Berlin
  Heidelberg, 2008.

\bibitem{wigner}
E.~Wigner.
\newblock On the quantum correction for thermodynamic equilibrium.
\newblock {\em Phys. Rev.}, 40:749--759, Jun 1932.

\bibitem{yajima1979}
K.~Yajima.
\newblock The quasiclassical limit of quantum scattering theory.
\newblock {\em Comm. Math. Phys.}, 69(2):101--129, 1979.

\bibitem{Yajima1981}
K.~Yajima.
\newblock The quasiclassical limit of quantum scattering theory. {II}.
  {L}ong-range scattering.
\newblock {\em Duke Math. J.}, 48(1):1--22, 1981.

\bibitem{Zhang-Zheng-Mauser}
P.~Zhang, Y.~Zheng, and N.~J. Mauser.
\newblock The limit from the {S}chr\"{o}dinger-{P}oisson to the
  {V}lasov-{P}oisson equations with general data in one dimension.
\newblock {\em Comm. Pure Appl. Math.}, 55(5):582--632, 2002.

\end{thebibliography}

\end{document}